\documentclass[12pt]{amsart}

\usepackage{amsmath}
\usepackage{amssymb}
\usepackage{amsthm}
\usepackage{pstricks}
\usepackage{pst-node}
\usepackage{pst-plot}
\usepackage{subfigure}

\numberwithin{equation}{section}

\newtheorem{thm}{Theorem}[section]
\newtheorem{prop}[thm]{Proposition}
\newtheorem{conj}[thm]{Conjecture}
\newtheorem{lem}[thm]{Lemma}

\theoremstyle{definition}
\newtheorem{defin}[thm]{Definition}
\newtheorem{rmk}[thm]{Remark}
\newtheorem{ex}[thm]{Example}

\def\x{{\bf x}}
\def\y{{\bf y}}
\def\H{{\mathcal H}}

\newcommand{\meet}[2]{#1 \cap #2}
\newcommand{\join}[2]{\langle #1, #2 \rangle}

\begin{document}
\title{$Y$-meshes and generalized pentagram maps}
\author{Max Glick \and Pavlo Pylyavskyy}
\address{Department of Mathematics, University of Minnesota,
Minneapolis, MN 55455, USA}
\thanks{M. G. was partially supported by NSF grant DMS-1303482.  P. P. was partially supported by NSF grants DMS-1148634, DMS-1351590, and Sloan Fellowship.}

\begin{abstract}
We introduce a rich family of generalizations of the pentagram map sharing the property that each generates an infinite configuration of points and lines with four points on each line.  These systems all have a description as $Y$-mutations in a cluster algebra and hence establish new connections between cluster theory and projective geometry.  Our framework incorporates many preexisting generalized pentagram maps due to M. Gekhtman, M. Shapiro, S. Tabachnikov, and A. Vainshtein and also B. Khesin and F. Soloviev.  In several of these cases a reduction to cluster dynamics was not previously known.%The reduction to cluster dynamics was not known for some of these systems an
\end{abstract}

\maketitle

\setcounter{tocdepth}{1}
\tableofcontents

\section{Introduction} \label{secintro}
The pentagram map, introduced by R. Schwartz \cite{S1} is a discrete dynamical system defined on the space of polygons in the projective plane.  Figure \ref{figT} gives an example of the pentagram map, which we denote $T$.  Although the underlying construction is simple, the resulting system has many amazing properties including the following:
\begin{itemize}
\item V. Ovsienko, Schwartz, and S. Tabachnikov \cite{OST1} established that the pentagram map is a completely integrable system.  Further work in this direction was conducted by those same authors \cite{OST2}, F. Soloviev \cite{So}, and M. Gekhtman, M. Shapiro, Tabachnikov, and A. Vainshtein \cite{GSTV}.
\item Schwartz \cite{S2} showed that the continuous limit of the pentagram map is the Boussineq equation.
\item The first author \cite{G} proved that the pentagram map can be described in certain coordinates as mutations in a cluster algebra.
\end{itemize}

\begin{figure} 
\subfigure[An application of the pentagram map with input $A$ and output $B=T(A)$.\label{figT}]{
\begin{pspicture}(6,5)
\rput(0,-.5){
  \pspolygon[linewidth=1.5pt, showpoints=true](1,2)(1,3)(2,4.5)(3,5)(4.5,4.5)(5,3.5)(4.5,2)(3.5,1)(2,1)
  \pspolygon[linestyle=dashed](1,2)(2,4.5)(4.5,4.5)(4.5,2)(2,1)(1,3)(3,5)(5,3.5)(3.5,1)
  \pspolygon[linewidth=1.5pt, showpoints=true](1.22,2.55)(1.67,3.67)(2.5,4.5)(3.67,4.5)(4.5,3.87)(4.5,2.67)(3.97,1.79)(2.75,1.3)(1.62,1.75)
  %\uput[ur](3.75,4.75){$A$}
  %\uput[d](3,4.5){$B$}
	\uput[dl](1,2){$A_1$}
	\uput[ul](1,3){$A_2$}
	\uput[ul](2,4.5){$A_3$}
	\uput[u](3,5){$A_4$}
	\uput[u](4.5,4.5){$A_5$}
	\uput[r](5,3.5){$A_6$}
	\uput[r](4.5,2){$A_7$}
	\uput[dr](3.5,1){$A_8$}
	\uput[d](2,1){$A_9$}
	\uput[r](1.22,2.55){$B_1$}
	\uput[dr](1.67,3.67){$B_2$}
	\uput[d](2.5,4.5){$B_3$}
	\uput[d](3.67,4.5){$B_4$}
	\uput[dl](4.5,3.87){$B_5$}
	\uput[l](4.5,2.67){$B_6$}
	\uput[ul](3.97,1.79){$B_7$}
	\uput[u](2.75,1.3){$B_8$}
	\uput[ur](1.62,1.75){$B_9$}
}
\end{pspicture}
}
\hfill
\subfigure[Part of a $Y$-mesh produced by the pentagram map.\label{figmesh}]{
%\psset{unit=.8cm}
\begin{pspicture}(5,5)
\rput(-.5,-.5){
  % Original Polygon
	\pspolygon[showpoints=true](1,2)(1,3)(2,4.5)(3,5)(4.5,4.5)(5,3.5)(4.5,2)(3.5,1)(2,1)
 
 % \pspolygon(1,2)(1,3)(2,4.5)(3,5)(4.5,4.5)(5,3.5)(4.5,2)(3.5,1)(2,1)
  \pspolygon(1,2)(2,4.5)(4.5,4.5)(4.5,2)(2,1)(1,3)(3,5)(5,3.5)(3.5,1)
  \pspolygon[showpoints=true](1.22,2.56)(1.67,3.67)(2.5,4.5)(3.67,4.5)(4.5,3.87)(4.5,2.67)(3.97,1.79)(2.75,1.3)(1.62,1.75)
	
%	\pspolygon(1.22,2.56)(1.67,3.67)(2.5,4.5)(3.67,4.5)(4.5,3.87)(4.5,2.67)(3.97,1.79)(2.75,1.3)(1.62,1.75)
  \pspolygon(1.22,2.56)(2.5,4.5)(4.5,3.87)(3.97,1.79)(1.62,1.75)(1.67,3.67)(3.67,4.5)(4.5,2.67)(2.75,1.3)
  \pspolygon[showpoints=true](1.64,2.22)(1.66,3.22)(2.06,3.83)(3.17,4.29)(3.86,4.08)(4.30,3.10)(4.12,2.37)(3.36,1.78)(2.19,1.76)
 
%  \pspolygon(1.64,2.22)(1.66,3.22)(2.06,3.83)(3.17,4.29)(3.86,4.08)(4.30,3.10)(4.12,2.37)(3.36,1.78)(2.19,1.76)
  \pspolygon(1.64,2.22)(2.06,3.83)(3.86,4.08)(4.12,2.37)(2.19,1.76)(1.66,3.22)(3.17,4.29)(4.30,3.1)(3.36,1.78)
  \pspolygon[showpoints=true](1.8,2.83)(1.95,3.43)(2.63,3.91)(3.42,4.02)(3.95,3.47)(4.06,2.76)(3.69,2.23)(2.75,1.93)(2.06,2.11)
}
\end{pspicture}
}
\caption{}
\label{figTmesh}
\end{figure}

The pentagram map has been generalized and modified in many ways with the hope that some or all of these properties continue to hold.  Work of B. Khesin and Soloviev \cite{KS1,KS2,KS3}, Gekhtman, Shapiro, Tabachnikov and Vainshtein \cite{GSTV} and G. Mari Beffa \cite{M1,M2,M3} all pursue this idea.  Of the three listed properties, it seems that a connection to cluster algebras is the most resistant to being pushed forward, with only \cite{GSTV} achieving a generalization of the cluster description of $T$.  

We propose a new family of generalized pentagram maps, all of which have a description in terms of cluster algebras.  Our family includes the higher pentagram maps of Gekhtman et al. \cite{GSTV}, some of the maps introduced by Khesin and Soloviev \cite{KS1,KS2}, as well as a wide variety of new systems.  Already for the systems of Khesin and Soloviev, the cluster structure is new.

\begin{rmk}
The cluster algebra associated to each of our systems has a quiver that can be embedded on a torus with alternating orientations around each vertex (see Section \ref{seclift}).  Therefore, the techniques of A. Goncharov and R. Kenyon \cite{GK} or Gekhtman et al. \cite{GSTV} should apply, with consequences for Arnold-Liouville integrability of the systems.  We choose not to pursue this direction in this paper.
\end{rmk}

\smallskip
\centerline{------------}
\smallskip

The pentagram map relates to cluster algebras by way of certain projectively natural coordinates on the space of polygons.  Specifically, if $A$ is a polygon and $P_{i,j}$ is the $i$th vertex of $T^j(A)$ (using the indexing method suggested by Figure \ref{figT}) then
\begin{displaymath}
P_{i,j}, P_{i+2,j}, P_{i,j+1}, P_{i+1,j+1}
\end{displaymath}
are collinear for all $i,j \in \mathbb{Z}$ (see Figure \ref{figmesh}).  The cross ratios of such quadruples play the role of $y$-variables in the associated cluster algebra.  To obtain more systems admitting a cluster description, we reverse engineer the property of producing collinear points in a specified pattern.

\begin{defin} \label{defpin}
Let $a,b,c,d \in \mathbb{Z}^2$ be distinct and assume $a_2 \leq b_2 \leq c_2 \leq d_2$.  Say that $S = \{a,b,c,d\}$ is a $Y$-pin if $b_2 < c_2$ and the vectors $b-a,c-a,d-a$ generate all of $\mathbb{Z}^2$.
When $S$ is a $Y$-pin, we will always assume its elements are called $a,b,c,d$ and that $a_2 \leq b_2 < c_2 \leq d_2$.
\end{defin}

\begin{rmk}
If $a_2 < b_2 < c_2 < d_2$ then there is no ambiguity in determining $a,b,c,d$ given $S$.  If $a_2=b_2$ or $c_2=d_2$ (or both) we allow any choice of $a,b,c,d$ satisfying $a_2 \leq b_2 < c_2 \leq d_2$.  The corresponding dynamical system we define will not depend on the choice.
\end{rmk}

\begin{defin} \label{defmesh}
Let $S = \{a,b,c,d\}$ be a $Y$-pin and suppose $D \geq 2$.  A \emph{$Y$-mesh} of type $S$ and dimension $D$ is a grid of points $P_{i,j}$ and lines $L_{i,j}$ in $\mathbb{RP}^D$ with $i,j \in \mathbb{Z}$ which together span all of $\mathbb{RP}^D$ and such that
\begin{itemize}
\item $P_{r+a}, P_{r+b}, P_{r+c}, P_{r+d}$ all lie on $L_r$ and are distinct for all $r \in \mathbb{Z}^2$
\item $L_{r-a}, L_{r-b}, L_{r-c}, L_{r-d}$ (all of which contain $P_r$) are distinct for all $r \in \mathbb{Z}^2$.
\end{itemize}
\end{defin}
The points determine the lines (for example, via $L_r = \join{P_{r+a}}{P_{r+b}}$) so we usually consider only the points.

\begin{rmk}
We note that the definition of a $Y$-mesh is similar in spirit to that of a $Q$-net given by A. Bobenko and Y. Suris \cite{BS}.  In the former certain quadruples of points are assumed to be collinear while in the latter certain quadruples are assumed to be coplanar.  We are not aware of any direct connections between the two theories.
\end{rmk}

Given four points $x_1,x_2,x_3,x_4 \in \mathbb{RP}^1 = \mathbb{R} \cup \{\infty\}$, their \emph{cross ratio} is defined to be
\begin{displaymath}
[x_1,x_2,x_3,x_4] = \frac{(x_1-x_2)(x_3-x_4)}{(x_2-x_3)(x_4-x_1)}.
\end{displaymath}
The cross ratio is invariant under projective transformations.  Given a $Y$-mesh $P = (P_{i,j})_{i,j \in \mathbb{Z}}$ of type $S=\{a,b,c,d\}$, let
\begin{equation} \label{eqdefy}
y_r(P) = -[P_{r+a}, P_{r+c}, P_{r+b}, P_{r+d}]
\end{equation}
for all $r \in \mathbb{Z}^2$.

\begin{thm} \label{thmcluster}
Fix a $Y$-mesh of type $S = \{a,b,c,d\}$ and let $y_r = y_r(P)$.  Then
\begin{equation} \label{eqmain}
y_{r+a+b}y_{r+c+d} = \frac{(1+y_{r+a+c})(1+y_{r+b+d})}{(1+y_{r+a+d}^{-1})(1+y_{r+b+c}^{-1})}
\end{equation}
for all $r \in \mathbb{Z}^2$.  Moreover, there exists a quiver $Q_S$ and a sequence of mutations so that the $y$-variables transform according to \eqref{eqmain}.
\end{thm}

The second part of Theorem \ref{thmcluster}, namely the construction of the desired quiver, can be seen as a special case of an extension of work of A. Fordy and R. Marsh \cite{FM} that we develop.  To give a sense of the types of quivers that arise, Figure \ref{figquiver} contains a small example of $Q_S$, specifically for the $Y$-pin $S= \{(-1,0),(1,0),(0,1),(0,2)\}$.  We will see that this $S$ corresponds to the short diagonal hyperplane map in three dimensions first studied by Khesin and Soloviev \cite{KS1} and Mari Beffa \cite{M1}.

\begin{figure}
\begin{pspicture}(-2,.8)(11,3.2)
\cnode(1,1){.1}{v11}
\cnode(2,1){.1}{v21}
\cnode(3,1){.1}{v31}
\cnode(4,1){.1}{v41}
\cnode(5,1){.1}{v51}
\cnode(6,1){.1}{v61}
\cnode(7,1){.1}{v71}
\cnode(8,1){.1}{v81}
\cnode(1,2){.1}{v12}
\cnode(2,2){.1}{v22}
\cnode(3,2){.1}{v32}
\cnode(4,2){.1}{v42}
\cnode(5,2){.1}{v52}
\cnode(6,2){.1}{v62}
\cnode(7,2){.1}{v72}
\cnode(8,2){.1}{v82}
\cnode(1,3){.1}{v13}
\cnode(2,3){.1}{v23}
\cnode(3,3){.1}{v33}
\cnode(4,3){.1}{v43}
\cnode(5,3){.1}{v53}
\cnode(6,3){.1}{v63}
\cnode(7,3){.1}{v73}
\cnode(8,3){.1}{v83}

\psset{arrowsize=5pt}
\psset{arrowinset=0}

\ncline{->}{v11}{v22}
\ncline{->}{v21}{v32}
\ncline{->}{v31}{v42}
\ncline{->}{v41}{v52}
\ncline{->}{v51}{v62}
\ncline{->}{v61}{v72}
\ncline{->}{v71}{v82}

\ncline{->}{v12}{v23}
\ncline{->}{v22}{v33}
\ncline{->}{v32}{v43}
\ncline{->}{v42}{v53}
\ncline{->}{v52}{v63}
\ncline{->}{v62}{v73}
\ncline{->}{v72}{v83}

\ncline{->}{v12}{v21}
\ncline{->}{v22}{v31}
\ncline{->}{v32}{v41}
\ncline{->}{v42}{v51}
\ncline{->}{v52}{v61}
\ncline{->}{v62}{v71}
\ncline{->}{v72}{v81}

\ncline{->}{v13}{v22}
\ncline{->}{v23}{v32}
\ncline{->}{v33}{v42}
\ncline{->}{v43}{v52}
\ncline{->}{v53}{v62}
\ncline{->}{v63}{v72}
\ncline{->}{v73}{v82}

\ncline{->}{v21}{v13}
\ncline{->}{v31}{v23}
\ncline{->}{v41}{v33}
\ncline{->}{v51}{v43}
\ncline{->}{v61}{v53}
\ncline{->}{v71}{v63}
\ncline{->}{v81}{v73}

\ncline{->}{v23}{v11}
\ncline{->}{v33}{v21}
\ncline{->}{v43}{v31}
\ncline{->}{v53}{v41}
\ncline{->}{v63}{v51}
\ncline{->}{v73}{v61}
\ncline{->}{v83}{v71}

\ncarc[arcangle= 15]{->}{v32}{v12}
\ncarc[arcangle=-15]{->}{v42}{v22}
\ncarc[arcangle= 15]{->}{v52}{v32}
\ncarc[arcangle=-15]{->}{v62}{v42}
\ncarc[arcangle= 15]{->}{v72}{v52}
\ncarc[arcangle=-15]{->}{v82}{v62}

\psdots(-1,2)(-.5,2)(0,2)
\psdots(9,2)(9.5,2)(10,2)
%\rput(9,2){$\ldots$}

\end{pspicture}
\caption{The quiver $Q_S$ for $S=\{(-1,0),(1,0),(0,1),(0,2)\}$.}
\label{figquiver}
\end{figure}
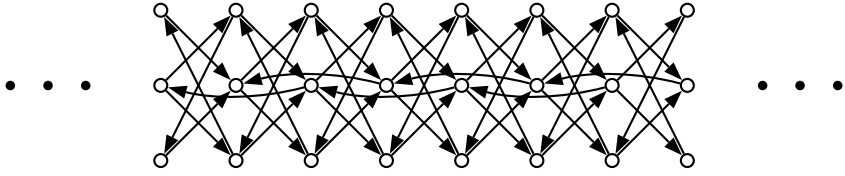

Theorem \ref{thmcluster} hints at the existence of a geometric dynamical system that builds a $Y$-mesh and that transforms in coordinates according to \eqref{eqmain}.  The pentagram map accomplishes this feat for $S=\{(0,0),(2,0),(0,1),(1,1)\}$ and $D=2$.  It is not hard to see for this $S$ that $Y$-meshes cannot exist in any higher dimensions.  Our main geometric result is a description of which pairs $S$ and $D$ can be obtained, with an explicit definition of the corresponding pentagram like map in all realizable cases.  First we need some definitions and notation.
%Theorem \ref{thmcluster} ignores the fundamental issue of whether $Y$-meshes actually exist for given type $S$ and dimension $D$.  In the case $S=\{(0,0),(2,0),(0,1),(1,1)\}$ and $D=2$, we have seen that such $Y$-meshes can be constructed by iterating the pentagram map.  It is not hard to show for the same $S$ that $D=2$ is the only case that is realizable.  Our main result is a description of which pairs $S$ and $D$ can be obtained, with an explicit definition of the corresponding pentagram like map in all realizable cases.  First we need some definitions and notation.

Let $\mathcal{U}_D$ be the space of infinite polygons $A$ with vertices $\ldots, A_{-1}, A_0, A_1, A_2, \ldots$ in $\mathbb{RP}^D$, considered modulo projective equivalence.  Let $\mathcal{U}_{D,m}$ denote the space of $m$-tuples $(A^{(1)}, \ldots, A^{(m)})$ of polygons, also modulo projective equivalence.  Let $\pi_j:\mathcal{U}_{D,m} \to \mathcal{U}_D$ denote the $j$th projection map.  For $A \in \mathcal{U}_{D,m}$, let $\langle A \rangle$ denote the span (i.e. affine hull) of the points of $A$.

Throughout we let $A^{(j)}$ denote the $j$th row of a grid $(P_{i,j})_{i,j \in \mathbb{Z}}$, that is $A_i^{(j)} = P_{i,j}$.  We consider birational maps $F$ of $\mathcal{U}_{D,m}$ (or subvarieties thereof) of the form $F(A^{(1)},\ldots, A^{(m)}) = (A^{(2)},\ldots, A^{(m+1)})$ and iterate $F$ and $F^{-1}$ to define $A^{(j)}$ for all $j \in \mathbb{Z}$.

\begin{defin}
Let $S= \{a,b,c,d\}$ be a $Y$-pin.  Let $X \subseteq \mathcal{U}_{D,m}$ and assume $\langle A \rangle = \mathbb{RP}^D$ for generic $A \in X$.  Say $F: X \to X$ as above is an \emph{$S$-map} if for all $r$ and generic $A \in X$, the points $P_{r+a}, P_{r+b}, P_{r+c}, P_{r+d}$ are collinear and distinct.  Call $m$ the \emph{order} of the $S$-map and $D$ the \emph{dimension}.
\end{defin}

By way of notation, let $D(S) = 2\alpha-1$ where $\alpha$ is the area of the convex hull of $S$ (note $D(S)$ is an integer by Pick's formula).

\begin{thm} \label{thmmain}
Let $S=\{a,b,c,d\}$ be a $Y$-pin.  
\begin{enumerate}
\item If $2 \leq D \leq D(S)$ then there exists an $S$-map in dimension $D$ of order $m=d_2-a_2$.
\item If $D > D(S)$ then no $Y$-mesh of type $S$ exists in dimension $D$.
\end{enumerate}
\end{thm}

The existence of an $S$-map of order $m=d_2-a_2$ implies that a $Y$-mesh $P$ can be determined by its points on $m$ consecutive rows alone.  The procedure to construct each new point will always have the same form: draw lines through two pairs of given points and take the intersection of these lines.  

The choice $m=d_2-a_2$, however, is not minimal.  We consider it a fundamental problem for given $S$ and $D \leq D(S)$ to determine the minimal order $m$ of an $S$-map and to describe the corresponding geometric construction.  We see in examples that as $D$ increases and $m$ decreases, the underlying constructions become extremely intricate.  It is on the surface a miracle that these constructions can be used to define coherent systems, let alone ones admitting a cluster description.  We invite the reader to skip to Section \ref{seczoo} to get a quick taste of the types of systems that fit in our framework.

\begin{rmk} \label{rmkequiv}
Say two $Y$-pins $S$ and $S'$ are equivalent if $S' = g(S)$ for some map $g: \mathbb{Z}^2 \to \mathbb{Z}^2$ of the form
\begin{displaymath}
g(i,j) = (\pm i + kj + i_0, j+j_0)
\end{displaymath}
with $i_0,j_0,k \in \mathbb{Z}$.  In this case there is an identification between $Y$-meshes $P_{i,j}$ of type $S$ and $P'_{i,j} = P_{g^{-1}(i,j)}$ of type $S'$.  Moreover, $g$ sends rows to rows and preserves the positive time (i.e. $j$) direction, so $S$-maps and $S'$-maps are essentially the same objects.  There are five degrees of freedom in choosing $S$ up to equivalence.  Alternatively, if $S$ and $S'$ are related by $(i,j) \mapsto (i,-j)$ then $S'$-maps are inverses of $S$-maps.
\end{rmk}

\begin{rmk}
Definition \ref{defpin} does not allow all four points of $S$ to be collinear, as then it would be impossible for $b-a, c-a, d-a$ to span $\mathbb{Z}^2$.  All the same, it would be interesting to consider this case, which would provide a cluster description of maps in the spirit of Schwartz's pentagram spirals \cite{S3} in terms of Gale-Robinson quivers.
\end{rmk}

The remainder of this paper is organized as follows.  The first half focuses on the geometry of $Y$-meshes.  Section \ref{secdefs} describes the $S$-maps in dimension two and higher.  In Section \ref{secex} we explain the connection between these $S$-maps and previously studied pentagram map generalizations.  Sections \ref{seclattice} and \ref{secrealize} prove Theorem \ref{thmmain} by identifying in which dimensions the $S$-maps can operate.  The problem of minimizing the order of $S$-maps is introduced in Section \ref{secorder}, and the minimal order is found in dimension two.  Section \ref{secfractal} defines certain fractal-shaped subsets of $\mathbb{Z}^2$ designed to better understand which collections of points of a $Y$-mesh lie in a common subspace of some given dimension.  Section \ref{seczoo} is a zoo of examples illustrating various aspects of the theory.

The second half of the paper explores the connection between $Y$-meshes and cluster algebras.  In Section \ref{secquiver} we introduce a two dimensional generalization of the periodic quivers defined by Fordy and Marsh \cite{FM}.  Section \ref{seccluster} identifies the quivers within this class relevant to our systems and also proves \eqref{eqmain} geometrically.  We discuss a one-dimensional version of $Y$-meshes in Section \ref{sec1d} generalizing the lower pentagram map \cite{GSTV}.  Although it is not immediately clear, the quivers corresponding to $S$-maps can be embedded on tori as is established in Section \ref{seclift}.  Finally, Section \ref{secratios} pushes beyond cross ratios to geometrically describe a larger class of $y$-variables of the cluster algebra.

\section{Definition of the $S$-maps} \label{secdefs}

\subsection{The $S$-map in the plane}
We begin by introducing the $S$-map in the plane, which will capture most of the essential components of the general definition.  Throughout the paper we let $\join{A}{B}$ denote the line passing through two points $A$ and $B$.

\begin{prop}
Let $S$ be a $Y$-pin with $D(S)>1$ and let $P$ be a $Y$-mesh of type $S$.  Then
\begin{equation} \label{eqF2}
P_{r+c+d} = \meet{\join{P_{r+a+c}}{P_{r+b+c}}}{\join{P_{r+a+d}}{P_{r+b+d}}}
\end{equation}
and
\begin{equation} \label{eqF2inv}
P_{r+a+b} = \meet{\join{P_{r+a+c}}{P_{r+a+d}}}{\join{P_{r+b+c}}{P_{r+b+d}}}.
\end{equation}
for all $r \in \mathbb{Z}^2$.
\end{prop}

\begin{proof}
By definition of a $Y$-mesh, the points $P_{r+a+c}$, $P_{r+b+c}$ are distinct and both lie on the line $L_{r+c}$.  Also, $P_{r+a+d}$ and $P_{r+b+d}$ are distinct and both lie on $L_{r+d}$.  Lastly, the lines $L_{r+c}$ and $L_{r+d}$ are distinct and both contain $P_{r+c+d}$, proving \eqref{eqF2}.  The proof of \eqref{eqF2inv} is similar.
\end{proof}

In \eqref{eqF2}, the $P_{i,j}$ with smallest $j$-value is $P_{r+a+c}$ and the one with highest $j$-value is $P_{r+c+d}$.  Put another way, all five points in the equation lie within a window of $d_2-a_2+1$ rows of the grid.  As such, \eqref{eqF2} can be seen as a recurrence that inputs $d_2-a_2$ rows of points $P_{i,j}$ and determines the next row.  Similarly, \eqref{eqF2inv} can be used to determine the previous row.  It follows by iterating that a $Y$-mesh is determined by the points of $d_2-a_2$ consecutive rows.

Let $m=d_2-a_2$.  We impose relations on $\mathcal{U}_{2,m}$ of two types (L1) and (L2) defined as follows:  
\begin{itemize}
\item (L1) $P_{r+a},P_{r+b},P_{r+c}$ are collinear %and distinct
\item (L2) $P_{r+b},P_{r+c},P_{r+d}$ are collinear %and distinct
\end{itemize}
Note that a relation like $P_{r+a}, P_{r+b}, P_{r+d}$ collinear would never fit within $m$ consecutive rows since $m=d_2-a_2$.  Define $X_{2,S} \subseteq \mathcal{U}_{2,m}$ by all relevant (L1) and (L2) relations.  Specifically, $A=(A^{(1)},\ldots, A^{(m)}) \in X_{2,S}$ if and only if (L1) holds for all $r$ with $-a_2 < r_2 \leq m-c_2$ and (L2) holds for all $r$ with $-b_2 < r_2 \leq m-d_2$.  Recall here the identification $P_{i,j} = A^{(j)}_i$.

Define $A^{(m+1)} \in \mathcal{U}_{2}$ using \eqref{eqF2} (in the case $r_2=m+1-c_2-d_2$) and define $A^{(0)} \in \mathcal{U}_2$ using \eqref{eqF2inv} (in the case $r_2=-a_2-b_2$).  Write
\begin{align*}
F(A) &= (A^{(2)},\ldots, A^{(m+1)}) \\
G(A) &= (A^{(0)},\ldots, A^{(m-1)})
\end{align*}

\begin{prop} \label{propT2S}
Generically, $F(A) \in X_{2,S}$ and $G(A) \in X_{2,S}$.  Moreover $F,G: X_{2,S} \to X_{2,S}$ are inverse as rational maps.
\end{prop}

\begin{proof}
First show $F(A) \in X_{2,S}$.  Since $(A^{(1)},\ldots, A^{(m)}) \in X_{2,S}$ it suffices to verify the (L1) and (L2) conditions that involve $A^{(m+1)}$.  Let $r$ satisfy $r_2=m+1-c_2$.  The formula \eqref{eqF2} shifted by $-d$ ensures that $P_{r+a}, P_{r+b}, P_{r+c}$ are collinear, confirming (P1).  Now assume $r_2 = m+1-d_2$.  Shifting instead by $-c$ shows $P_{r+a}, P_{r+b}, P_{r+d}$ are collinear.  In addition, we have $-a_2+1 = r_2 \leq m+1-c_2$ so $P_{r+a}, P_{r+b}, P_{r+c}$ are collinear, by a (P1) relation.  Generically $P_{r+a}$ and $P_{r+b}$ are distinct so these imply the (P2) relation $P_{r+b}, P_{r+c}, P_{r+d}$ collinar.  Therefore $F(A) \in X_{2,S}$.  A similar argument shows $G(A) \in X_{2,S}$.

Now let $A \in X_{2,S}$ and use $F$ to build $A^{(m+1)}$.  Showing $G(F(A)) = A$ amounts to verifying that \eqref{eqF2inv} holds for $r$ with $r_2 = 1-a_2-b_2$.  An (L2) relation shifted by $a$ shows that $P_{r+a+b}, P_{r+a+c}, P_{r+a+d}$ are collinear so
\begin{displaymath}
P_{r+a+b} \in \join{P_{r+a+c}}{P_{r+a+d}}.
\end{displaymath}
An (L1) relation shifted by $b$ shows that $P_{r+a+b}, P_{r+2b},P_{r+b+c}$ are collinear.  An (L2) relation shifted by $b$ shows $P_{r+2b}, P_{r+b+c}, P_{r+b+d}$ are collinear.  Generically, these imply
\begin{displaymath}
P_{r+a+b} \in \join{P_{r+b+c}}{P_{r+b+d}}.
\end{displaymath}
So \eqref{eqF2inv} does hold.  A similar argument shows $F(G(A)) = A$.
\end{proof}

We now fix notation $T_{2,S}: X_{2,S} \to X_{2,S}$ for the map $F$ constructed above.  Starting from $A \in X_{2,S}$, we can iterate $T_{2,S}$ and $T_{2,S}^{-1}$ to fill up the whole $P_{i,j}$ array, and all (L1) and (L2) relations will hold.  It follows that generically $P_{r+a}$, $P_{r+b}$, $P_{r+c}$, and $P_{r+d}$ are  collinear.

\begin{ex} \label{exF2}
Let $S=\{(-1,1), (1,2), (0,3), (0,4)\}$ and $D=2$.  Figure \ref{figF2} illustrates the $S$-map
\begin{displaymath}
T_{2,S} : X_{2,S} \to X_{2,S}
\end{displaymath}
A triple of enclosed points indicates that said points are collinear (as are any triple in the same relative position).  In this example, the (L1) relation is $P_{i-1,1}, P_{i+1,2}, P_{i,3}$ collinear for all $i$ and the (L2) relation is $P_{i+1,1}, P_{i,2}, P_{i,3}$ collinear for all $i$.  The space $X_{2,S} \subseteq \mathcal{U}_{2,3}$ consists of triples of infinite polygons satisfying these relations.  

The two line segments indicate the construction of a $P_{i,4}$, specifically the one marked by a $*$.  The $S$-map $T_{2,S}$ is defined by
\begin{displaymath}
P_{i,4} = \meet{\join{P_{i-1,1}}{P_{i+1,2}}}{\join{P_{i-1,2}}{P_{i+1,3}}}.
\end{displaymath}
\end{ex}

\begin{figure}
\begin{pspicture}(0,.5)(14,5.5)
\rput(1.5,1){
\pnode(-1,1){a}
\pnode(1,2){b}
\pnode(0,3){c}
\pnode(0,4){d}
\psdots(a)(b)(c)(d)
\uput[d](a){$a$}
\uput[d](b){$b$}
\uput[d](c){$c$}
\uput[d](d){$d$}
}
\rput(4.5,1){
\multirput(0,0)(1,0){7}{\multirput(0,0)(0,1){3}{\psdots(1,1)}}
\psaxes[axesstyle=none,ticks=none,showorigin=false,labelsep=0]{->}(7,3)
\uput[d](.3,0){$i$}
\uput[l](0,.3){$j$}
\psline(2,1)(4,2)
\psline(2,2)(4,3)
\psdots[dotstyle=asterisk](3,4)
\rput(6,0){
\psbezier(-1.18,.91)(-1.33,1.21)(.67,1.79)(.82,2.09)
\psbezier(.82,2.09)(.97,2.39)(-.38,2.91)(-.14,3.14)
\psbezier(-.14,3.14)(.09,3.38)(1.33,2.21)(1.18,1.91)
\psbezier(1.18,1.91)(1.03,1.61)(-1.03,.61)(-1.18,.91)
}
\rput(6,0){
\psbezier(1.14,.86)(.91,.62)(-.03,1.61)(-.18,1.91)
\psbezier(-.18,1.91)(-.33,2.21)(-.33,3.2 )(0   ,3.2 )
\psbezier(0   ,3.2 )(.33,3.2 )(.03,2.39)(.18,2.09)
\psbezier(.18,2.09)(.33,1.79)(1.38,1.09)(1.14,.86)

}

}
\end{pspicture}
\caption{The $Y$-pin $S = \{(-1,1), (1,2), (0,3), (0,4)\}$ (left) and the corresponding $S$-map (right). }
\label{figF2}
\end{figure}
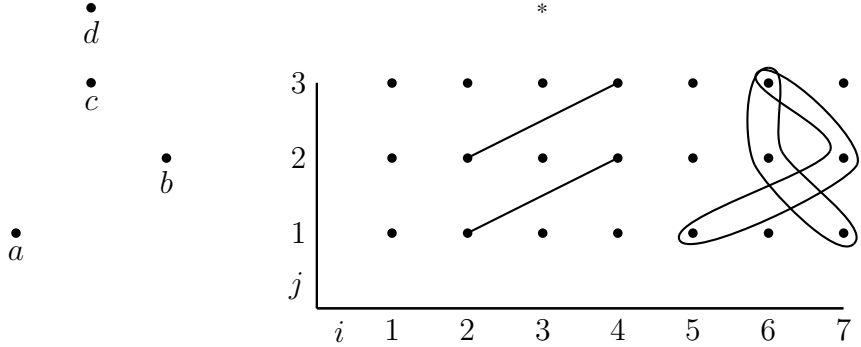

\subsection{Higher dimensional $S$-maps}
Now suppose $D > 2$, but keep $m=d_2-a_2$.  The identities \eqref{eqF2} and \eqref{eqF2inv} still follow from the definition of a $Y$-mesh.  The only new feature is a certain coplanarity condition (P3), which is needed in addition to (L1) and (L2) from before:
\begin{itemize}
\item (P3) $P_{r+a+c},P_{r+a+d},P_{r+b+c},P_{r+b+d}$ are coplanar. 
\end{itemize}

Define $X_{D,S} \subseteq \mathcal{U}_{D,m}$ by $A \in X_{D,S}$ if and only if the $P_{i,j} = A_i^{(j)}$ satisfy (L1), (L2), and (P3) for all relevant $r$.  Specifically, (L1) should hold for $-a_2 < r_2 \leq m-c_2$, (L2) should hold for $-b_2 < r_2 \leq m-d_2$, and (P3) should hold for $-a_2-c_2 < r_2 \leq m-b_2-d_2$.

The maps $F:X_{D,S} \to X_{D,S}$ and $G:X_{D,S} \to X_{D,S}$ are defined using \eqref{eqF2} and \eqref{eqF2inv} as before.  By a (P3) relation, the four points used in the construction are coplanar making it possible to join them in pairs and intersect the resulting lines.

\begin{prop} \label{propTS}
The maps $F$ and $G$ do in fact map $X_{D,S}$ to $X_{D,S}$, and they are inverse to each other.
\end{prop}

\begin{proof}
Most of the proof is identical to that of Proposition \ref{propT2S}.  What remains is to verify that the maps preserve the (P3) conditions.  Let $A=(A^{(1)},\ldots A^{(m)}) \in X_{D,S}$ and $F(A) = (A^{(2)}, \ldots, A^{(m+1)})$.  The new (P3) relations to check are those with $r$ satisfying $r_2+b_2+d_2 = m+1$.  By the definition of $F$, 
\begin{displaymath}
P_{r+b+d} = \meet{\join{P_{r+a+b}}{P_{r+2b}}}{\join{P_{r+a+b-c+d}}{P_{r+2b-c+d}}}.
\end{displaymath}
From the above and an (L1) relation, the points $P_{r+a+b}, P_{r+2b}, P_{r+b+c}, P_{r+b+d}$ generically are distinct and lie on a line.  By an (L2) relation, $P_{r+a+b}, P_{r+a+c}, P_{r+a+d}$ lie on another line.  These two lines intersect (at $P_{r+a+b}$), so $P_{r+a+c}, P_{r+a+d}, P_{r+b+c}, P_{r+b+d}$ lie on a plane as desired.
\end{proof}

Let $T_{D,S}:X_{D,S} \to X_{D,S}$ denote the map $F$.  Apply $T_{D,S}$ forwards and backwards to build the full $P_{i,j}$ array.  It follows as usual from the (L1) and (L2) relations that
\begin{displaymath}
P_{r+a}, P_{r+b}, P_{r+c}, P_{r+d}
\end{displaymath}
are collinear for all $r \in \mathbb{Z}^2$.

We seem to have obtained an $S$-map in all dimensions $D$.  What is missing is the requirement that the resulting $Y$-meshes actually span all of $\mathbb{RP}^D$ rather than some proper subspace.  We return to this matter in Section \ref{secrealize}.

\section{Connection with other pentagram maps} \label{secex}
\begin{ex} \label{exGSTV}
Fix $p,q$ relatively prime with $1 \leq q < p$.  Let $S = \{(0,0),(p,0),(0,1),(q,1)\}$ and $D=2$.  Then $m=1-0=1$.  The (L1) and (L2) relations do not fit, so arbitrary polygons are allowed, i.e. $X_{2,S} = \mathcal{U}_{2}$.  Letting $A_i=P_{i,1}$ and $B_i = P_{i,2}$, equation \eqref{eqF2} (applied at $r=(i-q,0)$) implies $T_{2,S}(A)=B$ where
\begin{equation} \label{eqTpq}
B_i = \meet{\join{A_{i-q}}{A_{i+p-q}}}{\join{A_{i}}{A_{i+p}}}.
\end{equation}
In words, $B$ is formed by taking distance $p$ diagonals of $A$, and then intersecting diagonals that are a distance $q$ apart.  The case $p=2,q=1$ corresponds to the pentagram map.  

Note the convex hull of $S$ is a trapezoid of area $(p+q)/2$ so $D(S)=p+q-1$.  If $2 < D \leq p+q-1$ then $X_{D,S} \subseteq \mathcal{U}_{D}$ is defined by a (P3) relation
\begin{displaymath}
X_{D,S} = \{A \in \mathcal{U}_{D} : A_i, A_{i+q}, A_{i+p}, A_{i+p+q} \textrm{ coplanar for all $i$}\}.
\end{displaymath}
The map $T_{D,S}$ is defined by \eqref{eqTpq} as before.  In the case $q=1$, Gekhtman, Shapiro, Tabachnikov, and Vainshtein \cite{GSTV} extensively studied the corresponding systems ($T_{D,S}$ in our notation, $T_{p+1}$ in theirs).  They call these maps \emph{higher pentagram maps} and the polygons on which they act \emph{corrugated polygons}.
\end{ex}

\begin{ex} \label{exKS}
Let $S=\{(-1,1), (1,1), (0,2), (0,3)\}$ and $D=3$.  Figure \ref{figKS} illustrates the corresponding map.  Letting $A_i = P_{i,1}$ and $B_i=P_{i,2}$, we have $(A,B) \in X_{3,S}$ if and only if $A_{i-1},B_i,A_{i+1}$ collinear and $A_{i-1},A_{i+1},B_{i-1},B_{i+1}$ coplanar for all $i$.  The $S$-map is $T_{3,S}(A,B) = (B,C)$ where
\begin{displaymath}
C_i = \meet{\join{A_{i-1}}{A_{i+1}}}{\join{B_{i-1}}{B_{i+1}}}.
\end{displaymath}
Now, by assumption $B_{i-1} \in \join{A_{i-2}}{A_i}$ and $B_{i+1} \in \join{A_i}{A_{i+2}}$, so the line $\join{B_{i-1}}{B_{i+1}}$ is contained in the plane $\langle A_{i-2}, A_i, A_{i+2} \rangle$.  Therefore
\begin{displaymath}
C_i \in \meet{\join{A_{i-1}}{A_{i+1}}}{\langle A_{i-2}, A_i, A_{i+2} \rangle}.
\end{displaymath}
The right hand side is the intersection of a line and a plane in three space, so it is generically a point and
\begin{displaymath}
C_i = \meet{\join{A_{i-1}}{A_{i+1}}}{\langle A_{i-2}, A_i, A_{i+2} \rangle}.
\end{displaymath}
As such we can forget about $B$ and write $C = F(A)$ where $F:\mathcal{U}_{3} \to \mathcal{U}_{3}$ is as above.  The map $F$ is the \emph{short diagonal hyperplane map} studied by Khesin and Soloviev \cite{KS1}.
\end{ex}

\begin{figure}
\begin{pspicture}(0,.5)(14,4.5)
\rput(0,1){
\pnode(1,1){a}
\pnode(3,1){b}
\pnode(2,2){c}
\pnode(2,3){d}
\psdots(a)(b)(c)(d)
\uput[d](a){$a$}
\uput[d](b){$b$}
\uput[d](c){$c$}
\uput[d](d){$d$}
}
\rput(5,1){
\multirput(0,0)(1,0){7}{\multirput(0,0)(0,1){2}{\psdots(1,1)}}
\psaxes[axesstyle=none,ticks=none,showorigin=false,labelsep=0]{->}(7,2)
\uput[d](.3,0){$i$}
\uput[l](0,.3){$j$}
\psline(2,1)(4,1)
\psline(2,2)(4,2)
\psdots[dotstyle=asterisk](3,3)
\rput(6,0){
\psbezier(-1.14,.86)(-1.38,1.09)(-.33,2.2 )(0   ,2.2 )
\psbezier(0   ,2.2 )(.33,2.2 )(1.38,1.09)(1.14,.86)
\psbezier(1.14,.86)(.91,.62)(.33,1.8 )(0   ,1.8 )
\psbezier(0   ,1.8 )(-.33,1.8 )(-.91,.62)(-1.14,.86)
}
}
\end{pspicture}
\caption{The $S$-map for $S = \{(-1,1),(1,1),(0,2),(0,3)\}$.}
\label{figKS}
\end{figure}
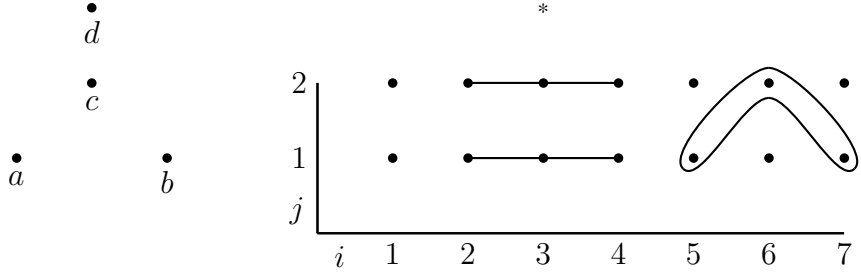

The map $F$ in the previous example is not itself an $S$-map.  More precisely, it does not trace out a full $Y$-mesh, but rather every second row of a $Y$-mesh.  A large family of $S$-maps are related in a similar way to so-called $(I,J)$-maps defined by Khesin and Soloviev \cite{KS2}.  Fix $D \geq 2$ and let $I=(s_1,\ldots, s_{D-1})$, $J=(t_1,\ldots, t_{D-1})$.  Given $A \in \mathcal{U}_{D}$ define the $I$-hyperplanes of $A$ to be
\begin{displaymath}
H_i = \langle A_i, A_{i+s_1}, A_{i+s_1+s_2}, \ldots, A_{i+s_1+\ldots + s_{D-1}} \rangle.
\end{displaymath}
Define a new polygon $B$ whose vertices are the $J$-intersections of the $H_i$
\begin{displaymath}
B_{i-i_0} = H_i \cap H_{i+t_1} \cap H_{i+t_1+t_2} \cap \ldots \cap H_{i+t_1+\ldots + t_{D-1}}.
\end{displaymath}
Everything is defined up to some uniform shift $i_0$ of indices.  The map $T_{I,J}:\mathcal{U}_{D} \to \mathcal{U}_{D}$ is $T_{I,J}(A)=B$.  

A basic property of these maps is that they are birational.  In fact, if $I=(s_1,\ldots, s_{D-1})$, $J=(t_1,\ldots, t_{D-1})$ then $T_{I,J}^{-1} = T_{J^*,I^*}$ where $I^*=(s_{D-1},\ldots, s_1)$ and $J^*=(t_{D-1},\ldots, t_1)$.  We can always take the $s_k$ and $t_k$ to be positive, but it will be convenient to allow negative values as well.  There is no problem with the definition provided the partial sums of $I$ are all distinct and similarly for $J$.

In some cases, the $(I,J)$-map which is usually defined as an intersection of $D$ hyperplanes can be expressed more simply as the intersection of two subspaces of complimentary dimension.
\begin{prop}
Suppose $I=(s,s,\ldots, s)$ and $J = (s,\ldots, s, t, s, \ldots s)$ where $\gcd(s,t)=1$ and the single $t$ is in position $k$ of $J$.  Then $B=T_{I,J}(A)$ is given by
\begin{equation} \label{eqIJconst}
\begin{split}
B_{i-i_0} = &\langle A_{i+(k-1)s}, A_{i+ks},\ldots, A_{i+(D-1)s} \rangle \\ 
&\cap \langle A_{i+(D-2)s+t}, A_{i+(D-1)s+t}, \ldots, A_{i+(D+k-2)s + t} \rangle.
\end{split}
\end{equation}
\end{prop}
\begin{proof}
Let $H_i = \langle A_i, A_{i+s}, A_{i+2s}, \ldots, A_{i+(D-1)s} \rangle$.  It is easy to see that 
\begin{displaymath}
H_i \cap H_{i+s} \cap \ldots \cap H_{i+(k-1)s} = \langle A_{i+(k-1)s}, A_{i+ks},\ldots, A_{i+(D-1)s} \rangle
\end{displaymath}
and
\begin{displaymath}
H_{i+(k-1)s+t} \cap H_{i+ks+t} \cap \ldots \cap H_{i+(D-2)s+t} = \langle A_{i+(D-2)s+t}, A_{i+(D-1)s+t}, \ldots, A_{i+(D+k-2)s + t} \rangle.
\end{displaymath}
The result follows from the definition of $T_{I,J}$.
\end{proof}

Say that a $Y$-pin $S$ is horizontal if $a_2=b_2$.  Applying a translation and by choice of $a$ we can assume $a_2=b_2=0$ and $0=a_1<b_1$.  Let $p=c_2$ and $q=d_2$.  In order to have $b-a,c-a,d-a$ span all of $\mathbb{Z}^2$, it must be the case that $p, q$ are relatively prime.  The connection to $(I,J)$-maps arises in dimension $p+q$, so assume $D(S) \geq p+q$.  Let $(P_{i,j})_{i,j\in\mathbb{Z}}$ be a $Y$-mesh of type $S$ and dimension $D=p+q$. 

\begin{prop} 
Given the above,
\begin{equation} \label{eqzontal}
%\begin{split} 
P_{qc+pd} \in \langle P_{pd}, P_{b+pd}, P_{2b+pd}, \ldots, P_{qb+pd} \rangle 
\cap \langle P_{qc}, P_{b+qc}, P_{2b+qc}, \ldots, P_{pb+qc} \rangle.
%P_{qc+pd} \in &\langle P_{qa+pd}, P_{(q-1)a+b+pd}, P_{(q-2)a+2b+pd}, \ldots, P_{qb+pd} \rangle \\
%&\cap \langle P_{pa+qc}, P_{(p-1)a+b+qc}, P_{(p-2)a+2b+qc}, \ldots, P_{pb+qc} \rangle.
%\end{split}
\end{equation}
\end{prop}

\begin{proof}
Let 
\begin{displaymath}
%V_k = \langle \{P_{\alpha a + \beta b + (q-k)c + pd} : \alpha, \beta \geq 0, \alpha+\beta=k\} \rangle.
V_k = \langle \{P_{\beta b + (q-k)c + pd} : 0 \leq \beta \leq k\} \rangle.
\end{displaymath}
Then $V_0 = P_{qc+pd}$ and $V_q$ is the first space in the intersection.  By induction, it suffices to show $V_{k-1} \subseteq V_k$ for $k=1,\ldots, q$.  Suppose $0 \leq \beta \leq k-1$.  The $Y$-mesh property applied for $r= \beta b + (q-k)c + pd$ shows
\begin{displaymath}
P_{\beta b + (q-(k-1))c + pd} \in \join{P_{(\beta b + (q-k)c + pd}}{P_{(\beta+1) b + (q-k)c + pd}} \subseteq V_k.
\end{displaymath}
Therefore $V_{k-1} \subseteq V_k$ as desired.  A similar argument shows
\begin{displaymath}
P_{qc+pd} \in \langle P_{qc}, P_{b+qc}, P_{2b+qc}, \ldots, P_{pb+qc} \rangle.
\end{displaymath}
\end{proof}

Note $qc+pd$ has $j=qc_2+pd_2=2pq$ while all the points on the right hand side of \eqref{eqzontal} have $j=qp$.  Moreover, if these $P_{i,pq}$ are in general position then the right hand side is the intersection of a $q$-space and a $p$-space in $\mathbb{RP}^{p+q}$ and hence the intersection is a single point, namely $P_{qc+pd}$.  We conjecture that under the current assumptions, a single row of the $Y$-mesh is arbitrary, and hence generically $P_{qc+pd}$ can be determined in this manner.  

\begin{conj} \label{conjcover}
Suppose $S$ is horizontal as above with $D(S) \geq p+q$.  Then the map $\pi_1:X_{p+q,S} \to \mathcal{U}_{p+q}$ has dense image.  Restricting to twisted configurations, i.e. $A$ satisfying $A_{i+n}^{(j)} = \phi(A_i^{(j)})$ for some fixed $n \geq 1$ and projective transformation $\phi$, the map $\pi_1$ is finite-to-one.
\end{conj}

\begin{prop} \label{propzontal}
Suppose $S$ is horizontal as above with $D(S) \geq p+q$.  Let $(P_{i,j})_{i,j\in\mathbb{Z}}$ be a $Y$-mesh of type $S$ and dimension $p+q$.  As usual, let $A^{(j)}$ be the polygon $\ldots, P_{0,j}, P_{1,j} \ldots$.  Assuming Conjecture \ref{conjcover},
\begin{displaymath}
A^{(j+pq)} = T_{I,J}(A^{(j)})
\end{displaymath}
where $I=(s,s,\ldots, s), J=(s,\ldots, s, t, s, \ldots, s)$, $s=b_1>0$, $t=qc_1-pd_1-(q-1)b_1$, and the unique $t$ of $J$ is in position $k=p$.
\end{prop}

\begin{proof}
By Conjecture \ref{conjcover} and \eqref{eqzontal}, we have generically that
\begin{displaymath}
P_{qc+pd} = \langle P_{pd}, P_{b+pd}, P_{2b+pd}, \ldots, P_{qb+pd} \rangle 
\cap \langle P_{qc}, P_{b+qc}, P_{2b+qc}, \ldots, P_{pb+qc} \rangle.
\end{displaymath}
Shifting indices, any $P_{i,j+pq}$ can be expressed in a similar way in terms of the $P_{i,j}$.  This equation is of the same form as \eqref{eqIJconst}, and the correspondence between $S=\{(0,0),(b_1,0),(c_1,p),(d_1,q)\}$ and $I=(s,s,\ldots, s)$, $J=(s,\ldots, s, t, s, \ldots, s)$ can be easily read off.
\end{proof}

In words, the map $T_{I,J}$ traces out every $pq$th row of a $Y$-mesh.  Hence the $S$-map $T_{p+q,S}$ can be thought of as a $pq$th root of $T_{I,J}$, up to the conjecturally finite cover $\pi_1$.  The process can be reversed and any $I$, $J$ as above can be realized, assuming $s,t$ are relatively prime and the position of the $t$ is relatively prime to $D$.  Taking inverses we also get such pairs with $I$ and $J$ reversed.

\begin{ex}
The \emph{short diagonal hyperplane map} \cite{KS1} is $T_{I,J}$ where $I=(2,2,\ldots, 2)$ and $J=(1,1,\ldots, 1)$.  Suppose first that $D$ is odd.  Then the partial sums $0,1,\ldots, D-1$ of $J$ can be reordered as $1,3,\ldots, D-2, 0, 2, \ldots, D-1$.  Therefore, it is equivalent to take $J=(2,\ldots, 2, 2-D, 2, \ldots, 2)$ where the $2-D$ is in position $(D-1)/2$.  We are in the setting of Proposition \ref{propzontal} with $s=2$, $t=2-D$, and $k=(D-1)/2$.  Working backwards, $p = (D-1)/2$, $q=(D+1)/2$, $b_1=2$ and $c_1,d_1$ satisfy
\begin{displaymath}
\frac{D+1}{2}c_1 - \frac{D-1}{2}d_1 - \frac{D-1}{2}2 = 2-D.
\end{displaymath}
A possible choice is $c_1=d_1=1$.  Therefore, the short diagonal hyperplane map in odd dimension $D$ is related to the $S$-map with 
\begin{displaymath}
S = \left\{(0,0),(2,0),\left(1,\frac{D-1}{2}\right), \left(1,\frac{D+1}{2}\right)\right\}.
\end{displaymath}

If $D$ is even then the partial sums reorder to $0,2,\ldots, D-2, 1, 3,\ldots, D-1$.  As such we can take $J=(2,\ldots, 2, 3-D, 2, \ldots, 2)$ where the $3-D$ is in position $k=D/2$.  Assuming $D \geq 4$, this violates $\gcd(k,D)=1$ so there is no affiliated $S$-map in this case.
\end{ex}

\begin{ex}
The \emph{dented pentagram map} \cite{KS2} is the map $T_{I,J}$ where $I=(1,\ldots, 1,2,1,\ldots, 1)$ and $J=(1,1,\ldots, 1)$.  Suppose the 2 in $I$ is in position $D-k$.  If $\gcd(k,D)=1$ then the inverse $T_{J^*,I^*}$ is of the form above with $s=1$ and $t=2$.  So $T_{J^*,I^*}$ is related to the $S$-map where 
\begin{displaymath}
S = \{(0,0), (1,0), (c_1,k), (d_1,D-k) \}
\end{displaymath}
and $c_1,d_1$ are chosen so that $(D-k)c_1-kd_1-(D-k-1) = 2$.  A solution is guaranteed to exist since $\gcd(D-k,k)=1$.  The $S$ corresponding to the dented map itself is obtained by applying $(i,j)\mapsto(i,-j)$
\begin{displaymath}
S = \{(0,0), (1,0), (c_1,-k), (d_1,-D+k) \}.
\end{displaymath}
\end{ex}

Table \ref{tab3D} lists some $(I,J)$-maps in three dimensions along with the corresponding $S$.  The list includes the short diagonal hyperplane map \cite{KS1}, the dented and deep dented maps \cite{KS2}, and an unnamed map investigated numerically in \cite{KS3}.  For these examples we can verify Conjecture \ref{conjcover}.  

\begin{table}
\begin{tabular}{l|l|l|l}
Map name& $I$ & $J$ & $S$ \\
\hline
Short diagonal & $(2,2)$ & $(1,1)$ & $\{(0,0),(2,0),(1,1),(1,2)\}$\\
Dented & $(1,2)$ & $(1,1)$ &  $\{(0,0),(1,0),(2,-1),(1,-2)\} $\\
Deep dented & $(1,t)$ & $(1,1)$ & $\{(0,0),(1,0),(t,-1),(t-1,-2)\}$\\
? & $(2,2)$ & $(1,2)$ & $\{(0,0),(2,0),(2,1),(1,2)\}$
\end{tabular} 
\caption{The $Y$-pin $S$ associated with several $(I,J)$-maps.}
\label{tab3D}
\end{table}

\begin{rmk}
For $I,J$ as above, we obtain a cluster description of $T_{I,J}$ by way of the corresponding $S$-map, settling in these cases the open question (see e.g. \cite[Problem 5.12]{KS2}) of whether such a description exists.  Another fundamental problem is determining which of the $(I,J)$-maps are integrable.  Integrability has been verified by Khesin and Soloviev for short diagonal hyperplane maps \cite{KS1} and dented and deep dented maps \cite{KS2}.  We consider it likely that all $S$-maps, and hence all associated $(I,J)$-maps are integrable.  It should be noted however that not all integrable cases fit into our framework.  For example, the dented map $I=(1,2,1), J=(1,1,1)$ in dimension $D=4$ is integrable but has the dent in position 2 which is not relatively prime to $D$.
\end{rmk}

\section{The auxiliary lattice} \label{seclattice}
We present two alternate definitions of $D(S)$ to the one given in the introduction, and prove their equivalence.  The three definitions are tied together by a certain sublattice of $\mathbb{Z}^2$ associated with $S$.  %The construction of the lattice depends on whether the convex hull of $S$ is a trianlge or a quadrilateral.

\begin{defin}
Let $S = \{s_1,\ldots, s_k\} \subseteq \mathbb{Z}^2$.  A \emph{convex relation} on $S$ is a nontrivial linear relation
\begin{displaymath}
m_1s_1 + \ldots + m_ks_k=0
\end{displaymath}
such that $m_1,\ldots, m_k \in \mathbb{Z}$ and $m_1+\ldots + m_k = 0$.  In this case
\begin{displaymath}
\sum_{m_i>0} m_i = \sum_{m_i < 0} -m_i,
\end{displaymath}
so call this common value the \emph{magnitude} of the convex relation.  Say the relation is \emph{primitive} if $\gcd(m_1,\ldots, m_k)=1$.
\end{defin}

The reason for the name is that, if $M$ is the magnitude then the relation can be rewritten as
\begin{displaymath}
\sum_{m_i>0} \frac{m_i}{M}s_i = \sum_{m_i<0} \frac{-m_i}{M}s_i
\end{displaymath}
i.e. a convex combination of some $s_i$ equals a convex combination of others.

We care only about the case $|S|=4$.  If the elements of $S$ are vertices of a convex quadrilateral, then the intersection of the diagonals is a convex combination of each pair of opposite vertices.  If the convex hull of $S$ is a triangle, then one element of $S$ is a convex combination of the other three (and also of itself).  In both cases we have a unique primitive convex relation for $S$ (up to multiplication by -1).  Let $M(S)$ denote the magnitude of this relation.

The unique convex relation of $S$ naturally partitions $S$ into two pieces (opposite pairs in the quadrilateral case, the triple of vertices and the single interior point in the triangle case).  Define $\Lambda(S)$ to be the sublattice of $\mathbb{Z}^2$ spanned by vectors of the form $s_2-s_1$ where $s_1,s_2\in S$ are elements of the same piece.  Use the notation $|\mathbb{Z}^2 : \Lambda(S)|$ for the index of this sublattice.

\begin{prop} \label{proplattice}
For any $Y$-pin $S$
\begin{displaymath}
M(S) = |\mathbb{Z}^2 : \Lambda(S)| = D(S)+1.
\end{displaymath}
\end{prop}

\begin{proof}
Let
\begin{align*}
m_1 &= \det \left[ \begin{array}{ccc} 
s_2 & s_3 & s_4 \\
1 & 1 & 1 \\
\end{array}
\right] \\
m_2 &= -\det \left[ \begin{array}{ccc} 
s_1& s_3 & s_4 \\
1 & 1 & 1 \\
\end{array}
\right] \\
m_3 &= \det \left[ \begin{array}{ccc} 
s_1 & s_2 & s_4 \\
1 & 1 & 1 \\
\end{array}
\right] \\
m_4 &= -\det \left[ \begin{array}{ccc} 
s_1 & s_2 & s_3 \\
1 & 1 & 1 \\
\end{array}
\right] 
\end{align*}
Standard determinant tricks show that $m_1s_1 + m_2s_2 + m_3s_3 + m_4s_4=0$ and $m_1+m_2+m_3+m_4=0$.  Since $s_2-s_1, s_3-s_1, s_4-s_1$ generate $\mathbb{Z}^2$, it follows that $\gcd(m_2,m_3,m_4)=1$ so this relation is primitive.  Geometrically, we have $|m_i|$ is twice the area of the triangle formed by the three elements of $S \setminus \{s_i\}$.

First suppose the convex hull of $S$ is a triangle, say with vertices $s_1,s_2,s_3$.  Then we can take $m_1,m_2,m_3 \leq 0$ and $m_4 >0$.  Twice the area of the convex hull, which by definition equals $D(S)+1$, is given by $m_4 = M(S)$.  Meanwhile, $\Lambda(S)$ is generated by $s_2-s_1, s_3-s_1$.  A fundamental parallelogram of $\Lambda$ clearly has area twice that of the convex hull of $S$, (see Figure \ref{figlatticetri}).  As such $|\mathbb{Z}^2 : \Lambda(S)| = D(S)+1$.

Now suppose the convex hull of $S$ is a quadrilateral, with vertices $s_1,s_2,s_3,s_4$ in that order.  Then we can take $m_1,m_3 < 0$ and $m_2,m_4 >0$.  The convex hull of $S$ can be broken up into two triangles $\triangle s_1s_3s_4$ and $\triangle s_1s_2s_3$ so $D(S)+1 = m_2 + m_4 = M(S)$.  The lattice $\Lambda(S)$ is generated by $s_3-s_1$ and $s_4-s_2$.  Figure \ref{figlatticequad} shows an example fundamental parallelogram in this case.  As is indicated, the original quadrilateral can be inscribed in the parallelogram.  Again, the area of the parallelogram is twice that of the convex hull of $S$.
\end{proof}

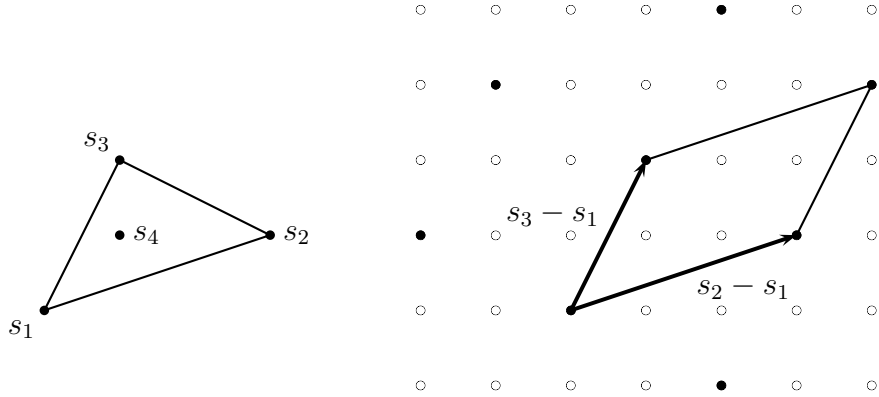
\begin{figure}
\begin{pspicture}(11,6)
\rput(-1,-.5){
\psdots(1,2)(4,3)(2,4)(2,3)
\pspolygon(1,2)(4,3)(2,4)
\uput[dl](1,2){$s_1$}
\uput[r](4,3){$s_2$}
\uput[ul](2,4){$s_3$}
\uput[r](2,3){$s_4$}
}

\rput(4,-.5){
\multirput(0,0)(1,0){7}{\multirput(0,0)(0,1){6}{\psdots[dotstyle=o](1,1)}}
\psdots(1,3)(2,5)(3,2)(4,4)(5,1)(5,6)(6,3)(7,5)

\pspolygon(3,2)(6,3)(7,5)(4,4)
\psline[linewidth=1.5pt]{->}(3,2)(6,3)
\uput[-30](4.5,2.5){$s_2-s_1$}
\psline[linewidth=1.5pt]{->}(3,2)(4,4)
\uput[ul](3.5,3){$s_3-s_1$}
%\pspolygon(3,2)(6,3)(7,5)(4,4)
}
\end{pspicture}
\caption{The lattice $\Lambda(\{s_1,s_2,s_3,s_4\})$ has basis $s_2-s_1,s_3-s_1$.  A fundamental parallelogram is indicated.}
%\caption{The lattice $\Lambda(S)$ for $S=\{(0,0),(4,1),(1,2),(1,1)\}$ with a fundamental parallelogram indicated.}
\label{figlatticetri}
\end{figure}

\begin{figure}
\begin{pspicture}(11,6)
\rput(-1,-.5){
\pspolygon[showpoints=true](2,2)(4,2)(3,4)(1,3)
\uput[dl](2,2){$s_1$}
\uput[dr](4,2){$s_2$}
\uput[u](3,4){$s_3$}
\uput[ul](1,3){$s_4$}
}

\rput(4,-.5){
\multirput(0,0)(1,0){7}{\multirput(0,0)(0,1){6}{\psdots[dotstyle=o](1,1)}}
\psdots(1,6)(2,1)(3,3)(4,5)(6,2)(7,4)

\pspolygon(3,3)(6,2)(7,4)(4,5)
\psline[linewidth=1.5pt]{->}(6,2)(7,4)
\uput[dr](6.5,3){$s_3-s_1$}
\psline[linewidth=1.5pt]{->}(6,2)(3,3)
\uput[210](4.5,2.5){$s_4-s_2$}
\pspolygon[linestyle=dashed](4.29,2.57)(6.29,2.57)(5.29,4.57)(3.29,3.57) %fractional parts (2/7,4/7)
}
\end{pspicture}
\caption{The lattice $\Lambda(\{s_1,s_2,s_3,s_4\})$ has basis $s_3-s_1,s_4-s_2$.  A fundamental parallelogram is indicated.}
%\caption{The lattice $\Lambda(S)$ for $S=\{(0,0),(4,1),(1,2),(1,1)\}$ with a fundamental parallelogram indicated.}
\label{figlatticequad}
\end{figure}
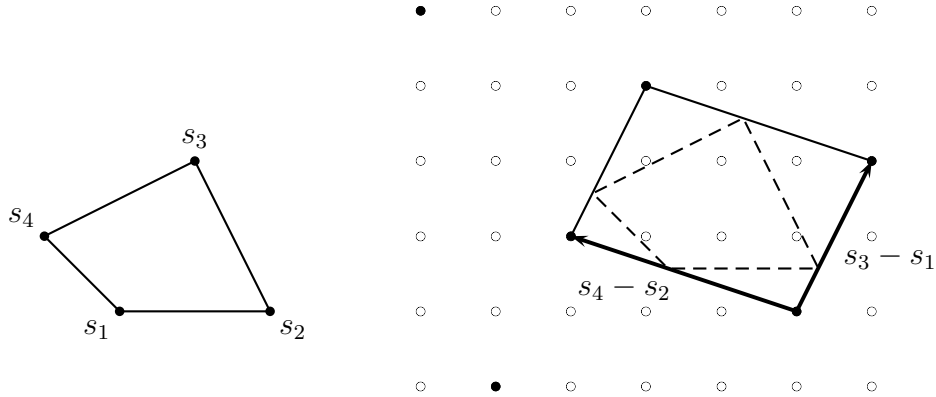

\begin{rmk}
There is a boundary case in which the convex hull of $S$ is a triangle, but the fourth point lies on a side rather than the interior of the triangle.  For example, let $S = \{a,b,c,d\} = \{(0,0), (1,0), (0,2), (0,3)\}$.  The area of the convex hull is $3/2$ so $D(S) = 2$.  A primitive convex relation, $a-3c+2d=0$ does have magnitude $M=3 =D(S)+1$.  There are two choices for the lattice $\Lambda(S)$.  Grouping $b$ with $a$ and $d$ yields $\Lambda(S)$ generated by $b-a=(1,0)$ and $d-a = (0,3)$.  Grouping $b$ with $c$ yields $\Lambda(S)$ generated by $d-a=(0,3)$ and $c-b=(-1,2)$.  In general, Proposition \ref{proplattice} holds whichever choice is made.
\end{rmk}

\section{Realizability} \label{secrealize}
To complete the proof of Theorem \ref{thmmain} we need to analyze the spaces $X_{D,S}$, whose elements are collections of points subject to (P1), (P2), and (L3) relations.  Specifically, we want to know if the vertices of a generic element actually span all of $\mathbb{RP}^D$.  We generalize the relation types and give a general approach to this sort of problem.  We also deprojectivize to $\mathbb{R}^{D+1}$ for convenience. 

Let $m \geq 1$.  Let $R = \mathbb{Z} \times \{1,2,\ldots, m\}$.  Let $C$ be a set of finite subsets of $R$, none of which being contained in any other.  Define a $C$-arrangement to be a collection of vectors $v_r \in \mathbb{R}^{D+1}$ indexed by $r \in R$ such that for each $I \in C$
\begin{itemize}
\item the set $\{v_r : r \in I\}$ is linearly dependent, and
\item each proper subset of $\{v_r : r \in I\}$ is linearly independent.
\end{itemize}
In matroid terminology, the elements of $C$ define circuits, although other circuits may exist in the arrangement.

In the above context, define a \emph{$C$-filtration} to be a triple $(f,g,H_t)$ where $f,g:C \to R$ and $H_t \subseteq R$ are finite for $t \in \mathbb{Z}$ satisfying the following properties
\begin{enumerate}
\item $f(I) \in I$, $g(I) \in I$, and $f(I) \neq g(I)$ for all $I \in C$
\item $f$ and $g$ are bijections
\item For each $r \in R$ there exist integers $t_1 \leq t_2$ such that $r \in H_t$ if and only if $t_1 < t \leq t_2$.
\item For each $t \in \mathbb{Z}$, $g\circ f^{-1}$ restricts to a bijection from $H_t \setminus H_{t+1}$ to $H_{t+1} \setminus H_t$.
\item For $I \in C$ if $f(I) \in H_t \setminus H_{t+1}$ (equiv. $g(I) \in H_{t+1} \setminus H_t$) then $I \subseteq H_t \cup H_{t+1}$.
%\item For each $I \in C$, there is a unique $t = t(I)$ such that $I \subseteq H_t \cup H_{t+1}$.
%\item $r \in H_{t+1} \setminus H_t$ if and only if $t=t(g^{-1}(r))$.  
%\item $r \in H_t \setminus H_{t+1}$ if and only if $t = t(f^{-1}(r))$.
%\item $t(g^{-1}(r)) < t(f^{-1}(r))$.
\item There is a linear order on $H_{t+1} \setminus H_t$ such that if $r,r' \in H_{t+1} \setminus H_t$ and $r' \in g^{-1}(r)$ then $r' \leq r$.
\item There is a linear order on $H_t \setminus H_{t+1}$ such that if $r,r' \in H_{t} \setminus H_{t+1}$ and $r' \in f^{-1}(r)$ then $r' \leq r$.
\end{enumerate}
Note property (4) implies that all $H_t$ have the same size.

\begin{prop} \label{propfilter}
Let $(f,g,H_t)$ be a filtration for $C$.  Then the span of any $C$-arrangement has dimension at most $|H_0|$.  %Suppose each element of $C$ has size at most $D+2$.    %Then a $C$-arrangement exists in $\mathbb{R}^{D+1}$ if and only if $D < |H_0|$
\end{prop}
\begin{proof}
%It follows from the properties of filtrations that $r \in H_t$ if and only if 
%\begin{displaymath}
%t(g^{-1}(r)) < t \leq t(f^{-1}(t)).  
%\end{displaymath}
We describe a procedure to generate any $C$-arrangement.  It follows from property (3) that the sets
\begin{displaymath}
\ldots, H_{-2} \setminus H_{-1}, H_{-1} \setminus H_0, H_0, H_1 \setminus H_0, H_2 \setminus H_1 \ldots 
\end{displaymath}
partition $R$.  Choose the $v_r$ for $r \in H_0$ arbitrarily.  Next construct the $v_r$ with $r \in H_1 \setminus H_0$ according to the linear order specified on that set.  Given such $r$, let  $I=g^{-1}(r)$.  By property (5), $I \subseteq H_0 \cup H_1$.  By definition of the order, if $r' \in I \setminus \{r\}$ is not in $H_0$ then $r' < r$ so $v_{r'}$ has already been constructed.  Choose $v_r$ generically on the span of the $v_{r'}$ with $r' \in I \setminus \{r\}$.

Iterate this process to construct $v_r$ with $r \in H_{t+1} \setminus H_t$ for all $t \geq 0$.  The same process can be run backwards to construct each $v_r$ with $r \in H_{t} \setminus H_{t+1}$ for all $t < 0$.  It is clear that any $C$-arrangement can be produced in this manner.  Conversely, the result is always a $C$-arrangement.  Indeed, let $I \in C$ be given.  Then $g(I) \in H_{t+1} \setminus H_t$ and $f(I) \in H_t \setminus H_{t+1}$ for some $t$.  If $t \geq 0$ then $v_{g(I)}$ was constructed to ensure that the $v_r$ with $r \in I$ define a circuit.  If instead $t < 0$ then $v_{f(I)}$ was constructed for this purpose.  So, all the circuit conditions are satisfied.

When carrying out the procedure, each $v_r$ for $r \notin H_0$ is constructed on the span of previous points.  Thus the span of the whole arrangement equals the span of the $v_r$ with $r \in H_0$ which has dimension at most $|H_0|$.
\end{proof}

We now fix a $Y$-pin $S$ and specialize $C$ to the set of triples and quadruples of $r \in R$ occurring in the (P1), (P2), and (L3) relations, i.e.
\begin{align*}
C = &\{\{r+a,r+b,r+c\}: -a_2 < r_2 \leq m-c_2\} \\
&\cup \{\{r+b,r+c,r+d\}: -b_2 < r_2 \leq m-d_2\} \\
&\cup \{\{r+a+c,r+a+d,r+b+c,r+b+d\}: -a_2-c_2 < r_2 \leq m-b_2-d_2\}
\end{align*}
where $m=d_2-a_2$.

\begin{prop} \label{propH}
The set $C$ above has a filtration $(f,g,H_t)$ with $|H_t| = D(S)+1$.
\end{prop}

The proof is rather technical and is handled in three cases.

\subsection{The long diagonal case}
Assume that $a,b,c,d$ are vertices of a convex quadrilateral with $\overline{ad}$ being a diagonal.  Define functions $f,g:C \to R$ as in Table \ref{tabfgdiag}.  Let $\phi: \mathbb{Z}^2 \to \mathbb{Z}$ be a linear functional satisfying $\phi(b) < \phi(a) = \phi(d) < \phi(c)$.  Define $H_t$ for $t \in \mathbb{Z}$ by
\begin{displaymath}
H_t = \{r : 0 < r_2 \leq d_2-a_2 \text{ and } t+\phi(b) \leq \phi(r) < t+\phi(c)\}.
\end{displaymath}
An example of this region appears in Figure \ref{figdiag}.  We assume in this and future examples that $a_2=0$ so that the region begins in the row above $a$.

\begin{table}
\begin{tabular}{c|c|c|c}
Type of $I \in C$ & $f(I)$ & $I \setminus \{f(I),g(I)\}$ & $g(I)$\\
\hline
(P1) & $r+b$ & $r+a$ & $r+c$  \\
\hline
(P2) & $r+b$ & $r+d$ & $r+c$ \\
\hline
(L3) & $r+b+d$ & $r+a+d,r+b+c$ & $r+a+c$ \\
\end{tabular}
\caption{The functions $f,g:C \to R$ in the long diagonal case}
\label{tabfgdiag}
\end{table}

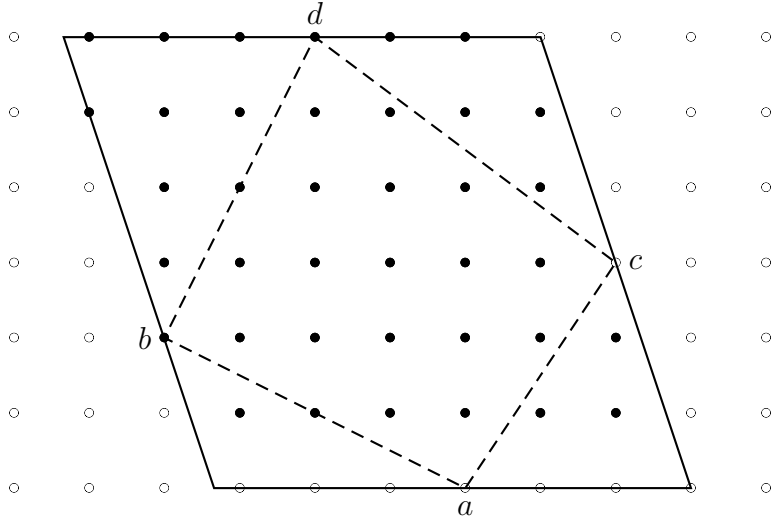
\begin{figure}
\begin{pspicture}(-1,0.5)(11,7.5)
%\psdots(6,1)(2,3)(8,4)(4,7)
\multirput(0,1)(0,1){7}{\multirput(0,0)(1,0){11}{\psdots[dotstyle=o](0,0)}}
\multirput(3,2)(1,0){6}{\psdots(0,0)}
\multirput(2,3)(1,0){7}{\psdots(0,0)}
\multirput(2,4)(1,0){6}{\psdots(0,0)}
\multirput(2,5)(1,0){6}{\psdots(0,0)}
\multirput(1,6)(1,0){7}{\psdots(0,0)}
\multirput(1,7)(1,0){6}{\psdots(0,0)}
\pspolygon(2.66,1)(9,1)(7,7)(.66,7)
\pspolygon[linestyle=dashed](6,1)(2,3)(4,7)(8,4)
\uput[d](6,1){$a$}
\uput[l](2,3){$b$}
\uput[r](8,4){$c$}
\uput[u](4,7){$d$}
\end{pspicture}
\caption{The region $H_0$ in the long diagonal case.}
\label{figdiag}
\end{figure}

We now check that $(f,g,H_t)$ gives a filtration.
\begin{enumerate}
\item Clear
\item Let $r \in H$.  Then $r = f(I)$ for some $I$ of type
\begin{displaymath}
\begin{cases}
\textrm{(P1)}, & \textrm{if } b_2-a_2 < r \leq m - (c_2-b_2) \\
\textrm{(P2)}, & \textrm{if } 0 < r \leq m + b_2 - d_2 = b_2-a_2\\ 
\textrm{(L3)}, & \textrm{if } m-(c_2-b_2) = b_2+d_2-a_2-c_2 < r \leq m \\
\end{cases}
\end{displaymath}
(recall $m = d_2-a_2$).  These intervals cover $\{1,2,\ldots m\}$ so $f$ is a bijection.  Similarly $r=g(I)$ for some $r$ of type 
\begin{displaymath}
\begin{cases}
\textrm{(P1)}, & \textrm{if } c_2-a_2 < r \leq m \\
\textrm{(P2)}, & \textrm{if } c_2-b_2 < r \leq c_2-a_2\\ 
\textrm{(L3)}, & \textrm{if } 0 < r \leq c_2 - b_2\\
\end{cases}
\end{displaymath}
so $g$ is a bijection.
\item Take $t_1 = \phi(r) - \phi(c)$ and $t_2 = \phi(r) - \phi(b)$.
\item We have
\begin{align*}
H_{t+1} \setminus H_t &= \{r \in R : \phi(r)=t+\phi(c)\} \\
H_t \setminus H_{t+1} &= \{r \in R : \phi(r)=t+\phi(b)\} \\
\end{align*}
By Table \ref{tabfgdiag}, $g\circ f^{-1}(r)$ equals either $r+c-b$ or $r+a+c-b-d$.  In either case
\begin{displaymath}
\phi(g\circ f^{-1}(r)) = \phi(r) + \phi(c)-\phi(b).
\end{displaymath}
and the result follows.
\item Let $I \in C$ with $f(I) \in H_t \setminus H_{t+1}$ and $g(I) \in H_{t+1} \setminus H_t$.  It is clear from Table \ref{tabfgdiag} that
\begin{displaymath}
\phi(f(I)) \leq \phi(r) \leq \phi(g(I))
\end{displaymath}
for all $r \in I$, with equality only possible if $r=f(I)$ or $r=g(I)$.  Therefore $r \in H_t \cap H_{t+1}$ whenever $r \in I \setminus \{f(I),g(I)\}$.
\item By the previous, if $r' \in g^{-1}(r)$ then $\phi(r') \leq \phi(r)$ with equality if and only if $r'=r$.  If in addition $r,r' \in H_{t+1} \setminus H_t$ then $\phi(r)=\phi(r')$ so $r'=r$.  Hence, any order on $H_{t+1} \setminus H_t$ will work.
\item Similar to (6).
\end{enumerate}

\subsection{The triangle case}
Now suppose the convex hull of $S$ is a triangle with one of the four points lying in the interior.  Because $a_2 \leq b_2 < c_2 \leq d_2$, $a$ and $d$ are always vertices.  We assume that $c$ is the third vertex and that $b$ is in the interior.  The other case can be handled via the symmetry $(i,j) \mapsto (i,-j)$ of the $\mathbb{Z}^2$ lattice.

Define $f,g:C \to R$ as in Table \ref{tabfgtri}.  Let $\phi:\mathbb{Z}^2 \to \mathbb{Z}$ be linear and satisfy $\phi(c) < \phi(a) = \phi(b) < \phi(d)$.  Define $H_t$ as a union of two parallelogram shaped regions
\begin{align*}
H_t = &\{r : 0 < r_2 \leq c_2-a_2 \textrm{ and } \phi(c) \leq \phi(r)-t < \phi(d)\} \\
&\cup \{r: c_2-a_2 < r_2 \leq m \textrm{ and } \phi(d)-(\phi(a)-\phi(c)) \leq \phi(r)-t < \phi(d)\}.
\end{align*}
An example of this region appears in Figure \ref{figtri}.

\begin{table}
\begin{tabular}{c|c|c|c}
Type of $I \in C$ & $f(I)$ & $I \setminus \{f(I),g(I)\}$ & $g(I)$\\
\hline
(P1) & $r+c$ & $r+b$ & $r+a$  \\
\hline
(P2) & $r+c$ & $r+b$ & $r+d$ \\
\hline
(L3) & $r+a+c$ & $r+b+d,r+b+c$ & $r+a+d$ \\
\end{tabular}
\caption{The functions $f,g:C \to R$ in the triangle case (assuming $b$ is the interior point)}
\label{tabfgtri}
\end{table}

\begin{figure}
\begin{pspicture}(0,0.5)(13,9.5)
\multirput(1,1)(0,1){9}{\multirput(0,0)(1,0){13}{\psdots[dotstyle=o](0,0)}}
\multirput(4,2)(1,0){8}{\psdots(0,0)}
\multirput(3,3)(1,0){8}{\psdots(0,0)}
\multirput(3,4)(1,0){8}{\psdots(0,0)}
\multirput(2,5)(1,0){8}{\psdots(0,0)}
\multirput(8,6)(1,0){2}{\psdots(0,0)}
\multirput(7,7)(1,0){2}{\psdots(0,0)}
\multirput(7,8)(1,0){2}{\psdots(0,0)}
\multirput(6,9)(1,0){2}{\psdots(0,0)}
\pspolygon(4,1)(2,5)(8,5)(6,9)(8,9)(12,1)
\pspolygon[linestyle=dashed](6,1)(2,5)(8,9)
\uput[d](6,1){$a$}
\uput[r](5,3){$b$}
\uput[l](2,5){$c$}
\uput[u](8,9){$d$}
\end{pspicture}
\caption{The region $H_0$ in the triangle case.}
\label{figtri}
\end{figure}
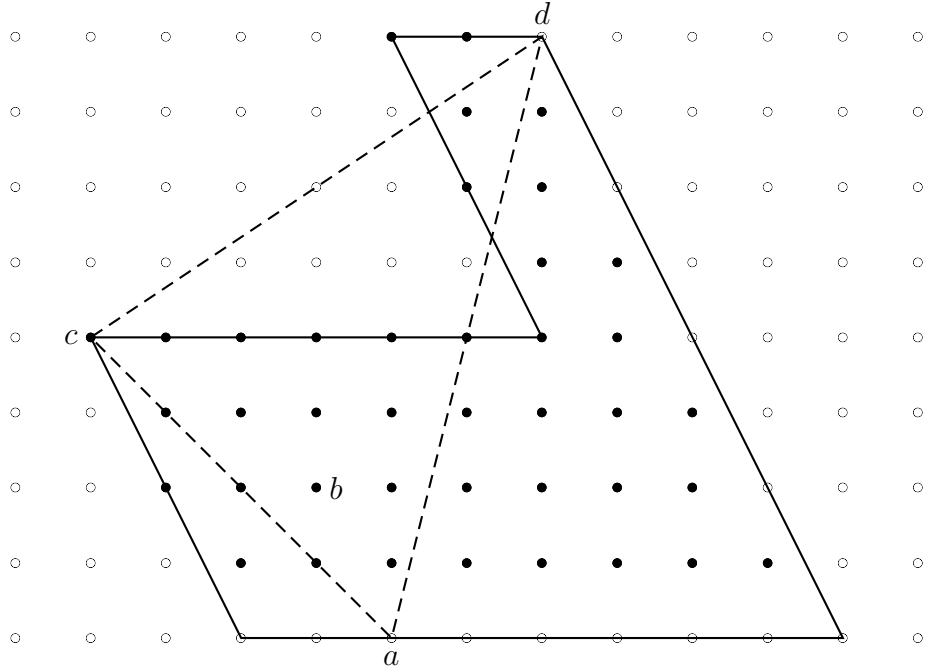

As before, conditions (1) and (3) are clear.  Condition (2) also holds, and it will be useful to keep track of how $f^{-1}$ and $g^{-1}$ behave:
\begin{align*}
\textrm{type of $f^{-1}(r)$} &= 
\begin{cases}
\textrm{(P1)}, & \textrm{if } c_2-a_2 < r_2 \leq m \\
\textrm{(P2)}, & \textrm{if } c_2-b_2 < r_2 \leq c_2-a_2\\ 
\textrm{(L3)}, & \textrm{if } 0 < r_2 \leq c_2 - b_2\\
\end{cases} \\
\textrm{type of $g^{-1}(r)$} &= 
\begin{cases}
\textrm{(P1)}, & \textrm{if } 0 < r_2 \leq d_2-c_2 \\
\textrm{(P2)}, & \textrm{if } d_2-b_2 < r_2 \leq m\\ 
\textrm{(L3)}, & \textrm{if } d_2-c_2 < r_2 \leq d_2 - b_2\\
\end{cases}
\end{align*}
Therefore
\begin{displaymath}
g\circ f^{-1}(r) = 
\begin{cases}
r+d-c, & \textrm{if } 0 < r_2 \leq c_2-a_2 \\
r+a-c, & \textrm{if } c_2-a_2 < r_2 \leq m \\
\end{cases}
\end{displaymath}
which implies condition (4).

Let $I \in C$ and suppose $r = g(I) \in H_{t+1} \setminus H_t$ (i.e. $\phi(r) = t+\phi(d)$).  We already know $f(I) \in H_t \setminus H_{t+1}$ so it suffices to take $r' \in I \setminus \{f(I),g(I)\}$.  If $0 < r_2 \leq d_2-c_2$ then $I$ has type (P1) and $r' = r+b-a$.  Hence $\phi(r')=\phi(r)$ so $r' \in H_{t+1} \setminus H_t$ as well.  If $d_2-c_2 < r_2 \leq d_2-b_2$ then $I$ has type (P3) so either $r' = r+b-a$ or $r'=r+b+c-a-d$.  In the first case $r' \in H_{t+1} \setminus H_t$ as before.  In the second case 
\begin{displaymath}
\phi(r') = t + \phi(b)+\phi(c)-\phi(a) = t + \phi(c)
\end{displaymath}
and 
\begin{displaymath}
r_2' = r_2 + b_2 + c_2 - a_2 - d_2 \leq c_2-a_2 
\end{displaymath}
so $r' \in H_t \setminus H_{t+1}$.  Finally, suppose $d_2-b_2 < r_2 \leq m$.  Then $I$ has type (P2) and $r' = r+b-d$.  Therefore $\phi(r')-t = \phi(b)$ so $\phi(c) < \phi(r')-t < \phi(d)$.  Also $r_2' = r_2+b_2-d_2 \leq b_2-a_2 \leq c_2-a_2$ so $r' \in H_t \cap H_{t+1}$.  This completes the verification of (5).

The above analysis shows that $r,r' \in H_{t+1} \setminus H_t$ distinct with $r' \in g^{-1}(r)$ can only occur if $r' = r+(b-a)$, and similarly for $H_t \setminus H_{t+1}$.  On both sets, order the elements from top to bottom (i.e. by decreasing $r_2$ value) to get (6) and (7).

\subsection{The long side case}
Lastly, suppose the convex hull of $S$ is again a quadrilateral but that now $\overline{ad}$ is a side.  Define $f,g:C \to R$ as in Table \ref{tabfgside}.  Choose a linear function $\phi:\mathbb{Z}^2 \to \mathbb{Z}$ so that $\phi(b) = \phi(c) < \min(\phi(a),\phi(d))$ (we do not specify the relative sizes of $\phi(a)$ and $\phi(d)$).  In this case, $H_t$ is a union of three parallelograms
\begin{align*}
H_t &= \{r : 0 < r_2 \leq b_2-a_2 \textrm{ and } \phi(a) \leq \phi(r)-t < \phi(a)+\phi(d)-\phi(b)\} \\
&\cup \{r : b_2-a_2 < r_2 \leq c_2-a_2 \textrm{ and } \phi(b) \leq \phi(r)-t < \phi(a)+\phi(d)-\phi(b)\} \\
&\cup \{r : c_2-a_2 < r_2 \leq m \textrm{ and } \phi(d) \leq \phi(r)-t < \phi(a)+\phi(d)-\phi(b)\}.
\end{align*}
An example of this region appears in Figure \ref{figside}.

\begin{table}
\begin{tabular}{c|c|c|c}
Type of $I \in C$ & $f(I)$ & $I \setminus \{f(I),g(I)\}$ & $g(I)$\\
\hline
(P1) & $r+c$ & $r+b$ & $r+a$  \\
\hline
(P2) & $r+b$ & $r+c$ & $r+d$ \\
\hline
(L3) & $r+b+c$ & $r+a+c,r+b+d$ & $r+a+d$ \\
\end{tabular}
\caption{The functions $f,g:C \to R$ in the long side case}
\label{tabfgside}
\end{table}

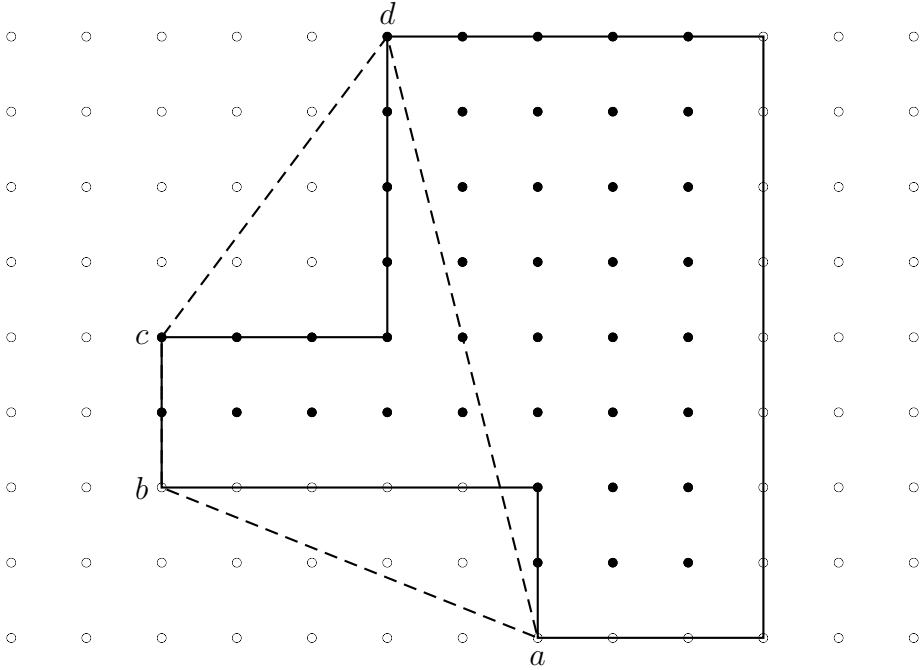
\begin{figure}
\begin{pspicture}(-1,0.5)(12,9.5)
\multirput(-1,1)(0,1){9}{\multirput(0,0)(1,0){13}{\psdots[dotstyle=o](0,0)}}
\multirput(6,2)(1,0){3}{\psdots(0,0)}
\multirput(6,3)(1,0){3}{\psdots(0,0)}
\multirput(1,4)(1,0){8}{\psdots(0,0)}
\multirput(1,5)(1,0){8}{\psdots(0,0)}
\multirput(4,6)(1,0){5}{\psdots(0,0)}
\multirput(4,7)(1,0){5}{\psdots(0,0)}
\multirput(4,8)(1,0){5}{\psdots(0,0)}
\multirput(4,9)(1,0){5}{\psdots(0,0)}
\pspolygon(6,1)(6,3)(1,3)(1,5)(4,5)(4,9)(9,9)(9,1)
\pspolygon[linestyle=dashed](6,1)(1,3)(1,5)(4,9)
\uput[d](6,1){$a$}
\uput[l](1,3){$b$}
\uput[l](1,5){$c$}
\uput[u](4,9){$d$}
\end{pspicture}
\caption{The region $H_0$ in the long side case.}
\label{figside}
\end{figure}

As usual, conditions (1) and (3) in the definition of filtration clearly hold.  One can check (2) and more specifically
\begin{align*}
\textrm{type of $f^{-1}(r)$} &= 
\begin{cases}
\textrm{(P1)}, & \textrm{if } c_2-a_2 < r_2 \leq m \\
\textrm{(P2)}, & \textrm{if } 0 < r_2 \leq b_2-a_2\\ 
\textrm{(L3)}, & \textrm{if } b_2-a_2 < r_2 \leq c_2 - a_2\\
\end{cases} \\
\textrm{type of $g^{-1}(r)$} &= 
\begin{cases}
\textrm{(P1)}, & \textrm{if } 0 < r_2 \leq d_2-c_2 \\
\textrm{(P2)}, & \textrm{if } d_2-b_2 < r_2 \leq m\\ 
\textrm{(L3)}, & \textrm{if } d_2-c_2 < r_2 \leq d_2 - b_2\\
\end{cases}
\end{align*}
Following these conditions and Table \ref{tabfgside} shows
\begin{displaymath}
g\circ f^{-1}(r) = 
\begin{cases}
r+d-b, & \textrm{if } 0 < r_2 \leq b_2-a_2 \\
r+a+d-b-c, & \textrm{if } b_2-a_2 < r_2 \leq c_2-a_2 \\
r+a-c, & \textrm{if } c_2-a_2 < r_2 \leq m \\
\end{cases}
\end{displaymath}
which confirms (4).

Let $I \in C$ and suppose $r = g(I) \in H_{t+1} \setminus H_t$.  Then 
\begin{displaymath}
\phi(r) = t+\phi(a) + \phi(d) - \phi(b).
\end{displaymath}
We know $g(I) \in H_t \setminus H_{t+1}$.  Let $r' \in I \setminus \{f(I),g(I)\}$.  If $0 < r_2 \leq d_2-c_2$ then $I$ has type (P1) and $r' = r+b-a$.  Therefore $\phi(r') = t+\phi(d)$ and
\begin{displaymath}
\phi(b) < \phi(d) = \phi(r')-t < \phi(a)+\phi(d)-\phi(b).
\end{displaymath}
Moreover, $r_2' = r_2 + b_2-a_2 > b_2-a_2$.  So $r' \in H_t$ always holds.  In the case that $r_2' \geq c_2-a_2$ then we have $r' \in H_t \setminus H_{t+1}$.  Note for later that $r' = f(I) - (c-b)$.

If $d_2-c_2 < r_2 \leq d_2-b_2$ then $I$ has type (L3).  The two cases $r'= r+c-d$ and $r'=r+b-a$ are symmetric, so consider just the first one.  Then $\phi(r') = t+\phi(a)$ and 
\begin{displaymath}
r_2' = r_2 + c_2 - d_2 \leq c_2-b_2 \leq c_2-a_2.
\end{displaymath}
So $r_2' \in H_t$ and $r_2' \in H_t \setminus H_{t+1}$ if and only if $r_2' \leq b_2-a_2$.

Lastly suppose $d_2-b_2 < r_2 \leq m$.  Then $I$ has type (P2) and $r' = r+c-d$.  Then $\phi(r') = t+\phi(a)$ so
\begin{displaymath}
\phi(b) < \phi(a) = \phi(r')-t < \phi(a)+\phi(d)-\phi(b).
\end{displaymath}
Also $r_2' = r_2 + c_2 - d_2 \leq m+ c_2-d_2 = c_2-a_2$ so $r' \in H_t$.  It is possible $r' \in H_t \setminus H_{t+1}$ and in this case $r' = f(I) + c-b$.

By the above, it is not possible to have distinct $r,r' \in H_{t+1} \setminus H_t$ with $r' \in g^{-1}(r)$.  So there is nothing to check for condition (6).  However, it is possible to have distinct $r,r' \in H_t \setminus H_{t+1}$ with $r' \in f^{-1}(r)$.  In fact, this occurs under any of the following circumstances
\begin{itemize}
\item $r' = r+c-b$ with $r_2 \leq r_2' \leq b_2-a_2$
\item $r' = r+a-b$ with $r_2' \leq b_2-a_2 < r_2 \leq c_2-a_2$
\item $r' = r+d-c$ with $b_2-a_2 < r_2 \leq r_2' \leq m$
\item $r' = r-(c-b)$ with $c_2-a_2 < r_2' \leq r_2 \leq m$.
\end{itemize}
Define an order on $H_t \setminus H_{t+1}$ as follows.  Order the $r$ with $r_2 \leq b_2-a_2$ from bottom to top.  After that, order the $r$ with $r_2 > c_2 - a_2$ from top to bottom.  
The remaining $r$ (with $b_2-a_2 < r_2 \leq c_2-a_2$) are defined to be the largest (and can be ordered arbitrarily amongst themselves).  Any such order proves (7).

\subsection{Proof of $|H_t| = D(S)+1$}

We give a case-uniform proof of $|H_t| = D(S)+1$. First, let us make time $t$ continuous in the following way. Define $\H_t$ be the region of $\bar R = \mathbb R \times (0,m]$ 
cut by the same inequalities as $H_t$, except with coordinates of vector $r$ interpreted continuously.  
Denote $\# \H_t$ the number of integer points inside $\H_t$. It is easy to see that $\#\H_t = |H_t|$. 

First, we argue that $\#\H_t$ is independent of $t$. This follows from part (4) of the definition of $C$-filtration. Indeed, at any time $t$ when the left boundary of $\H_t$
is about to lose $|H_t \setminus H_{t+1}|$ of integer points, the right boundary of $\H_t$ is about to gain $|H_{t+1} \setminus H_t|$ integer points. Since those two sets are in bijection,
the number $\#\H_t$ never changes. We will thus also denote the same number as $\#\H$.

\begin{lem}
The area of $\H_t$ equals $D(S)+1$.
\end{lem}
\begin{proof}
Recall that $D(S)+1$ equals twice the area of the convex hull of $S$.  In Figures \ref{figdiag}, \ref{figtri}, and \ref{figside} the region $\H_t$ is bounded by a solid path while the convex hull of $S$ is bounded by a dashed path.  The result that the former has twice the area of the latter can be proven with elementary geometry.  For example, Figure \ref{figareas} shows a subdivision of both regions in the triangle case into three pieces.  It is clear that each piece of $\H_t$ has twice the area of the corresponding piece of the triangle.
\end{proof}

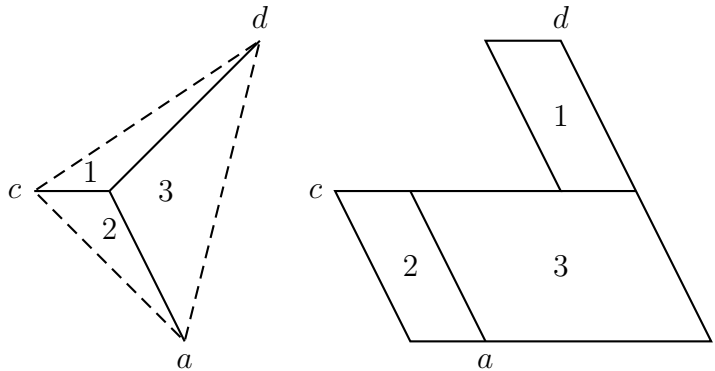
\begin{figure}
\begin{pspicture}(10,5)
\pspolygon[linestyle=dashed](3,.5)(1,2.5)(4,4.5)
\psline(3,.5)(2,2.5)
\psline(1,2.5)(2,2.5)
\psline(4,4.5)(2,2.5)
\uput[d](3,.5){$a$}
%\uput[r](2.5,1.5){$b$}
\uput[l](1,2.5){$c$}
\uput[u](4,4.5){$d$}
\rput(1.75,2.75){1}
\rput(2,2){2}
\rput(2.75,2.5){3}

\rput(4,0){
\pspolygon(2,.5)(1,2.5)(4,2.5)(3,4.5)(4,4.5)(6,.5)
\psline(3,.5)(2,2.5)
\psline(4,2.5)(5,2.5)
\uput[d](3,.5){$a$}
%\uput[r](2.5,1.5){$b$}
\uput[l](1,2.5){$c$}
\uput[u](4,4.5){$d$}
\rput(4,3.5){1}
\rput(2,1.5){2}
\rput(4,1.5){3}
}
\end{pspicture}
\caption{The regions of Figure \ref{figtri} separated and subdivided to illustrate that the larger one has twice the area of the smaller one.}
\label{figareas}

\end{figure}

It remains to show that $\#\H$ equals the area of $\H_{t}$.  Let us compute the same integral in two ways.
On the one hand, it is clear that 
$$\int_{t=0}^{N} \#\H_{t} d t  = \int_{t=0}^{N} \#\H d t = N \#\H.$$
On the other hand, let us swap the order of summation and integration: denoting $\bullet$ the integer points inside regions, we have:
$$\int_{t=0}^{N} \#\H_{t} d t  = \int_{t=0}^{N} (\sum_{\bullet \in \H_{t}} 1) d t = \sum_{\bullet \in \cup_{0 \leq t \leq N} \H_t} \int_{\{t \mid 0 \leq t \leq N, \; \bullet \in \H_t\}} 1 d t.$$ 
For each integer point $\bullet \in \cup_{0 \leq t \leq N} \H_t$ denote $\ell(\bullet)$ the length of the intersection of the horizontal line through $\bullet$ with $\H_t$. This value clearly does not depend on $t$.
For all but $o(N)$ integer points $\bullet \in \cup_{0 \leq t \leq N} \H_t$ we have $\int_{\{t \mid 0 \leq t \leq N, \; \bullet \in H_t\}} 1 d t = \ell(\bullet)$. Therefore 
$$\sum_{\bullet \in \cup_{0 \leq t \leq N} \H_t} \int_{\{t \mid 0 \leq t \leq N, \; \bullet \in H_t\}} 1 d t = o(N) + \sum_{\bullet \in \cup_{0 \leq t \leq N} \H_t} \ell(\bullet).$$
Each of the $\H_t$ can be cut by horizontal lines into one, two or three parallelograms. Let $h_i$ be their heights and let $\ell_i$ be their horizontal lengths. It is clear that for any point 
$\bullet$ we have $\ell(\bullet) = \ell_i$ for the appropriate parallelogram number $i$. It is also clear that $\sum_i h_i \ell_i$ is the area of $\H_t$, equal to $D(S)+1$ as explained above.
Now, each integer horizontal line in $\bar R$ contains $N + o(N)$ integer points. Furthermore, the number of such lines that have $\ell(\bullet) = \ell_i$ for points $\bullet$ on the line is exactly $h_i$.
This means that 
$$o(N) + \sum_{\bullet \in \cup_{0 \leq t \leq N} \H_t} \ell(\bullet) = o(N) + \sum_i N h_i \ell_i = o(N) + N (D(S)+1).$$ Equating it with our first way to evaluate the integral we see that 
$$\#\H = o(N)/N + (D(S)+1),$$ where passing to the limit $N \rightarrow \infty$ we obtain the desired result. 

\subsection{Proof of Theorem \ref{thmmain}}
We are now ready to prove Theorem \ref{thmmain}.  Let $S=\{a,b,c,d\}$ be a $Y$-pin and $D \geq 2$.  First suppose $D \leq D(S)$.  To prove that
\begin{displaymath}
T_{D,S}:X_{D,S} \to X_{D,S}
\end{displaymath}
is an $S$-map, it remains to show that generic elements of $X_{D,S}$ span $\mathbb{RP}^D$.  Let $(f,g,H_t)$ be the filtration guaranteed by Proposition \ref{propH}.  Then $|H_0| = D(S)+1 > D$.  By the proof of Proposition \ref{propfilter}, an element $A \in X_{D,S}$ can have arbitrary vertices $A_r$ for $r \in H_0$.  Choosing these $D(S)+1$ points generically, they will span all of $\mathbb{RP}^D$.

On the other hand, suppose $D > D(S)$.  Any $Y$-mesh $(P_{i,j})_{i,j \in \mathbb{Z}}$ would satisfy all (P1), (P2), and (L3) conditions, so its restriction to $R$ would be an element $A \in X_{D, S}$.  By the proof of Proposition \ref{propfilter}, all vertices of $A$ lie in the span of the $A_r$ with $r \in H_0$ which has dimension at most $|H_0|-1 = D(S) < D$.  So all vertices of $A$ lie in some proper subspace.  As explained in Section \ref{secdefs}, all other points of the $Y$-mesh can be computed recursively via \eqref{eqF2} and \eqref{eqF2inv}.  Hence all $P_{i,j}$ lie in the subspace which violates the definition of a $Y$-mesh.

\section{Decreasing the order in two dimensions} \label{secorder}
We have focused on constructing $S$-maps of given type $S$ and dimension $D$, when possible.  Rather than also determine the possible orders $m$ for the maps, we have simply fixed $m=d_2-a_2$, which works.  It is natural to look for the smallest order possible, but this value is not known in general.  In this section, we address the case $D=2$ proving that the minimal order is $m=\max(c_2-a_2,d_2-b_2)$.

\begin{ex} \label{exF2cont}
Recall Example \ref{exF2} in which $S=\{(-1,1), (1,2), (0,3), (0,4)\}$ and $D=2$.  Let $A_i=P_{i,1}$, $B_i=P_{i,2}$, and $C_i=P_{i,3}$ for all $i \in \mathbb{Z}$.  For $(A,B,C) \in X_{2,S}$, the (L1) relation says $A_{i-1}, B_{i+1}, C_i$ are collinear and the (L2) relation says $A_{i+1}, B_i, C_i$ are collinear.  Therefore, $A$ and $B$ can be arbitrary, and for a generic choice of such, $C$ is uniquely determined by
\begin{equation} \label{eqG2ex}
C_i = \meet{\join{A_{i-1}}{B_{i+1}}}{\join{A_{i+1}}{B_i}}.
\end{equation}
Hence $T_{2,S}$ has a smaller order analogue $F: \mathcal{U}_{2,2} \to \mathcal{U}_{2,2}$ given by $F(A,B) = (B,C)$ for $C$ as in \eqref{eqG2ex}.  
\end{ex}

The general construction depends on which of $c_2-a_2$ and $d_2-b_2$ is larger.  We present the case $d_2-b_2 \geq c_2-a_2$.  For the remainder of the section, let $m'=d_2-b_2$.  Define a new space $X' \subseteq \mathcal{U}_{2,m'}$ by (L1) relations alone, that is $A \in X'$ if and only if (L1) holds for all $r$ with $-a_2 < r_2 \leq m'-c_2$.  

Now for $A \in X'$ define $F(A) = (A^{(2)},\ldots, A^{(m'+1)})$ where $A^{(m'+1)}$  is defined by
\begin{equation} \label{eqG2}
P_{r+c+d} = \meet{\join{P_{r+2c}}{P_{r+b+c}}}{\join{P_{r+a+d}}{{P_{r+b+d}}}}
\end{equation}
for $r \in \mathbb{Z}^2$ with $r_2=m'+1-c_2-d_2$.  Define $G(A) = (A^{(0)},\ldots, A^{(m'-1)})$ where $A^{(0)}$ is defined by 
\begin{equation} \label{eqG2inv}
P_{r+a+b} = \meet{\join{P_{r+a+c}}{P_{r+a+d}}}{\join{P_{r+b+c}}{{P_{r+2b}}}}.
\end{equation}

\begin{prop}
For generic $A \in X'$, $F(A)$ and $G(A)$ are also in $X'$, and $F,G:X' \to X'$ are inverse as rational maps.  Moreover, $F$ is an $S$-map of order $m'$.  
\end{prop}

\begin{proof}
The proof of the first statement is quite similar to that of Proposition \ref{propT2S}.  To get that $F$ is an $S$-map, it is necessary to show that upon iteration that (L2) relations hold.  By \eqref{eqG2}, we get $P_{r+b+c}, P_{r+2c}, P_{r+c+d}$ are collinear, which up to shifts is the (L2) relation.
\end{proof}

\begin{ex} \label{exkangaroo}
Let $S=\{(1,0),(1,1),(0,2),(2,4)\}$.  Then there is an order $\max(3-1,5-2)=3$ map $F:X' \to X'$ when $D=2$.  The map is pictured in Figure \ref{figexorder}.  The space $X' \subseteq \mathcal{U}_{2,3}$ is defined by the relations
\begin{displaymath}
P_{i,1}, P_{i,2}, P_{i-1,3}
\end{displaymath}
collinear, and $F$ is defined by
\begin{displaymath}
P_{i,4} = \meet{\join{P_{i-1,1}}{P_{i-2,2}}}{\join{P_{i+1,2}}{P_{i+1,3}}}.
\end{displaymath}
\end{ex}

\begin{figure}
\begin{pspicture}(0,.5)(14,5.5)
\rput(1.5,0){
\pnode(0,1){a}
\pnode(0,2){b}
\pnode(-1,3){c}
\pnode(1,5){d}
\psdots(a)(b)(c)(d)
\uput[d](a){$a$}
\uput[d](b){$b$}
\uput[d](c){$c$}
\uput[d](d){$d$}
}
\rput(4,1){
\multirput(0,0)(1,0){7}{\multirput(0,0)(0,1){3}{\psdots(1,1)}}
\psaxes[axesstyle=none,ticks=none,showorigin=false,labelsep=0]{->}(7,3)
\uput[d](.3,0){$i$}
\uput[l](0,.3){$j$}
\psline(3,1)(2,2)
\psline(5,2)(5,3)
\psdots[dotstyle=asterisk](4,4)
\rput(7,0){
\psbezier(0   ,.8 )(-.33,.8 )(-.03,1.61)(-.18,1.91)
\psbezier(-.18,1.91)(-.33,2.21)(-1.38,2.91)(-1.14,3.14)
\psbezier(-1.14,3.14)(-.91,3.38)(.03,2.39)(.18,2.09)
\psbezier(.18,2.09)(.33,1.79)(.33,.8 )(0   ,.8 )
}
}
\end{pspicture}
\caption{The $S$-map for $S=\{(1,0),(1,1),(0,2),(2,4)\}$.}
\label{figexorder}
\end{figure}
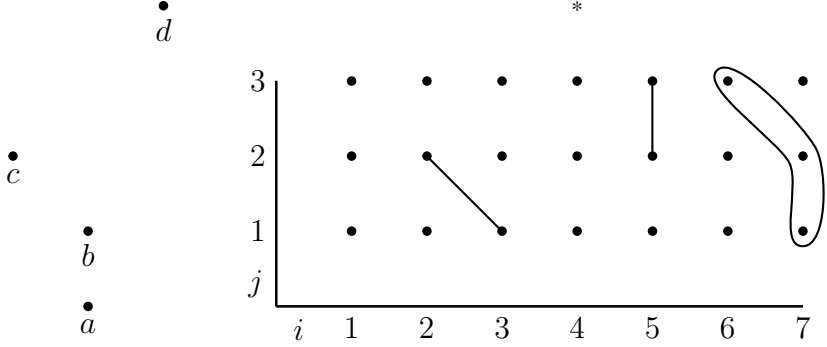

Nothing is lost going from the original $S$-map $T_{2,S}: X_{2,S} \to X_{2,S}$ of order $m$ to the newly defined map $F: X' \to X'$ which has order $m'$.  More precisely, an element $A \in X_{2,S}$ can be determined by its first $m'$ rows by applying $F$ a sufficient number ($m-m'$) of times.  For general $D$ with $2 \leq D \leq D(S)$, let $m_D(S)$ denote the smallest order to which the map $T_{D,S}$ can be reduced in this way.

\begin{prop}
For any $S$ with $D(S) \geq 2$, $m_2(S) = \max(c_2-a_2, d_2-b_2)$.
\end{prop}

\begin{proof}
We continue without loss of generality in the case $m' = d_2-b_2 \geq c_2-a_2$.  The map $F: X' \to X'$ is an order $m'$ reduction of $T_{2,S}$.  Let $A=(A^{(1)}, \ldots, A^{(m')}) \in X'$ and consider some $P_r$ where $r_2=m'$.  The only (L1) relation in which this point takes part is $P_{r+a-c}, P_{r+b-c}, P_r$ collinear.  As such $P_r$ can be replaced by any other point on that line without violating $A \in X'$.  Therefore $A^{(m')}$ cannot be determined from $A^{(1)}, \ldots, A^{(m'-1)}$.
\end{proof}

We have shown $m_2(S) = \max(c_2-a_2,d_2-b_2)$.  In general, it seems $m_2(S) \geq m_3(S) \geq \ldots \geq m_{D(S)}(S)$.  The $m_D(S)$ for $D > 2$ seem mysterious in general, although we address several examples in Section \ref{seczoo}.

\begin{rmk}
In Section \ref{secex}, we related the $S$-maps for $S$ horizontal to an $(I,J)$-map in a certain dimension $D$.  The $(I,J)$-maps are by definition order 1, so it is tempting to say $m_D(S)=1$ in this case.  However, the $(I,J)$-map skips several rows of the $Y$-mesh which we do not allow.
\end{rmk}  

\section{Fractals} \label{secfractal}

\begin{defin}
Let $S = \{a,b,c,d\}$ be a $Y$-pin. 
 A {\it {$k$-fractal}} $f$ associated with $S$ is a collection of the following points:
$$f = \{r + \alpha a + \beta b + \gamma c + \delta d : 0 \leq \alpha, \beta, \gamma, \delta; \alpha + \beta + \gamma + \delta = k\}.$$
\end{defin}
Thus, a $0$-fractal is just a single point in $\mathbb Z^2$. A $1$-fractal is any of the equivalent $Y$-pins 
$$f = \{r+a, r+b, r+c, r+d\}.$$
A $2$-fractal looks as follows:
$$f = \{r+2a, r+a+b, r+a+c, r+2b, r+a+d, r+b+d,r+b+c,r+2c,r+c+d,r+2d\},$$ etc. 

\begin{ex}
Figure \ref{figgiraffe123} shows 1, 2, and 3-fractals for $S = \{(0,0), (2,0), (1,1), (2,3)\}.$
\end{ex}

\begin{figure}
\begin{pspicture}(0,-.5)(9,5)
\psdots(0,0)(1,0)(.5,.5)(1,1.5)
\rput(2.5,0){
\psdots(0,0)(1,0)(2,0)(.5,.5)(1.5,.5)(1,1)(1,1.5)(2,1.5)(1.5,2)(2,3)
}
\rput(6,0){
\psdots(0,0)(1,0)(2,0)(3,0)(.5,.5)(1.5,.5)(2.5,.5)(1,1)(2,1)(1,1.5)(1.5,1.5)(2,1.5)(3,1.5)(1.5,2)(2.5,2)(2,2.5)(2,3)(3,3)(2.5,3.5)(3,4.5)
%\psframe(-.3,-.3)(2.3,3.3)
%\psframe(.7,1.2)(3.3,4.8)
%\pspolygon(.6,-.2)(3.2,3.6)(3.2,-.2)(1.8,-.2)(1.5,.2)(1.2,-.2)
%\pspolygon(.6,1.3)(3.2,5.1)(3.2,1.3)(1.8,1.3)(1.5,1.7)(1.2,1.3)
\multirput(0,0)(0,1.5){2}{
\pspolygon(.6,-.2)(3.2,3.6)(3.2,-.2)(2.8,-.2)(2.5,.2)(2.2,-.2)(1.8,-.2)(1.5,.2)(1.2,-.2)}
}
\end{pspicture}  
\caption{The 1, 2, and 3-fractals for $S = \{(0,0), (2,0), (1,1), (2,3)\}$.  The 3-fractal is made of four overlapping 2-fractals (two of which are outlined) each pair of which intersects at a 1-fractal.}
\label{figgiraffe123}
\end{figure}
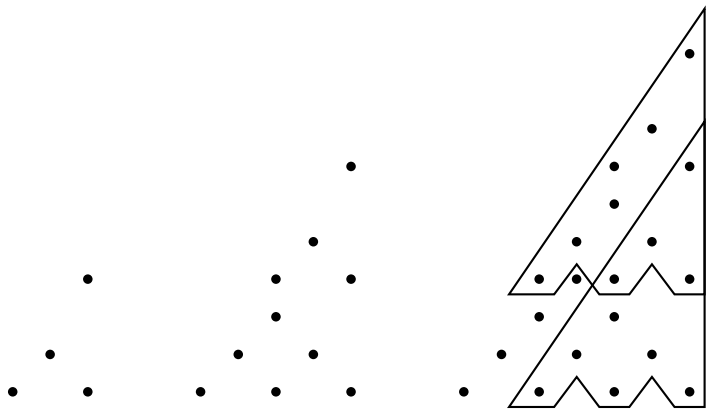

Let $f = \{r + \alpha a + \beta b + \gamma c + \delta d\}$ be a $k$-fractal as above. We can consider four distinguished $(k-1)$-fractals inside of it:
$f_a = \{r + \alpha a + \beta b + \gamma c + \delta d \in f : 0 < \alpha\}$, same for $f_b, f_c, f_d$. We say that those four $(k-1)$-fractals {\it {belong}}
to the $k$-fractal $f$. 

\begin{lem} \label{lem:fabcd}
 If two distinct $(k-1)$-fractals $f_1, f_2$ belong to a $k$-fractal $f$, then their intersection $f_1 \cap f_2$ is a $(k-2)$-fractal belonging to each of them. 
\end{lem}

\begin{proof}
 Easy verification from the definition. For example, intersection of $f_a$ and $f_b$ is $(k-2)$-fractal $\{r + \alpha a + \beta b + \gamma c + \delta d \in f : 0 < \alpha, \beta\}$.
\end{proof}

\begin{ex}
On the right in Figure \ref{figgiraffe123}, two of the 2-fractals are outlined, and their intersection is a 1-fractal.  \end{ex}

The informal meaning of $k$-fractals is as follows: those collections of points in a $Y$-mesh have to lie in a $k$-dimensional affine subspace of the whole space. 
For example, $1$-fractals are $Y$-pins, which by definition are quadruples of points that lie on one line.
To make this informal meaning rigorous however we usually need to make certain genericity assumptions about the $Y$-mesh. 

Let $S$ be a $Y$-pin and let $P$ be a $Y$-mesh of type $S$ and dimension $D \leq
D(S)$. For any finite set $A \subset Z^2$ denote $d_P(A)$ the dimension of the
affine span $$\overline A = \langle \{ P_r : r \in A\} \rangle$$ of points associated with elements of $A$ in $P$.

\begin{defin}
We say that a $Y$-mesh $P$ is {\it {$d$-generic}} for some $d \leq D$ if for any $i$-fractal $f$ with $i
\leq d$ we have $d_P(f) = i$.
\end{defin}

\begin{lem} \label{lem:dgen}
 If a $Y$-mesh $P$ is $d$-generic and $i < d$, then any two $i$-fractals $f_1$, $f_2$
belonging to the same $(i+1)$-fractal $f$ have distinct affine spans.
\end{lem}

\begin{proof}
 The proof is by induction on $i$, case $i=0$ being true by the definition of a $Y$-mesh. 
 
 Let us do one step of induction. Without loss of generality assume $f_1 = f_a$ and $f_2 = f_b$. 
 By $d$-genericity of $P$ we know that affine spans of $f_a$ and $f_b$ are $i$-dimensional. 
 Assume the claim of the lemma is false and those affine spans coincide. 
 
 Consider one of the 
 two remaining $i$-fractals that belong to $f$, for example $f_c$. By Lemma \ref{lem:fabcd}
 we know that $f_c \cap f_a$ and $f_c \cap f_b$ are $(i-1)$-fractals belonging to $f_c$. It is also 
 easy to see they are distinct. By $d$-genericity of $P$ their affine spans $\overline {f_c \cap f_a}$ and $\overline {f_c \cap f_b}$ are $(i-1)$-dimensional, 
 and by induction assumption they are distinct. Thus, their sum $\overline {f_c \cap f_a} + \overline {f_c \cap f_b}$ is at least an $i$-dimensional space.
 
 On the other hand, this sum is contained in the affine span of $f_c$, and thus is at most 
 $i$-dimensional. Therefore, this sum coincides with the affine span of $f_c$. However, sum $\overline {f_c \cap f_a} + \overline {f_c \cap f_b}$  
 lies in the $i$-dimensional space which is the affine span $\overline f_a = \overline f_b$ by our assumption. This is because $f_c \cap f_a \subset f_a$ and 
 $f_c \cap f_b \subset f_b$. Thus, $\overline f_c$ lies in this 
 $i$-dimensional space. Same argument shows  $\overline f_d$ lies there as well. This means however that the affine 
 span of $(i+1)$-fractal $f$ is $i$-dimensional, which contradicts $d$-genericity of $P$. The contradiction implies 
 our assumption was false and the lemma holds.  
\end{proof}

\begin{prop} \label{thm:dgen}
 If a $Y$-mesh $P$ is $d$-generic then for any $(d+1)$-fractal $f$ we have $d_P(f) \leq d+1$.
\end{prop}

\begin{proof}
 Consider the four $d$-fractals $f_a$, $f_b$, $f_c$ and $f_d$. If affine spans of all four coincide $\overline f_a = \overline f_b = \overline f_c = \overline f_d$, then in turn the affine span $\overline f$ of the whole $f$
 coincides with each of them. This is because $f = f_a \cup f_b \cup f_c \cup f_d$. Then by $d$-genericity of $P$ we conclude $d_P(f) = d < d+1.$
 
 Assume now not all of the four affine spans coincide. Without loss of generality consider the case when $\overline f_a$ and $\overline f_b$ are distinct. Consider the two $(d-1)$-fractals $f_c \cap f_a$
 and $f_c \cap f_b$. Its easy to see they are distinct, and by Lemma \ref{lem:dgen} so are their affine spans. In that case $\overline {f_c \cap f_a} + \overline {f_c \cap f_b}$ is $d$-dimensional 
 and must coincide with $\overline f_c$. Thus, $\overline f_c \subset \overline f_a + \overline f_b$. Similarly $\overline f_c \subset \overline f_a + \overline f_b$. This implies 
 $\overline f = \overline f_a + \overline f_b$. Now, $\overline f_a$ and $\overline f_b$ are two distinct $d$-dimensional spaces who's intersection is at least $(d-1)$-dimensional, since it is a 
 $(d-1)$-fractal and $P$ is $d$-generic. It is easy to see that in this case $\overline f_a + \overline f_b$ is $(d+1)$-dimensional, as desired. 
\end{proof}

\begin{prop}
 Any $Y$-mesh $P$ is $2$-generic.
\end{prop}

\begin{proof}
 By definition of a $Y$-mesh, the dimension of any $Y$-pin $S = \{a,b,c,d\}$ in $P$ is $1$-dimensional: $P_a$, $P_b$, $P_c$ and $P_d$ lie on one line, but do not coincide. 
 This takes care of $1$-fractals. For any $2$-fractal $f$ note that lines $\overline f_a$ and $\overline f_b$ intersect (at point $r+a+b$) but do not coincide, again as one of the 
 requirements in the definition of $Y$-mesh. Then their affine span $\overline f_a + \overline f_b$ is $2$-dimensional. Now, each of the lines $\overline f_c$ and $\overline f_d$ must 
 similarly intersect but not coincide with $\overline f_a$, $\overline f_b$. Furthermore, points $P_{r+a+c}$ and $P_{r+b+c}$ of those intersections cannot coincide by one of the defining assumptions on $P$.
 This implies $\overline f_c$ must lie in the plane $\overline f_a + \overline f_b$, and same for $\overline f_d$. The claim that all $2$-fractals are $2$-dimensional in $P$ follows. 
\end{proof}

\begin{conj} \label{conjfract}
 If $D \leq D(S)$, then $D$-generic $D$-dimensional $Y$-meshes of
type $S$ exist.
\end{conj}

The fractals are quite useful for discovering low order $S$-maps in various dimensions.  Suppose for some $m$ that we want to express the point $P_{0,m+1}$ of a $Y$-mesh in terms of the $P_{i,j}$ with $1 \leq j \leq m$.  One approach is to look for some fractals $f_1,\ldots, f_k$ each containing $(0,m+1)$ such that
\begin{itemize}
\item each of $\overline{f}_1,\ldots, \overline{f}_k$ is spanned by some subset of the $P_{i,j}$ with $1 \leq j \leq m$
\item $P_{0,m+1} = \overline{f}_1 \cap \overline{f}_2 \cap \ldots \cap \overline{f}_k$
\end{itemize}
Conjecture \ref{conjfract} gives a generic sense of the dimension of the fractal subspaces, which helps determine when these conditions hold.  Once an appropriate construction is found in this manner, it can be verified geometrically.  We illustrate this technique in Section \ref{seczoo}.

\section{Zoo} \label{seczoo}
This section examines a zoo of $Y$-pins $S$ in order to illustrate the great variety of geometric constructions that ensue.
The examples covered are summarized in Table \ref{tabzoo}.

\begin{table}
\begin{tabular}{c|c|c|c}
Picture & $S$ & $D(S)$ & Name / Citation \\
\hline
\begin{pspicture}(0,0)(1.5,.7)
\psdots(0,0)(.5,0)(0,.5)(.5,.5)
\end{pspicture} & \{(0,0), (1,0), (0,1), (1,1)\} & 1 & lower pentagram \cite{GSTV} \\
\hline
\begin{pspicture}(0,0)(1.5,.7)
\psdots(0,0)(1,0)(0,.5)(.5,.5)
\end{pspicture} & \{(0,0), (2,0), (0,1), (1,1)\} & 2 & pentagram \cite{S1} \\
\hline
\begin{pspicture}(0,0)(1.5,.7)
\psdots(0,0)(1.5,0)(.5,.5)(1,.5)
\end{pspicture} & \{(0,0), (3,0), (1,1), (2,1)\} & 3 & higher pentagram \cite{GSTV} \\
\hline
\begin{pspicture}(0,0)(1.5,1.2)
\psdots(0,0)(.5,0)(.5,.5)(0,1)
\end{pspicture} & \{(0,0), (1,0), (1,1), (0,2)\} & 2 &  sideways pentagram \\
\hline
\begin{pspicture}(0,0)(1.5,1.2)
\psdots(0,0)(1,0)(.5,.5)(.5,1)
\end{pspicture} & \{(-1,0), (1,0), (0,1), (0,2)\} & 3 & short diagonal hyperplane \cite{KS1} \\
\hline
\begin{pspicture}(0,0)(1.5,1.2)
\psdots(.5,0)(1,.5)(0,1)(.5,1)
\end{pspicture} & \{(1,0), (2,1), (0,2), (1,2)\} & 3 & dented pentagram \cite{KS2} \\
\hline
\begin{pspicture}(0,0)(2,1.7)
\psdots(.5,0)(1,0)(1,.5)(1.5,1.5)
\end{pspicture} & \{(0,0), (1,0), (1,1), (2,3)\} & 2 &  gopher \\
\hline
\begin{pspicture}(0,0)(2,1.7)
\psdots(.5,0)(1,0)(.5,1)(.5,1.5)
\end{pspicture} & \{(0,0), (1,0), (0,2), (0,3)\} & 2 &  penguin \\
\hline
\begin{pspicture}(0,0)(1.5,1.7)
\psdots(0,0)(1,.5)(.5,1)(.5,1.5)
\end{pspicture} & \{(-1,0), (1,1), (0,2), (0,3)\} & 4 &  rabbit \\
\hline
\begin{pspicture}(0,0)(1.5,1.7)
\psdots(0,0)(1,0)(.5,.5)(1,1.5)
\end{pspicture} & \{(0,0), (2,0), (1,1), (2,3)\} & 5 &  giraffe \\
\hline
\begin{pspicture}(0,0)(1.5,2.2)
\psdots(1,2)(0,1)(.5,.5)(.5,0)
\end{pspicture} & \{(1,0), (1,1), (0,2), (2,4)\} & 5 &  kangaroo \\
\hline
\begin{pspicture}(0,0)(1.5,1.7)
\psdots(1.5,1.5)(0,1)(.5,.5)(.5,0)
\end{pspicture} & \{(1,0),(1,1),(0,2),(3,3)\} & 6 &  elephant \\
\hline
\end{tabular}
\caption{A zoo of $Y$-pins.}
\label{tabzoo}
\end{table}

\subsection{Lower pentagram} The simplest possible $Y$-pin is $$S=\{(0,0),(1,0),(0,1),(1,1)\}.$$  
The corresponding recurrence \eqref{eqmain} and quiver $Q_S$ were first studied by Gekhtman, Shapiro, Tabachnikov, and Vainshtein \cite{GSTV}.  
Note that $D(S)=1$, but we have only defined $Y$-meshes in dimensions two and higher.  
However, the lower pentagram map \cite{GSTV} provides a geometric interpretation on the projective line and motivates a definition of a one dimensional $Y$-mesh for general $S$, including those with $D(S) > 1$.  We pursue this matter in Section \ref{sec1d}.

\subsection{Pentagram and higher pentagram}
The cases $$S= \{(0,0), (2,0), (0,1), (1,1)\} \text{ and } S=\{(0,0), (3,0), (1,1), (2,1)\}$$ were discussed in Example \ref{exGSTV}.  The former gives rise to the standard pentagram map.  
The latter $S$ (previously expressed as $S = \{(0,0), (3,0), (0,1),(1,1)\}$ which is equivalent as per Remark \ref{rmkequiv}) corresponds to the higher pentagram map on corrugated polygons 
in three dimensions \cite{GSTV}.

\subsection{Sideways pentagram} Now let $$S=\{(0,0), (1,0), (1,1), (0,2)\}$$ and $D=2$.  The corresponding $S$-map $T_{2,S}$ is defined on the space $X_{2,S}$ of pairs $A, B$ of polygons satisfying
\begin{displaymath}
A_{i-1}, A_{i}, B_i
\end{displaymath}
collinear for all $i \in \mathbb{Z}$.  The map is $T_{2,S}(A,B) = (B,C)$ where
\begin{displaymath}
C_i = \meet{\join{A_i}{A_{i+1}}}{\join{B_{i-1}}{B_i}}.
\end{displaymath}
The map $T_{2,S}$ is related to the original pentagram map as follows.  Begin with any polygon and apply the pentagram map to obtain a sequence of polygons.  Let $A_i$ denote vertex 1 of the $i$th polygon in the sequence.  Let $B_i$ denote vertex 2 of the $i$th polygon of the sequence.  Then $(A,B) \in X_{2,S}$ and $T_{2,S}(A,B) = (B,C)$ where $C_i$ is vertex 3 of the $i$th polygon in the sequence.  The relationship between the two systems can be seen on the level of the $Y$-pins which are 90 degree rotations of each other.

\subsection{Short diagonal hyperplane and dented pentagram}
The cases $$S=\{(-1,0),(1,0),(0,1),(0,2)\} \text{ and } S=\{(1,0), (2,1), (0,2), (1,2)\}$$ were discussed in Section \ref{secex}.  
The corresponding systems can be thought of respectively as square roots of the short diagonal hyperplane map \cite{KS1} and the dented pentagram map \cite{KS2}, both in dimension three.

\subsection{Gopher}

Let $S=\{(0,0), (1,0), (1,1), (2,3)\}$.  In this case we have $p=1$ and $q=3$ in the notation of Section \ref{secex}. We also have $D(S)=2$. This dimension is smaller than $p+q$, thus the methods of Section \ref{secex} do not apply.

This system has the following elegant geometric description.  Let $A, B, C$ be polygons such that $B_{i+1}$ lies on $\langle A_i, A_{i+1} \rangle$ and $C_{i+1}$ lies on $\langle B_i, B_{i+1} \rangle$ for all $i$. Then the next 
layer of $D_i$'s is obtained as follows:
$$D_{i+2} = \langle A_i, A_{i+1} \rangle \cap \langle C_{i+1}, C_{i+2} \rangle.$$
An example of the gopher map is given in Figure \ref{figgopher}.

\begin{figure}
\psset{unit=.5cm}
\begin{pspicture}(5,-16)(18,-1)
\pnode(11.7,-3.86){A1}
\pnode(15.86,-7.56){A2}
\pnode(14.16,-12.48){A3}
\pnode(8.24,-11.68){A4}
\pnode(6.88,-6.82){A5}
\pnode(9.53,-5.19){B1}
\pnode(13.86,-5.78){B2}
\pnode(15.14,-9.65){B3}
\pnode(11.38,-12.1){B4}
\pnode(7.56,-9.25){B5}
\pnode(8.12,-8.09){C1}
\pnode(11.11,-5.41){C2}
\pnode(14.34,-7.22){C3}
\pnode(13.82,-10.51){C4}
\pnode(9.72,-10.87){C5}
\pnode(6.4,-5.09){D1}
\pnode(15.23,-1.69){D2}
\pnode(17.44,-8.97){D3}
\pnode(12.97,-15.91){D4}
\pnode(5.15,-11.26){D5}
\pspolygon[showpoints=true](A1)(A2)(A3)(A4)(A5)
\pspolygon[showpoints=true](B1)(B2)(B3)(B4)(B5)
\pspolygon[showpoints=true](C1)(C2)(C3)(C4)(C5)
\pspolygon[showpoints=true](D1)(D2)(D3)(D4)(D5)
%\psdots(A1)(A2)(A3)(A4)(A5)
%\psdots(B1)(B2)(B3)(B4)(B5)
%\psdots(C1)(C2)(C3)(C4)(C5)
%\psdots(D1)(D2)(D3)(D4)(D5)
\psline[linestyle=dashed](A1)(D2)
\psline[linestyle=dashed](A2)(D3)
\psline[linestyle=dashed](A3)(D4)
\psline[linestyle=dashed](A4)(D5)
\psline[linestyle=dashed](A5)(D1)
\psline[linestyle=dashed](C1)(D1)
\psline[linestyle=dashed](C2)(D2)
\psline[linestyle=dashed](C3)(D3)
\psline[linestyle=dashed](C4)(D4)
\psline[linestyle=dashed](C5)(D5)
\uput[l](A1){$A_1$}
\uput[u](A2){$A_2$}
\uput[ur](A3){$A_3$}
\uput[d](A4){$A_4$}
\uput[240](A5){$A_5$}
\uput[150](B1){$B_1$}
\uput[ur](B2){$B_2$}
\uput[r](B3){$B_3$}
\uput[d](B4){$B_4$}
\uput[dl](B5){$B_5$}
\uput[r](C1){$C_1$}
\uput[d](C2){$C_2$}
\uput[dl](C3){$C_3$}
\uput[ul](C4){$C_4$}
\uput[ur](C5){$C_5$}
\uput[ul](D1){$D_1$}
\uput[r](D2){$D_2$}
\uput[r](D3){$D_3$}
\uput[r](D4){$D_4$}
\uput[dl](D5){$D_5$}
\end{pspicture}
\caption{The gopher map $(A,B,C) \mapsto (B,C,D)$.}
\label{figgopher}
\end{figure}
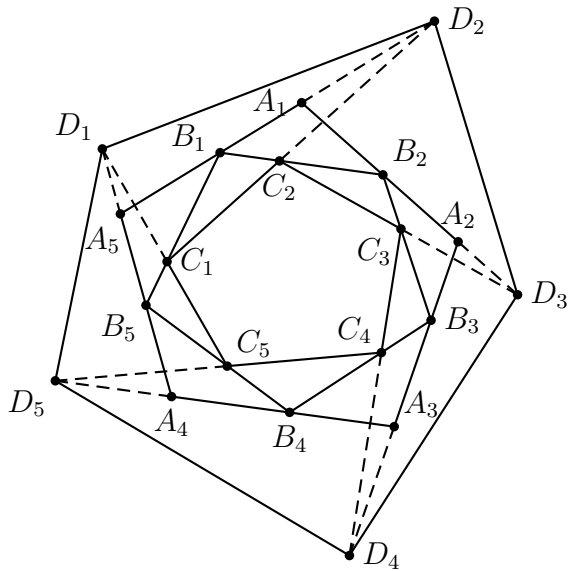

\subsection{Penguin}
Let $S = \{(0,0), (1,0), (0,2), (0,3)\}$.  In this case we have $p=2$ and $q=3$ in the notation of Section \ref{secex}. We also have $D(S)=2$. This dimension is smaller than $p+q$, thus the methods of Section \ref{secex} do not apply.

This system has the following elegant geometric description. Let $A, B, C$ be polygons such that $C_i$ lies on $\langle A_i, A_{i+1} \rangle$ for all $i$. Then the next 
layer of $D_i$'s is obtained as follows:
$$D_i = \langle A_i, A_{i+1} \rangle \cap \langle B_i, B_{i+1} \rangle.$$
An example of the penguin map is given in Figure \ref{figpenguin}.

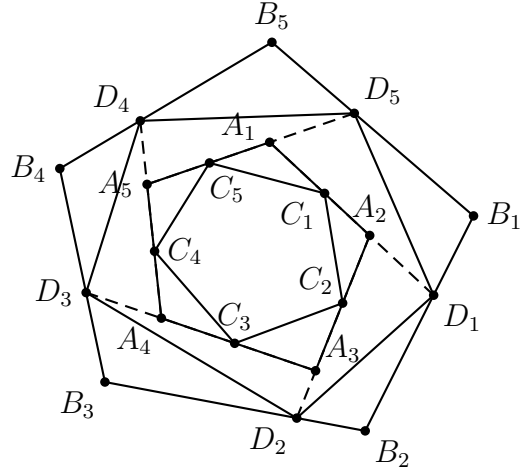
\begin{figure}
\psset{unit=.5cm}
\begin{pspicture}(7,-13)(19,-1)
\pnode(12.72,-4.7){A1}
\pnode(15.38,-7.18){A2}
\pnode(13.94,-10.78){A3}
\pnode(9.84,-9.38){A4}
\pnode(9.46,-5.82){A5}
\pnode(18.14,-6.66){B1}
\pnode(15.26,-12.38){B2}
\pnode(8.34,-11.08){B3}
\pnode(7.14,-5.4){B4}
\pnode(12.78,-2.04){B5}
\pnode(14.18,-6.06){C1}
\pnode(14.66,-8.98){C2}
\pnode(11.79,-10.05){C3}
\pnode(9.66,-7.6){C4}
\pnode(11.12,-5.25){C5}
\pnode(17.08,-8.77){D1}
\pnode(13.44,-12.04){D2}
\pnode(7.84,-8.7){D3}
\pnode(9.28,-4.13){D4}
\pnode(14.97,-3.93){D5}
\pspolygon[showpoints=true](A1)(A2)(A3)(A4)(A5)
\pspolygon[showpoints=true](B1)(B2)(B3)(B4)(B5)
\pspolygon[showpoints=true](C1)(C2)(C3)(C4)(C5)
\pspolygon[showpoints=true](D1)(D2)(D3)(D4)(D5)
%\psdots(A1)(A2)(A3)(A4)(A5)
%\psdots(B1)(B2)(B3)(B4)(B5)
%\psdots(C1)(C2)(C3)(C4)(C5)
%\psdots(D1)(D2)(D3)(D4)(D5)
\psline[linestyle=dashed](A1)(D1)
\psline[linestyle=dashed](A2)(D2)
\psline[linestyle=dashed](A3)(D3)
\psline[linestyle=dashed](A4)(D4)
\psline[linestyle=dashed](A5)(D5)
\uput[150](A1){$A_1$}
\uput[u](A2){$A_2$}
\uput[ur](A3){$A_3$}
\uput[dl](A4){$A_4$}
\uput[l](A5){$A_5$}
\uput[r](B1){$B_1$}
\uput[dr](B2){$B_2$}
\uput[dl](B3){$B_3$}
\uput[l](B4){$B_4$}
\uput[u](B5){$B_5$}
\uput[dl](C1){$C_1$}
\uput[ul](C2){$C_2$}
\uput[u](C3){$C_3$}
\uput[r](C4){$C_4$}
\uput[300](C5){$C_5$}
\uput[dr](D1){$D_1$}
\uput[dl](D2){$D_2$}
\uput[l](D3){$D_3$}
\uput[ul](D4){$D_4$}
\uput[ur](D5){$D_5$}
\end{pspicture}
\caption{The penguin map $(A,B,C) \mapsto (B,C,D)$.}
\label{figpenguin}
\end{figure}

\subsection{Rabbit}
A first example of a nonhorizontal $Y$-pin (one with $a_2<b_2<c_2<d_2$) is $$S = \{(-1,0), (1,1), (0,2), (0,3)\}.$$  
We considered the $D=2$ case in Examples \ref{exF2} and \ref{exF2cont} showing that the order of the $S$-map can be reduced from three to two using the techniques of Section \ref{secorder}.  The map in the plane inputs two generic polygons $A$ and $B$ and builds another $C$ via
\begin{equation} \label{eqrabbit}
C_i = \meet{\join{A_{i-1}}{B_{i+1}}}{\join{A_{i+1}}{B_i}}.
\end{equation}
The rabbit map in the plane is pictured in Figure \ref{figrabbit}.

\begin{figure}
\psset{unit=.5cm}
\begin{pspicture}(6,-12)(20,0)
\pnode(12.2,-3.18){A1}
\pnode(16.2,-5.92){A2}
\pnode(14.96,-10.2){A3}
\pnode(10,-10){A4}
\pnode(8.66,-5.76){A5}
\pnode(11.48,-.76){B1}
\pnode(18,-5){B2}
\pnode(17.02,-10.78){B3}
\pnode(9.3,-11.86){B4}
\pnode(6.72,-5.76){B5}
\pnode(15.54,-5.2){C1}
\pnode(15.77,-8.81){C2}
\pnode(11.29,-10.14){C3}
\pnode(8.78,-6.87){C4}
\pnode(11,-3.74){C5}
%\pspolygon[showpoints=true](A1)(A2)(A3)(A4)(A5)
%\pspolygon[showpoints=true](B1)(B2)(B3)(B4)(B5)
%\pspolygon[showpoints=true](C1)(C2)(C3)(C4)(C5)
\psdots(A1)(A2)(A3)(A4)(A5)
\psdots(B1)(B2)(B3)(B4)(B5)
\psdots(C1)(C2)(C3)(C4)(C5)
\psline[linestyle=dashed](A1)(B3)
\psline[linestyle=dashed](A2)(B4)
\psline[linestyle=dashed](A3)(B5)
\psline[linestyle=dashed](A4)(B1)
\psline[linestyle=dashed](A5)(B2)
\psline[linestyle=dashed](A1)(B5)
\psline[linestyle=dashed](A2)(B1)
\psline[linestyle=dashed](A3)(B2)
\psline[linestyle=dashed](A4)(B3)
\psline[linestyle=dashed](A5)(B4)
\uput[u](A1){$A_1$}
\uput[r](A2){$A_2$}
\uput[r](A3){$A_3$}
\uput[120](A4){$A_4$}
\uput[ur](A5){$A_5$}
\uput[u](B1){$B_1$}
\uput[r](B2){$B_2$}
\uput[dr](B3){$B_3$}
\uput[dl](B4){$B_4$}
\uput[l](B5){$B_5$}
\uput[ur](C1){$C_1$}
\uput[r](C2){$C_2$}
\uput[d](C3){$C_3$}
\uput[dl](C4){$C_4$}
\uput[ul](C5){$C_5$}

\end{pspicture}
\caption{The rabbit map $(A,B) \mapsto (B,C)$.}
\label{figrabbit}
\end{figure}
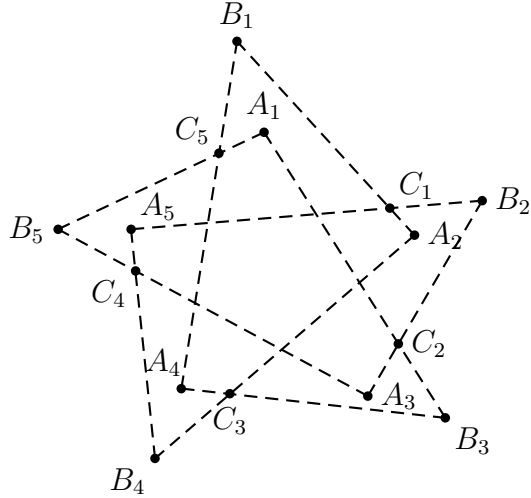

Now $D(S)=4$, and we argue that this order 2 map can be lifted to dimensions $D=3$ and $D=4$.  A natural definition would be to work with pairs $A$, $B$ of polygons with  $A_{i-1}, B_{i+1}, A_{i+1}, B_i$ coplanar so that $\eqref{eqrabbit}$ can still be applied.  However, this coplanarity condition will not by itself propagate (i.e. continue to hold for the pair $B,C$ of polygons).  It turns out that the pair of conditions  
$$A_{i-1}, B_{i+1}, A_{i+1}, B_i \text { are coplanar; } A_{i-1}, B_{i-1}, A_{i+2}, B_{i+1} \text { are coplanar}$$ does propagate.
Indeed, lines $\join{C_{i-1}}{B_{i-1}}$ and $\join{C_{i+1}}{B_{i+2}}$ intersect at $A_i$, according to the way we create $C_{i-1}$ and $C_{i+1}$. Thus, those 
four points are coplanar. On the other hand, line $\langle C_i, B_{i+1} \rangle$ coincides with line $\langle A_{i-1}, B_{i+1} \rangle$, while point $C_{i+1}$ lies on line $\langle B_{i+1},A_{i+2} \rangle$. 
This implies that both 
lines $\langle C_i, B_{i+1} \rangle$ and $\langle B_{i-1}, C_{i+1} \rangle$ lie in the plane $\langle A_{i-1}, B_{i-1}, A_{i+2}, B_{i+1} \rangle$ and thus intersect. Therefore, those four points are also coplanar. 

The mysterious extra condition that $A_{i-1}, B_{i-1}, A_{i+2}, B_{i+1} \text { are coplanar}$ is in fact what we would use to propagate the $Y$-mesh in the opposite time direction using \eqref{eqG2inv}.

\subsection{Giraffe}
Let $$S = \{(0,0), (2,0), (1,1), (2,3)\}.$$  Then $S$ is horizontal with $p=1$ and $q=3$ in the notation of Section \ref{secex}.  
By Proposition \ref{propzontal} the $S$-map in dimension $D=p+q=4$ is related to the $(I,J)$-map with $I=(2,2,2)$, $J = (-3,2,2)$.  
It is equivalent up to index shifts to take $J=(2,1,1)$.  More concretely, a $Y$-mesh $P$ in dimension 4 will satisfy
\begin{equation} \label{eqp1q3}
B_i = \meet{\join{A_{i-2}}{A_i}}{\langle A_{i-3}, A_{i-1}, A_{i+1}, A_{i+3} \rangle}
%P_{i,3} = \meet{\join{P_{i-2,0}}{P_{i,0}}}{\langle P_{i-3,0}, P_{i-1,0}, P_{i+1,0}, P_{i+3,0} \rangle}.
\end{equation}
where $A_i = P_{i,0}$ and $B_i = P_{i,3}$.  Now $D(S) = 5$ so the $S$-map can be defined in that dimension.  It is natural to try to define the corresponding $(I,J)$-map for $D=5$ as well.  Unlike in dimension 4 it is unexpected for a line and a three-space to meet.  To apply \eqref{eqp1q3} then we need the six points $A_{i-3}, A_{i-2}, A_{i-1}, A_i, A_{i+1}, A_{i+3}$  to lie in a 4-space.  
%The conclusion is that the six points on the right hand side of \eqref{eqp1q3} actually lie on a common four-space.  This fact is supported by the 4-fractal, pictured in Figure \ref{figfract}.  The fourth row from the bottom contains six points arranged in the same manner as those on the right hand side of \eqref{eqp1q3}.

\begin{figure}
\begin{pspicture}(6,6)
\psdots(0,0)(1,0)(2,0)(3,0)(4,0)(.5,.5)(1.5,.5)(2.5,.5)(3.5,.5)(1,1)(2,1)(3,1)(1,1.5)(1.5,1.5)(2,1.5)(2.5,1.5)(3,1.5)(4,1.5)(1.5,2)(2,2)(2.5,2)(3.5,2)(2,2.5)(3,2.5)(2,3)(2.5,3)(3,3)(4,3)(2.5,3.5)(3.5,3.5)(3,4)(3,4.5)(4,4.5)(3.5,5)(4,6)
\end{pspicture}  
\caption{The 4-fractal for giraffe $S = \{(0,0), (2,0), (1,1), (2,3)\}$.}
\label{figfract}
\end{figure}
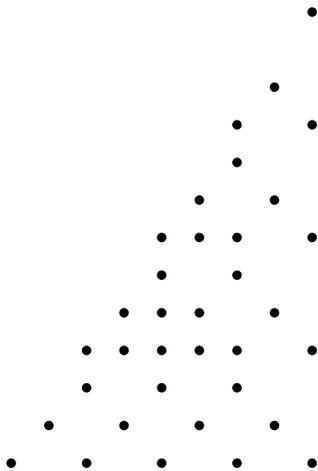

Amazingly, the condition 
$$A_{i-3}, A_{i-2}, A_{i-1}, A_i, A_{i+1}, A_{i+3} \text { lie on one $4$-space}$$ self-propagates under \eqref{eqp1q3}. Indeed, one can easily check that 
the $4$-space passing through $A_{i-5}, A_{i-3}, A_{i-1}, A_{i+1}, A_{i+3}$ passes through all six of $B_{i-3}$, $B_{i-2}$, $B_{i-1}$, $B_{i}$, $B_{i+1}$, $B_{i+3}$.
For example, $B_{i-3}$ lies on the line $\join{ A_{i-5}}{ A_{i-3}}$, while $B_{i-2}$ lies on the $3$-space $\langle A_{i-5}, A_{i-3}, A_{i-1}, A_{i+1} \rangle$. 

This sort of lift of the $(I,J)$-map can be performed anytime $D(S) > p+q$.  The simplest example is the lift of the higher diagonal map in the plane ($I=(k), J=(1)$ for $k \geq 3$) to the higher pentagram map on corrugated polygons \cite{GSTV}.  One way to see the needed extra condition is from the $(p+q)$-fractal.  For example, in the giraffe case the fourth row from the bottom of the $4$-fractal, pictured in Figure \ref{figfract}, indicates which six points of $A$ must lie on a $4$-space.

\subsection{Kangaroo}
Let $$S = \{(1,0), (1,1), (0,2), (2,4)\}.$$ Example \ref{exkangaroo} describes an $S$-map of order $3$ in the plane.  Now $D(S)=5$, and we show the order can be
reduced to $2$ for any $3 \leq D \leq 5$ as
follows.

Looking at the $2$-fractal (see Figure \ref{figkangaroo}), we see that
$$P_{i,k+2} = \langle P_{i+1,k+1},P_{i+1,k}\rangle \cap \langle
P_{i,k+1},P_{i-3,k+1},P_{i-2,k}\rangle \cap \langle
P_{i-2,k+1},P_{i-2,k},P_{i-1,k}\rangle.$$
Already in $3$ dimensions, and even more so in $4$ or $5$, such
triples of plane-plane-line do not have to intersect in general. Thus,
it is natural to ask if
this intersection property propagates.

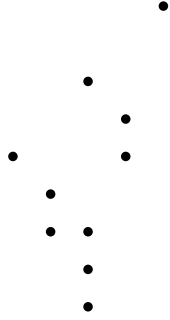
\begin{figure}
\begin{pspicture}(2,4)
\psdots(1,0)(1,.5)(.5,1)(1,1)(.5,1.5)(0,2)(1.5,2)(1.5,2.5)(1,3)(2,4)
\end{pspicture}  
\caption{The 2-fractal for kangaroo $S = \{(1,0), (1,1), (0,2), (2,4)\}$.}
\label{figkangaroo}
\end{figure}

Indeed it does. One checks from the propagation rule that all three of
the planes  $\langle P_{i,k+1},P_{i-3,k+1},P_{i-2,k}\rangle$, $\langle
P_{i-2,k+1},P_{i-2,k},P_{i-1,k}\rangle$
and $\langle P_{i-2,k},P_{i+1,k+1},P_{i+1,k}\rangle$ have a line
$\langle P_{i-2,k},P_{i-1,k-1}\rangle$ in common. This means that the
line along which the first two of those planes intersect must
intersect the line $\langle P_{i+1,k+1},P_{i+1,k}\rangle$, which is exactly the desired condition one time iteration later. 

\subsection{Elephant}
Let $$S = \{(1,0),(1,1),(0,2),(3,3)\}.$$ Then $D(S)=6$, and we can reduce the order of our system to $1$ for $D=6$, even though this is not a horizontal $Y$-pin.  The 4-fractal for $S$ is given in Figure \ref{figelephant}.

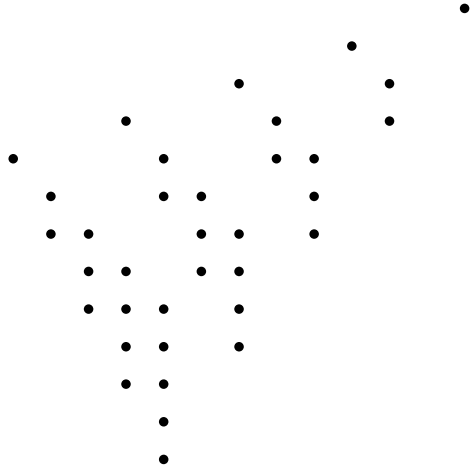
\begin{figure}
\begin{pspicture}(6,6)
\psdots
(2,0)
(2,0.5)
(1.5,1)(2,1)
(1.5,1.5)(2,1.5)(3,1.5)
(1,2)(1.5,2)(2,2)(3,2)
(1,2.5)(1.5,2.5)(2.5,2.5)(3,2.5)
(.5,3)(1,3)(2.5,3)(3,3)(4,3)
(.5,3.5)(2,3.5)(2.5,3.5)(4,3.5)
(0,4)(2,4)(3.5,4)(4,4)
(1.5,4.5)(3.5,4.5)(5,4.5)
(3,5)(5,5)
(4.5,5.5)
(6,6)
\end{pspicture}  
\caption{The 4-fractal for elephant $S = \{(1,0),(1,1),(0,2),(3,3)\}$.}
\label{figelephant}
\end{figure}

We see that one of the rows in the middle has $5$ points in it, which means that those points determine the $4$-space which is the affine span of the fractal in a $Y$-mesh. Looking 
at the next row up, we see that each point in $\mathbb Z^2$ is an intersection of four $4$-spaces determined by the points of the previous row. In $6$ dimensions this is certainly enough to determine a point.
Thus, we arrive to the following order $1$ description of the $Y$-mesh.
$$B_{i+8} = \langle A_{i+1},A_{i+2},A_{i+5},A_{i+6},A_{i+8} \rangle \cap \langle A_{i+4},A_{i+5},A_{i+8},A_{i+9},A_{i+11} \rangle$$ $$\cap \langle A_{i+5},A_{i+6},A_{i+9},A_{i+10},A_{i+12} \rangle \cap \langle A_{i+8},A_{i+9},A_{i+12},A_{i+13},A_{i+15} \rangle.$$
Note that in a $6$-space four $4$-spaces do not necessarily have a common point of intersection. Thus, this assumption needs to be made about the initial data, and a natural question arises: does this condition propagate?
We present the following argument for the fact that it does. Note that as usually, at several points in this argument we need to make additional genericity assumptions.

First we claim that lines $\join{A_{i+9}}{B_{i+9}}$ and $\join{A_{i+6}}{B_{i+5}}$ intersect. Indeed, they both lie in $4$-spaces $\langle A_{i+2},A_{i+3},A_{i+6},A_{i+7},A_{i+9} \rangle$ and 
$\langle A_{i+5},A_{i+6},A_{i+9},A_{i+10},A_{i+12} \rangle$. In a $6$-space those two $4$-spaces generically intersect in a plane, and any two lines in the same plane intersect. 
Denote $C'_{i+8}$ this point of intersection. We shall argue that $C'_{i+8}=C_{i+8}$, in other words that $C'_{i+8}$ is the common point of the four $4$-spaces that are supposed to define $C_{i+8}$.

In a similar fashion one can check that lines $\langle A_{i+5} B_{i+8} \rangle$ and $\langle A_{i+6} B_{i+5} \rangle$ intersect. Now, applying those two auxiliary claims several times, we see that affine spans 
$\langle A_{i+5},A_{i+8},B_{i+4},B_{i+8} \rangle$, $\langle A_{i+5},A_{i+6},B_{i+5},B_{i+8} \rangle$, $\langle A_{i+8},A_{i+9},B_{i+8},B_{i+11} \rangle$, and $\langle A_{i+6},A_{i+9},B_{i+5},B_{i+9} \rangle$
are planes. The second and third intersect the first among them along lines, so the union of those three belongs to a $4$-space. The last plane is contained in this $4$-space because points 
$B_{i+5}, A_{i+6}, A_{i+9}$ are. Therefore, all nine points involved, the four $A$'s and the five $B$'s, have affine span of dimension $4$. 

This means however that $A_{i+9}$ lies on $\langle B_{i+8},B_{i+9},B_{i+12},B_{i+13},B_{i+15} \rangle$, then so does $C'_{i+8}$. Similarly, $A_{i+6}$ and $A_{i+9}$ lie on 
$\langle B_{i+4},B_{i+5},B_{i+8},B_{i+9},B_{i+11} \rangle$ and also on $\langle B_{i+5},B_{i+6},B_{i+9},B_{i+10},B_{i+12} \rangle$, and thus so does $C'_{i+8}$. Finally, 
$A_{i+6}$ lies on $\langle B_{i+1},B_{i+2},B_{i+5},B_{i+6},B_{i+8} \rangle$, and thus so does $C'_{i+8}$. The argument is complete. 

\section{Periodic two dimensional quivers} \label{secquiver}
In this section, we review background on cluster algebras and then present a generalization of work of Fordy and Marsh \cite{FM}.

Cluster algebras, introduced by S. Fomin and A. Zelevinsky \cite{FZ1} are commutative rings endowed with certain combinatorial structure.  We will restrict ourselves to the case when this extra
structure can be encoded by a quiver. A \emph{quiver} $Q$ is a directed graph without loops or oriented two-cycles. 

Let $v$ be a vertex in a quiver $Q$. The quiver mutation $\mu_v$ transforms $Q$ into the new quiver $\mu_v(Q)$ defined as follows:
\begin{itemize}
 \item For each pair of directed edges $u \rightarrow v$ and $v \rightarrow w$, add a new edge $u \rightarrow w$. 
 \item Reverse directions of all edges incident to $v$.
 \item Remove all oriented two-cycles.
\end{itemize}

A \emph{seed} is a pair $(Q,\x)$ where $\x= (x_v)_{v \in V}$ is indexed by the vertex set $V$ of $Q$.  Given a vertex $v$, the \emph{seed mutation} $\mu_v$ transforms $(Q, \x)$ into the seed
$(Q', \x') = \mu_v(Q, \x)$ where $Q'= \mu_v(Q)$ as defined above, $x_u'=x_u$ whenever $u \neq v$, and
\begin{equation} \label{eqmux}
%x_u' = \begin{cases}
x_vx_v' =
\displaystyle\prod_{v \rightarrow w} x_w + \displaystyle\prod_{w \rightarrow v} x_w%, & u = v \\
%x_u, & \textrm{ otherwise}
%\end{cases}
\end{equation}

There is an auxiliary notion of $Y$-seeds $(Q,\y)$ with $\y=(y_v)_{v\in V}$ and an associated mutation rule \cite{FZ2}.  Specifically, given a vertex $v$, $\mu_v$ transforms $(Q,\y)$ to the $Y$-seed $(Q',\y')=\mu_v(Q,\y)$ where $Q'=\mu_v(Q)$ and
\begin{equation} \label{eqmuy}
y_u' = \begin{cases}
(y_v)^{-1} & u = v \\
y_u\left(\displaystyle\prod_{u \rightarrow v} (1+y_v)\right)\left(\displaystyle\prod_{v \rightarrow u} (1+y_v^{-1})\right)^{-1} , & \textrm{ otherwise}
\end{cases}
\end{equation}
Figure \ref{figmutate} gives an example of seed and $Y$-seed mutations.

\begin{figure}
\begin{pspicture}(0,1.5)(13,6)
\SpecialCoor
\rput(0,3){
% The vertices
\cnode(1,1){.25}{v10}
\rput(v10){1}
\cnode(2.5,1){.25}{v20}
\rput(v20){2}
\cnode(4,1){.25}{v30}
\rput(v30){3}
\cnode(1,2.5){.25}{v40}
\rput(v40){4}
\cnode(2.5,2.5){.25}{v50}
\rput(v50){5}
\cnode(4,2.5){.25}{v60}
\rput(v60){6}
% The edges
\psset{arrowsize=5pt}
\psset{arrowinset=0}
\ncline{->}{v10}{v20}
\ncline{->}{v30}{v20}
\ncline{->}{v40}{v10}
\ncline{->}{v20}{v50}
\ncline{->}{v60}{v30}
\ncline{->}{v50}{v40}
\ncline{->}{v50}{v60}
}
\rput(7,3){
% The vertices
\cnode(1,1){.25}{v11}
\rput(v11){1}
\cnode(2.5,1){.25}{v21}
\rput(v21){2}
\cnode(4,1){.25}{v31}
\rput(v31){3}
\cnode(1,2.5){.25}{v41}
\rput(v41){4}
\cnode(2.5,2.5){.25}{v51}
\rput(v51){5}
\cnode(4,2.5){.25}{v61}
\rput(v61){6}
% The edges
\psset{arrowsize=5pt}
\psset{arrowinset=0}
\ncline{->}{v11}{v51}
\ncline{->}{v31}{v51}
\ncline{->}{v21}{v11}
\ncline{->}{v21}{v31}
\ncline{->}{v41}{v11}
\ncline{->}{v51}{v21}
\ncline{->}{v61}{v31}
\ncline{->}{v51}{v41}
\ncline{->}{v51}{v61}
}
\psset{arrowsize=1.5pt 2}
\psline{->}(5,4.75)(7,4.75)
\uput[u](6,4.75){$\mu_2$}
\rput(2.5,3){$(x_1,x_2,x_3,x_4,x_5,x_6)$}
\psline{->}(4.5,3)(7,3) 
\rput(9.5,3){$(x_1,\frac{x_1x_3+x_5}{x_2},x_3,x_4,x_5,x_6)$}
\rput(2.5,2){$(y_1,y_2,y_3,y_4,y_5,y_6)$}
\psline{->}(4.5,2)(6,2) 
\rput(10,2){$\left(y_1(1+y_2),\frac{1}{y_2},y_3(1+y_2),y_4,\frac{y_5}{1+y_2^{-1}},y_6\right)$}
\end{pspicture}
\caption{An example of a quiver mutation, along with the corresponding birational maps of the $x$ and $y$ variables.}
\label{figmutate}
\end{figure}
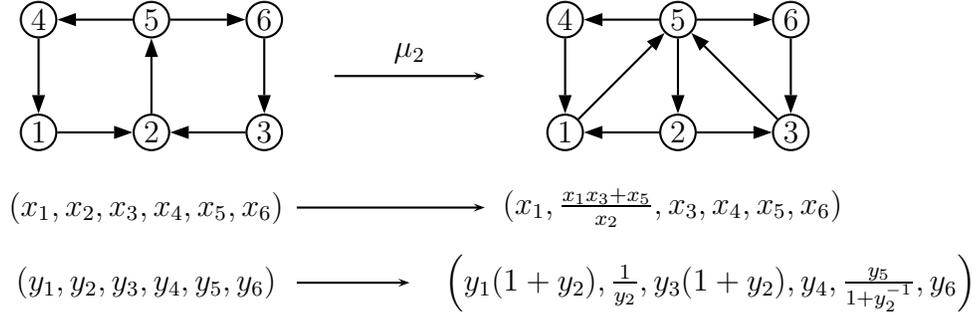

\begin{defin}[Fordy and Marsh \cite{FM}]
Say a quiver $Q$ on vertex set $V=\{1,2,\ldots, m\}$ has \emph{period one} if $\mu_1(Q) = \rho(Q)$ where $\rho$ denotes the procedure of relabeling the vertices according to $\rho(j) = j+1$ for $1 \leq j < m$ and $\rho(m)=1$.
\end{defin}
Begin with an initial seed $(\x,Q)$ with $\x = (x_j)_{j=1,\ldots, m}$ and $Q$ period one.  Perform mutations in the following order: $\mu_1, \mu_2,\ldots, \mu_m,\mu_1,\mu_2,\ldots$.  Call the single new variable produced by the $k$th mutation in this sequence $x_{m+k}$.  Then the periodicity of the quiver ensures 
\begin{displaymath}
x_{j+m}x_j = \displaystyle\prod_{1 \rightarrow k+1} x_{j+k} + \prod_{k+1 \rightarrow 1} x_{j+k}
\end{displaymath}
for all $j \in \mathbb{Z}$ where the products are over arrows in the original quiver $Q$.

It is more difficult to describe the $Y$-dynamics of a period one quiver, because (potentially) each variable changes at each step.  Let $Q$ be period one as before, but now take a $Y$-seed $(\y,Q)$ where
\begin{displaymath}
\y = (y_1^{(m)}, y_2^{(m-1)}, \ldots, y_m^{(1)}).
\end{displaymath}
The superscripts indicate how many steps have passed since the corresponding vertex of the quiver was mutated.  The subscripts behave as for the $x$-variables: when a vertex is mutated the corresponding variable has its subscript increased by $m$.  For example, the first mutation is
\begin{displaymath}
\mu_1(Q, (y_1^{(m)}, y_2^{(m-1)}, \ldots, y_m^{(1))}) = \{\rho(Q), (y_{m+1}^{(1)}, y_2^{(m)}, y_3^{(m-1)}, \ldots, y_m^{(2)})
\end{displaymath}
where $y_{m+1}^{(1)} = 1/y_1^{(m)}$ and
\begin{displaymath}
y_j^{(k+1)} = y_j^{(k)}\left(\displaystyle\prod_{j \longrightarrow 1} (1+y_1^{(m)})\right)\left(\displaystyle\prod_{1 \longrightarrow j} (1+(y_1^{(m)})^{-1})\right)^{-1}
\end{displaymath}
for all $j=2,\ldots, m$ and $k=m-j+1$.  As before, the period one property ensures that variables continue to transform in this manner.  The result is a family $y_j^{(k)}$ with $j \in \mathbb{Z}$ and $k=1,\ldots, m$ subject to
\begin{equation}
y_{j+m}^{(1)} = (y_j^{(m)})^{-1}
\end{equation}
and
\begin{equation} \label{eqystep}
y_{j}^{(k+1)} = y_j^{(k)}\frac{\prod_{(m-k+1) \rightarrow 1} (1+y_{j-m+k}^{(m)})}{\prod_{1 \rightarrow (m-k+1)} (1+(y_{j-m+k}^{(m)})^{-1})}
\end{equation}
for $k=1,2,\ldots, m-1$ where the products are over arrows in the original quiver $Q$.

It is not hard to see from the above that all $y_j^{(k)}$ can be determined by the $y_j^{(m)}$, i.e. the variables at vertices that are about to be mutated.  As such we define $y_j = y_j^{(m)}$ and throw out the rest of the variables.  

\begin{prop}
Let $Q$ be a period one quiver and define $y_j$ for $j \in \mathbb{Z}$ as above.  Then
\begin{displaymath}
y_{j+m}y_j = \frac{\prod_{(k+1) \rightarrow 1}(1+y_{j+m-k})}{\prod_{1 \rightarrow (k+1)}(1+y_{j+m-k}^{-1})}
\end{displaymath}
\end{prop}

\begin{proof}
Applying \eqref{eqystep} recursively yields
\begin{align*}
y_{j+m} = y_{j+m}^{(m)} &= y_{j+m}^{(1)}\prod_{k=1}^{m-1}\left(\frac{\prod_{(m-k+1) \rightarrow 1} (1+y_{j+k})}{\prod_{1 \rightarrow (m-k+1)} (1+y_{j+k}^{-1})}\right) \\
&= y_{j+m}^{(1)}\frac{\prod_{(k+1) \rightarrow 1} (1+y_{j+m-k})}{\prod_{1 \rightarrow (k+1)} (1+y_{j+m-k}^{-1})}.
\end{align*}
The result follows since $y_j = 1/y_{j+m}^{(1)}$.
\end{proof}

We now generalize the notion of period one quivers in order to model certain two dimensional recurrences.  For this purpose, it is convenient to allow infinite quivers.   We will require them however to be {\it {locally finite}}, meaning that each vertex is connected to only finitely many other vertices.  Under this assumption, mutations can be defined as before.  In fact, if $U$ is an infinite set of vertices no two of which are connected by an arrow, then we can define a simultaneous mutation $\mu_U$ at the vertices of $U$ thanks to the following property.

\begin{prop}
 Assume vertices $u$ and $v$ in $Q$ are not connected by an edge. Then the seed mutations (or $Y$-seed mutations) $\mu_u$ and $\mu_v$ commute. 
\end{prop}

Now fix vertex set $V = \mathbb{Z} \times \{0,\ldots, l-1\}$.  Let $Q$ be a quiver on this set and let $\mu_{*,j}$ denote mutation at the set $\{(i,j) : i \in \mathbb{Z}\}$ when defined.  Say that $Q$ is \emph{period one} if 
\begin{itemize}
\item $Q$ is invariant under the translation $(i,j) \mapsto (i+1,j)$ of its vertices,
\item no two vertices $\{(i,0)\}$ are connected by an arrow, and
\item $\mu_{*,0}(Q) = \rho(Q)$ where $\rho$ denotes the permutation of the vertices 
\begin{displaymath}
\rho(i,j) = \begin{cases}
(i,j+1), & 0 \leq j < l-1\\
(i+i_0,1), & j=l-1
\end{cases}
\end{displaymath}
for some fixed shift $i_0 \in \mathbb{Z}$.  
\end{itemize}

Similarly to the one dimensional case, it is natural to consider the mutation sequence $\mu_{*,0}, \mu_{*,1}, \ldots, \mu_{*,l-1}, \mu_{*,0}, \ldots$.  Given an initial seed $(Q,\x)$ we can define $x_{i+i_0,j+l}$ to be the variable that $x_{i,j}$ transforms to during the $j$th mutation in this sequence.  Alternately, for an initial $Y$-seed $(Q, \y)$ with
\begin{displaymath}
\y = \{y_{i,j}^{(l-j)} : i \in \mathbb{Z}, j=0,\ldots, l-1\}
\end{displaymath}
say at each step that $y_{i,j}^{(k-1)}$ transforms to $y_{i,j}^{(k)}$ for $1 < k \leq l$ and $y_{i,j}^{(l)}$ transforms to $y_{i+i_0,j+l}^{(1)}$.  Let $y_{i,j} = y_{i,j}^{(l)}$.  Arguing exactly as before, the $x_{i,j}$ and $y_{i,j}$ satisfy recurrences as follows.

\begin{prop} \label{propexchange}
\begin{align}
\label{eqexchangex} 
x_{u+(i_0,l)}x_{u} &= \prod_{(0,0) \rightarrow v}x_{u+v} + \prod_{v \rightarrow (0,0)}x_{u+v} \\
\label{eqexchangey}
y_{u+(i_0,l)}y_u &= \frac{\prod_{v \rightarrow (0,0)}(1+y_{u+(i_0,l)-v})}{\prod_{(0,0) \rightarrow v}(1+y_{u+(i_0,l)-v}^{-1})}
\end{align}
\end{prop}

In order to match \eqref{eqmain} and \eqref{eqexchangey}, we will need a period one quiver with prescribed arrows in and out of $(0,0)$.  The following Proposition guarantees the existence of such a quiver, under certain conditions on these arrows which is satisfied in our case.  

\begin{prop} \label{propQexists}
Let $m_v$ be a collection of integers indexed by $v \in \mathbb{Z}^2$ with $1 \leq v_2 \leq l-1$ such that all but finitely many $m_v$ are zero and $m_v = m_{(i_0,l)-v}$ for all $v$.  Then there exists a period one quiver $Q$ such that for all $v$, there are $|m_v|$ arrows from $(0,0)$ to $v$ (resp. from $v$ to $(0,0))$ if $m_v \geq 0$ (resp. $m_v \leq 0$).
\end{prop}

\begin{ex}
Let $(i_0,l)=(0,3)$.  Let $m_{1,1}=m_{-1,2} = 1$, $m_{-1,1}=m_{1,2}=-1$, and $m_v=0$ for all other $v$.  Figure \ref{figquiver} provides the corresponding periodic quiver.  We will see that this quiver relates to the $S$-map for $S=\{(-1,1),(1,1),(0,2),(0,3)\}$ (see Example \ref{exKS}).
\end{ex}

\begin{proof}
 Let $\epsilon_{v,w} = \frac{1}{2}(|m_w|m_v-m_w|m_v|)$. Construct the quiver as follows:
 \begin{itemize}
  \item For any $0 \leq r \leq l-1-v_2$ and $k \in \mathbb Z$, add $|m_v|$ arrows from $(k,r)$ to $(k,r)+v$ (resp. from $(k,r)+v$ to $(k,r))$ if $m_v \geq 0$ (resp. $m_v \leq 0$).
  \item For any pair $v,w$ such that $\epsilon_{v,w} > 0$ and $v_2 + w_2 \leq l-1$, add $\epsilon_{v,w}$ arrows from $(k,r)+v$ to $(k,r)+w$ for all $0 \leq r \leq l-1-v_2-w_2$ and 
  $k \in \mathbb Z$.
  \item if some of the resulting arrows cancel out, cancel them out, so that the quiver does not have any $2$-cycles.
 \end{itemize}
We claim that the result is the desired period one quiver. It is clear it is symmetric under shift $(i,j) \mapsto (i+1,j)$. 
Perform the mutation $\mu_{*,0}$ and rename each vertex $(i,0)$ as $(i+i_0,l)$.  We want to show the result is a shift of the quiver by $(i,j) \mapsto (i,j+1)$.  Here is the complete list of changes to the quiver caused by this mutation. 
\begin{itemize}
 \item $|m_v|$ arrows between $(k,0)$ and $(k,0)+v$ become $|m_v|$ arrows between $(k+i_0,l)-v$ and $(k+i_0,l)$.
 \item For $v,w$ such that $\epsilon_{v,w} > 0$ and $v_2 + w_2 \leq l-1$ total of $\epsilon_{v,w}$ arrows between $(k,0)+w$ and $(k,0)+v$ are created, canceling the same number of arrows in 
 the opposite direction. 
 \item For $v,w$ such that $\epsilon_{v,w} > 0$ and $v_2 + w_2 \leq l-1$ total of $\epsilon_{v,w}$ arrows between $(k+i_0,l)-w$ and $(k+i_0,l)-v$ are created. 
\end{itemize}
It is easy to see that this essentially shifts the quiver described above by $(i,j) \mapsto (i,j+1)$. This completes the proof. 
\end{proof}

\section{The cluster description of $S$-maps} \label{seccluster}
This section contains the proof of Theorem \ref{thmcluster}.  The first part of the Theorem, namely the verification of \eqref{eqmain}, is a purely geometric result.

Fix points $P_1,P_2$ and lines $L_1,L_2$ in the plane.  By way of notation, let
\begin{displaymath}
[P_1,L_1,P_2,L_2] = [P_1, \meet{\join{P_1}{P_2}}{L_1}, P_2, \meet{\join{P_1}{P_2}}{L_2}].
\end{displaymath}
Using similar triangles, it is easy to show
\begin{displaymath}
[P_1,L_1,P_2,L_2] = \frac{d(P_1,L_1)}{d(P_2,L_1)}\frac{d(P_2,L_2)}{d(P_1,L_2)},
\end{displaymath}
where $d(P_i,L_j)$ denotes the signed perpendicular distance from $P_i$ to $L_j$.  The signs are chosen so that, for example
\begin{displaymath}
\frac{d(P_1,L_1)}{d(P_2,L_1)}
\end{displaymath}
is positive if and only if $P_1,P_2$ lie on the same side of $L_1$.

\begin{lem} \label{lemmain}
Let $P_1, P_2, P_3, P_4$ be points and $L_1, L_2, L_3, L_4$ lines, all in the projective plane.  For $1 \leq i < k \leq 4$, let
\begin{displaymath}
y_{i,k} = [P_i, L_j, P_k, L_l]
\end{displaymath}
where $\{j,l\} = \{1,2,3,4\} \setminus \{i,k\}$ are chosen so that $i,j,k,l$ is an even permutation.  Then
\begin{displaymath}
\prod_{1 \leq i < k \leq 4} y_{i,j} = 1
\end{displaymath}
\end{lem}

\begin{proof}
By the above,
\begin{displaymath}
y_{i,k} = \frac{d(P_i,L_j)}{d(P_k,L_j)}\frac{d(P_k,L_l)}{d(P_i,L_l)}
\end{displaymath}
In $\prod_{1 \leq i < k \leq 4} y_{i,j}$ there are twelve distinct factors, namely the distances $d(P_i,L_j)$ for $i \neq j$.  Each factor appears twice, once in a numerator and once in a denominator, and hence cancels out.  
\end{proof}

When $D \geq 3$, a more general version of Lemma \ref{lemmain} will be needed.  Let $P_1,\ldots, P_4$ be points in any projective space and $L_1,\ldots, L_4$ lines so that each line $L_i$ lies in the plane spanned by the $P_j$ with $j \neq i$.  Then the $y_{i,k}$ are still defined.  The result of Lemma \ref{lemmain} continues to hold via an identical argument.

\begin{prop} \label{propmain}
The $y_r = y_r(P)$ (see \eqref{eqdefy}) satisfy
\begin{displaymath} %\label{eqmain}
y_{r+a+b}y_{r+c+d} = \frac{(1+y_{r+a+c})(1+y_{r+b+d})}{(1+y_{r+a+d}^{-1})(1+y_{r+b+c}^{-1})}
\end{displaymath}
\end{prop}

\begin{proof}
Everything is translation invariant, so it suffices to consider the $r=0$ case, namely
\begin{align*}
&[P_{2a+b}, P_{a+b+c}, P_{a+2b}, P_{a+b+d}][P_{a+c+d}, P_{2c+d}, P_{b+c+d}, P_{c+2d}] \\
&= \frac{(1-[P_{2a+c}, P_{a+2c}, P_{a+b+c}, P_{a+c+d}])(1-[P_{a+b+d}, P_{b+c+d}, P_{2b+d}, P_{b+2d}])} {(1-[P_{2a+d}, P_{a+c+d}, P_{a+b+d}, P_{a+2d}]^{-1})(1-[P_{a+b+c}, P_{b+2c}, P_{2b+c}, P_{b+c+d}]^{-1})}
\end{align*}
In general
\begin{displaymath}
1 - [x_1,x_2,x_3,x_4] = [x_1,x_3,x_2,x_4]
\end{displaymath}
and
\begin{displaymath}
[x_1,x_2,x_3,x_4]^{-1} = [x_1,x_4,x_3,x_2].
\end{displaymath}
Applying these repeatedly, the desired identity can be expressed in the form that a product of six cross ratios equals 1.  These cross ratios are
\begin{align*}
&[P_{2a+b}, P_{a+b+c}, P_{a+2b}, P_{a+b+d}] \\
&[P_{a+c+d}, P_{2c+d}, P_{b+c+d}, P_{c+2d}] \\
&[P_{2a+c}, P_{a+c+d}, P_{a+2c}, P_{a+b+c}] \\
&[P_{a+b+d}, P_{b+2d}, P_{b+c+d}, P_{2b+d}] \\
&[P_{2a+d}, P_{a+b+d}, P_{a+2d}, P_{a+c+d}] \\
&[P_{a+b+c}, P_{2b+c}, P_{b+c+d}, P_{b+2c}]
\end{align*}
Also $[x_1,x_2,x_3,x_4] = [x_2,x_1,x_4,x_3] = [x_3,x_4,x_1,x_2] = [x_4,x_3,x_2,x_1]$ so the six cross ratios can be rewritten as
\begin{align*}
&[P_{a+b+c}, P_{2a+b}, P_{a+b+d}, P_{a+2b}] \\
&[P_{a+c+d}, P_{2c+d}, P_{b+c+d}, P_{c+2d}] \\
&[P_{a+b+c}, P_{a+2c}, P_{a+c+d},  P_{2a+c}] \\
&[P_{a+b+d}, P_{b+2d}, P_{b+c+d}, P_{2b+d}] \\
&[P_{a+b+d}, P_{2a+d}, P_{a+c+d},  P_{a+2d}] \\
&[P_{a+b+c}, P_{2b+c}, P_{b+c+d}, P_{b+2c}]
\end{align*}

Let $L_r$ denote the common line containing $P_{r+a},P_{r+b},P_{r+c},P_{r+d}$.  Then the fact that these cross ratios have product 1 amounts to Lemma \ref{lemmain} applied to the points
\begin{displaymath}
P_{a+b+c}, P_{a+b+d}, P_{a+c+d}, P_{b+c+d}
\end{displaymath}
and the lines
\begin{displaymath}
L_{2d}, L_{2c}, L_{2b}, L_{2a}.
\end{displaymath}
\end{proof}

\begin{proof}[Proof of Theorem \ref{thmcluster}]
Let $S= \{a,b,c,d\}$ be a $Y$-pin.  We have just finished proving \eqref{eqmain}. 
We build the corresponding quiver $Q_S$ using the results of Section \ref{secquiver}.  Let $(i_0,l)=c+d-(a+b)$.  
Define $m_v$ for $v \in \mathbb{Z}^2$, $0 < v_2 < l$ as follows.  Let $m_{c-a}=m_{d-b}=1$, $m_{c-b}= m_{d-a} = -1$, and $m_v=0$ for all other $v$.  By Proposition \ref{propQexists}, there is a period one quiver $Q_S$ on vertex set $\mathbb{Z} \times \{0,1,\ldots, l-1\}$ such that $(0,0)$ has outgoing arrows to $c-a$ and $d-b$, and incoming arrows from $c-b$ and $d-a$.  By Proposition \ref{propexchange}, the periodic mutations of this quiver produce $y$-variables $y_{i,j}$ satisfying
\begin{displaymath}
y_{u+(c+d)-(a+b)}y_u = \frac{(1+y_{u+d-a})(1+y_{u+c-b})}{(1+y_{u+d-b}^{-1})(1+y_{u+c-a}^{-1})}.
\end{displaymath}
which after shifting indices matches \eqref{eqmain}.
\end{proof}

The quivers $Q_S$ are infinite and periodic modulo translation by $(1,0)$ (one unit to the right).  Let $Q_{n,S}$ be the quotient of $Q_S$ by $(n,0)$.  Then $Q_{n,S}$ is a finite quiver with $mn$ vertices where $m=c_2+d_2-a_2-b_2$.  Mutations of this quiver model the finite dimensional system \eqref{eqmain} subject to $y_{i+n,j} = y_{i,j}$.  Geometrically, this system amounts to the restriction of the $S$-map to families of twisted polygons.  

 \section{One dimensional $Y$-meshes} \label{sec1d}
According to Definition \ref{defmesh}, a $Y$-mesh is a grid of points satisfying certain collinearity conditions.  The definition is only reasonable in dimension at least two, as in one dimension these conditions would be vacuous.  There is an alternate definition in one dimension which we explore in this section.

Suppose $k \geq 2$, $P_1,\ldots, P_{2k}$ are points in a projective space, and the triples 
\begin{displaymath}
\{P_1,P_2,P_3\}, \{P_3, P_4, P_5\}, \ldots, \{P_{2k-1}, P_{2k}, P_1\}
\end{displaymath}
are all collinear.  The \emph{multi-ratio} of the points is
\begin{displaymath}
[P_1, P_2, \ldots, P_{2k}] = \left(\frac{P_1P_2}{P_2P_3}\right)\left(\frac{P_3P_4}{P_4P_5}\right)\left(\frac{P_{2k-1}P_{2k}}{P_{2k}P_1}\right)
\end{displaymath}
where $P_iP_j$ denotes the signed distance from $P_i$ to $P_j$.  The multi-ratio is always a projective invariant.  The $k=2$ case is the cross ratio, and when $k=3$ the multi-ratio is sometimes called a \emph{triple ratio}.

\begin{prop}
Let $S$ be a $Y$-pin and $P$ a $Y$-mesh of type $S$ and dimension $D \geq 2$.  Then
\begin{displaymath}
[P_{r+a+d}, P_{r+a+c}, P_{r+a+b}, P_{r+b+c}, P_{r+b+d}, P_{r+c+d}] = -1
\end{displaymath}
for all $r \in \mathbb{Z}$.
\end{prop}

\begin{proof}
By definition of a $Y$-mesh, the triples of points 
\begin{align*}
&\{P_{r+a+d}, P_{r+a+c}, P_{r+a+b}\}, \{P_{r+a+b}, P_{r+b+c}, P_{r+b+d}\}, \{P_{r+b+d}, P_{r+c+d}, P_{r+a+d}\}, \\
&\{P_{r+a+c}, P_{r+b+c}, P_{r+c+d}\}
\end{align*} are collinear.  The result follows from Menelaus' Theorem.
\end{proof}

The triple ratio condition can be used to define one dimensional $Y$-meshes.

\begin{defin} \label{defmesh1}
Fix a $Y$-pin $S$.  A \emph{$Y$-mesh} in dimension one is a grid of points $P_{i,j} \in \mathbb{RP}^1$ indexed by $i,j \in \mathbb{Z}$ satisfying
\begin{equation} \label{eqtriple}
[P_{r+a+d}, P_{r+a+c}, P_{r+a+b}, P_{r+b+c}, P_{r+b+d}, P_{r+c+d}] = -1
\end{equation}
for all $r \in \mathbb{Z}^2$.  
\end{defin}

Let $m=c_2+d_2-a_2-b_2$.  Then the points involved in the relation \eqref{eqtriple} are contained within $m+1$ consecutive rows of the grid.  Moreover, we can solve for $P_{r+c+d}$ and determine it from the other five points.  The analogue of an $S$-map, then, is a function $T_{1,S} : \mathcal{U}_{1,m} \to \mathcal{U}_{1,m}$ taking $(A^{(1)}, \ldots, A^{(m)}) \mapsto (A^{(2)}, \ldots, A^{(m+1)})$ where the $A^{(m+1)}_i = P_{i,m+1}$ are defined so that \eqref{eqtriple} holds for all $r$ with $r_2+c_2+d_2=m+1$.

Geometrically speaking, one dimensional $Y$-meshes seem somewhat different from their higher dimensional analogues.  The real justification for giving both objects the same name is that they both exhibit the same cluster dynamics given by \eqref{eqmain}.  

\begin{prop}
Let $P$ be a $Y$-mesh of dimension one.  As in the higher dimensional case, define $y_r$ for $r \in \mathbb{Z}^2$ by $y_r =-[P_{r+a}, P_{r+c}, P_{r+b}, P_{r+d}]$.  Then
\begin{equation} \label{eqmain1d}
y_{r+a+b}y_{r+c+d} = \frac{(1+y_{r+a+c})(1+y_{r+b+d})}{(1+y_{r+a+d}^{-1})(1+y_{r+b+c}^{-1})}
\end{equation}  
for all $r \in \mathbb{Z}^2$.
\end{prop}

\begin{proof}
Assume without loss of generality that $r=0$.  One can check for $x_1,\ldots, x_6 \in \mathbb{RP}^1$, that the condition $[x_1,\ldots, x_6]=-1$ is invariant under any permutation $\pi \in S_6$ satisfying
\begin{displaymath}
|\pi(4)-\pi(1)| = |\pi(5)-\pi(2)| = |\pi(6)-\pi(3)| = 3.
\end{displaymath}
As such, each relation \eqref{eqtriple} can be written in many different ways.  Making a particular choice for each of $r=a,b,c,d$ yields
\begin{align*}
[P_{2a+d}, P_{a+b+d}, P_{2a+b}, P_{a+b+c}, P_{2a+c}, P_{a+c+d}] &= -1 \\
[P_{2b+c}, P_{a+b+c}, P_{a+2b}, P_{a+b+d}, P_{2b+d}, P_{b+c+d}] &= -1 \\
[P_{a+2c}, P_{a+b+c}, P_{b+2c}, P_{b+c+d}, P_{2c+d}, P_{a+c+d}] &= -1 \\
[P_{a+2d}, P_{a+c+d}, P_{c+2d}, P_{b+c+d}, P_{b+2d}, P_{a+b+d}] &= -1
\end{align*}
so the product of these four triple ratios equals 1.  Regrouping the factors, this product of four triple ratios equals a product of six cross ratios, specifically the six cross ratios given in the proof of Proposition \ref{propmain}.  By said proof, this product equaling 1 is equivalent to \eqref{eqmain1d}.
\end{proof}

\section{Quivers on tori} \label{seclift}
We show that the quiver $Q_S$ for any $Y$-pin $S$ can be embedded on a cylinder.  Therefore the finite quivers $Q_{n,S}$ can be embedded on a torus.  Special cases of these results were proven by Gekhtman, Shapiro, Tabachnikov and Vainshtein \cite{GSTV}.

\begin{rmk}
The idea of this section will be to ``unfold'' $Q_{n,S}$ into certain infinite quivers embedded in the plane.  These infinite quivers, and indeed the unfolding process itself, have been studied by many others including R. Eager \cite{E}, S. Franco et. al \cite{FHMSVW}, and I. Jeong, G. Musiker, and S. Zhang \cite{JMZ}.  For completeness we give an independent presentation, but many of the ideas we use are similar to ones from these other works.
\end{rmk}

Let $S = \{a,b,c,d\}$ be a $Y$-pin.  Recall the quiver $Q_S$ is constructed by applying Proposition \ref{propQexists} with $(i_0,l) = c+d-(a+b)$, $m_{c-a}=m_{d-b} = 1$, and $m_{c-b}=m_{d-a}=-1$.  By the proof of Proposition \ref{propQexists}, the arrows of the quiver will come in six types, which we assign compass direction names (see Table \ref{tabcompass}) for future use.

\begin{table}
\begin{tabular}{l|l|l|l}
Name & Source & Target & Displacement  \\
\hline
North & $(i,j)+d-a$ & $(i,j)$ & $a-d$ \\
South & $(i,j)+c-b$ & $(i,j)$ & $b-c$ \\
East & $(i,j)$ & $(i,j)+d-b$ & $d-b$ \\
West & $(i,j)$ & $(i,j)+c-a$ & $c-a$ \\
Northeast & $(i,j)+c-a$ & $(i,j)+c-b$ & $a-b$ \\
Northwest & $(i,j)+d-b$ & $(i,j)+c-b$ & $c-d$ \\
\end{tabular}
\caption{The six types of arrows appearing in $Q_S$.}
\label{tabcompass}
\end{table}  

Specifically, the four main compass direction arrows are drawn wherever possible within the vertex set $V = \mathbb{Z} \times \{0,1,\ldots, l-1\}$.  A Northeast arrow from $(i,j)+c-a$ to $(i,j)+c-b$ is only drawn if it completes two oriented triangles, one with $(i,j)$ and one with $(i,j)+2c-a-b$.  A Northwest arrow from $(i,j)+d-b$ to $(i,j)+c-b$ is only drawn if it completes two oriented triangles, one with $(i,j)$ and one with $(i,j)+c+d-2b$.  Figure \ref{figCompass} shows the five types of arrows in the quiver $Q_S$ for $S=\{(-1,1),(1,1),(0,2),(0,3)\}$.  The Northwest arrow from $d$ to $c$ is not drawn because it never actually appears in the quiver.  The quiver itself appears in Figure \ref{figquiver}.

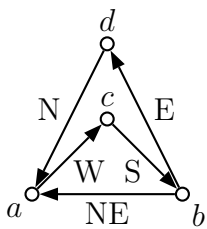
\begin{figure}
\begin{pspicture}(4,4)
%\pnode(1,1){a}
%\pnode(3,1){b}
%\pnode(2,2){c}
%\pnode(2,3){d}
%\psdots(a)(b)(c)(d)
\cnode(1,1){.1}{a}
\cnode(3,1){.1}{b}
\cnode(2,2){.1}{c}
\cnode(2,3){.1}{d}
\uput[dl](a){$a$}
\uput[dr](b){$b$}
\uput[u](c){$c$}
\uput[u](d){$d$}

\psset{arrowsize=5pt}
\psset{arrowinset=0}
\ncline{<-}{b}{c}
\Aput[1pt]{S}
\ncline{<-}{a}{d}
\Aput[2pt]{N}
\ncline{<-}{c}{a}
\Aput[1pt]{W}
\ncline{<-}{d}{b}
\Aput[2pt]{E}
\ncline{<-}{a}{b}
\Bput[2pt]{NE}
%\ncline{->}{c}{d}
\end{pspicture}
\caption{The five arrow types in $Q_S$ for $S=\{(-1,1),(1,1),(0,2),(0,3)\}$.}
\label{figCompass}
\end{figure}

%Identify $V = \mathbb{Z} \times \{0,1,\ldots, l-1\}$ with the quotient $\mathbb{Z}^2 / (c+d-(a+b))$ (recall $l = c_2+d_2-(a_2+b_2)$).  Define a map $\phi: \mathbb{Z}^2 \to V$ by $\phi(i,j) = (a-c)i + (c-b)j$ reduces modulo $c+d-(a+b)$.  Note that $\phi(0,1) = c-b$, $\phi(0,-1) = b-c = $

The compass directions are used to unfold $Q_S$ to an infinite quiver $\tilde{Q}_S$ on vertex set $\mathbb{Z}^2$.  Let $I$ be the principal ideal $I=(c+d-a-b) \subseteq \mathbb{Z}^2$.  Note that $V$ is a set of coset representatives for $I$ as $l = c_2+d_2-a_2-b_2$.  Define $\phi : \mathbb{Z}^2 \to V$ so that 
\begin{displaymath}
\phi(i,j) \equiv (d-b)i + (a-d)j \pmod{I}
\end{displaymath}
for all $i,j \in \mathbb{Z}$.

\begin{lem} \label{lemlift}
Let $\alpha$ be an arrow of $Q_S$ from $u$ to $v$.
\begin{itemize}
\item Suppose $\tilde{u} \in \phi^{-1}(u)$ and $\tilde{\alpha}$ is the arrow starting at $\tilde{u}$ with displacement dictated by the type of $\alpha$ (e.g. if $\alpha$ is North then $\tilde{\alpha}$ has displacement $(0,1)$, if $\alpha$ is Northwest then $\tilde{\alpha}$ has displacement $(-1,1)$).  Then the target $\tilde{v}$ of $\tilde{\alpha}$ satisfies $\phi(\tilde{v}) = v$.
\item Suppose $\tilde{v} \in \phi^{-1}(v)$ and $\tilde{\alpha}$ is the arrow with target $\tilde{v}$ and displacement dictated by the type of $\alpha$.  Then the source $\tilde{u}$ of $\tilde{\alpha}$ satisfies $\phi(\tilde{u}) = u$.
\end{itemize}
\end{lem}

\begin{proof}
The different compass directions are handled case by case.  Suppose $\alpha$ is an East arrow, meaning $v = u+d-b$.  If $\tilde{v} = \tilde{u} + (1,0)$ then clearly
\begin{align*}
&u \equiv \tilde{u}_1(d-b) + \tilde{u}_2(a-d) \pmod{I} \\
&\Longleftrightarrow v \equiv \tilde{v}_1(d-b) + \tilde{v}_2(a-d) \pmod{I}
\end{align*}
so $\phi(\tilde{u}) = u \Leftrightarrow \phi(\tilde{v}) = v$.  Now suppose $\alpha$ is a West arrow, so $v = u + c-a$.  If $\tilde{v} = \tilde{u} + (-1,0)$ then
\begin{align*}
&u \equiv \tilde{u}_1(d-b) + \tilde{u}_2(a-d) \pmod{I} \\
&\Longleftrightarrow v \equiv \tilde{u}_1(d-b) + \tilde{u}_2(a-d) + (c-a) \pmod{I} \\
&\Longleftrightarrow v \equiv \tilde{u}_1(d-b) + \tilde{u}_2(a-d) + (b-d) \pmod{I} \\
&\Longleftrightarrow v \equiv \tilde{v}_1(d-b) + \tilde{v}_2(a-d) \pmod{I}
\end{align*}
The other cases are similar.
\end{proof}

In the context of Lemma \ref{lemlift}, say that $\tilde{\alpha}$ is a lift of $\alpha$.  The unfolded quiver $\tilde{Q}_S$ is defined to be the vertex set $\mathbb{Z}^2$ together with all lifts of all arrows of $Q_S$.  Figure \ref{figlift} shows the unfolding of the quiver $Q_S$ from Figure \ref{figquiver}.  

\begin{figure}
\begin{pspicture}(-1,-1)(7,7)
\cnode(1,1){.1}{v11}
\cnode(2,1){.1}{v21}
\cnode(3,1){.1}{v31}
\cnode(4,1){.1}{v41}
\cnode(5,1){.1}{v51}
\cnode(1,2){.1}{v12}
\cnode(2,2){.1}{v22}
\cnode(3,2){.1}{v32}
\cnode(4,2){.1}{v42}
\cnode(5,2){.1}{v52}
\cnode(1,3){.1}{v13}
\cnode(2,3){.1}{v23}
\cnode(3,3){.1}{v33}
\cnode(4,3){.1}{v43}
\cnode(5,3){.1}{v53}
\cnode(1,4){.1}{v14}
\cnode(2,4){.1}{v24}
\cnode(3,4){.1}{v34}
\cnode(4,4){.1}{v44}
\cnode(5,4){.1}{v54}
\cnode(1,5){.1}{v15}
\cnode(2,5){.1}{v25}
\cnode(3,5){.1}{v35}
\cnode(4,5){.1}{v45}
\cnode(5,5){.1}{v55}

\uput[l](v11){$4$}
\uput[d](v21){$3''$}
\uput[d](v31){$2'$}
\uput[d](v41){$1$}
\uput[r](v51){$0''$}
\uput[l](v12){$3'$}
\uput[dl](v22){$2$}
\uput[dl](v32){$1''$}
\uput[dr](v42){$0'$}
\uput[r](v52){$-1$}
\uput[l](v13){$2''$}
\uput[dr](v23){$1'$}
\uput[dr](v33){$0$}
\uput[ur](v43){$-1''$}
\uput[r](v53){$-2'$}
\uput[l](v14){$1$}
\uput[ul](v24){$0''$}
\uput[ul](v34){$-1'$}
\uput[ur](v44){$-2$}
\uput[r](v54){$-3''$}
\uput[l](v15){$0'$}
\uput[u](v25){$-1$}
\uput[u](v35){$-2''$}
\uput[u](v45){$-3'$}
\uput[r](v55){$-4$}

\psset{arrowsize=5pt}
\psset{arrowinset=0}
\ncline{<-}{v11}{v12}
\ncline{<-}{v12}{v13}
\ncline{->}{v13}{v14}
\ncline{<-}{v14}{v15}
\ncline{->}{v21}{v22}
\ncline{<-}{v22}{v23}
\ncline{<-}{v23}{v24}
\ncline{->}{v24}{v25}
\ncline{<-}{v31}{v32}
\ncline{->}{v32}{v33}
\ncline{<-}{v33}{v34}
\ncline{<-}{v34}{v35}
\ncline{<-}{v41}{v42}
\ncline{<-}{v42}{v43}
\ncline{->}{v43}{v44}
\ncline{<-}{v44}{v45}
\ncline{->}{v51}{v52}
\ncline{<-}{v52}{v53}
\ncline{<-}{v53}{v54}
\ncline{->}{v54}{v55}

\ncline{->}{v11}{v21}
\ncline{<-}{v21}{v31}
\ncline{<-}{v31}{v41}
\ncline{->}{v41}{v51}
\ncline{<-}{v12}{v22}
\ncline{->}{v22}{v32}
\ncline{<-}{v32}{v42}
\ncline{<-}{v42}{v52}
\ncline{<-}{v13}{v23}
\ncline{<-}{v23}{v33}
\ncline{->}{v33}{v43}
\ncline{<-}{v43}{v53}
\ncline{->}{v14}{v24}
\ncline{<-}{v24}{v34}
\ncline{<-}{v34}{v44}
\ncline{->}{v44}{v54}
\ncline{<-}{v15}{v25}
\ncline{->}{v25}{v35}
\ncline{<-}{v35}{v45}
\ncline{<-}{v45}{v55}

\ncline{<-}{v23}{v12}
\ncline{<-}{v34}{v23}
\ncline{<-}{v45}{v34}
\ncline{<-}{v42}{v31}
\ncline{<-}{v53}{v42}

\psdots(5.5,5.5)(6,6)(6.5,6.5)
\psdots(5.5,0.5)(6,0)(6.5,-0.5)
\psdots(0.5,5.5)(0,6)(-0.5,6.5)
\psdots(0.5,0.5)(0,0)(-0.5,-0.5)

\end{pspicture}
\caption{The lift $\tilde{Q}_S$ of the quiver $Q_S$ in Figure \ref{figquiver}.  Each vertex $v$ is labeled by $\phi(v)$ using the identification $i=(i,0)$, $i'=(i,1)$, $i''=(i,2)$.}
\label{figlift}
\end{figure}
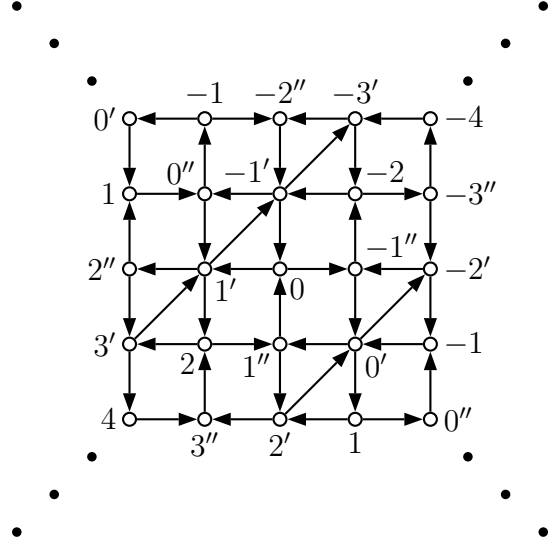

By Lemma \ref{lemlift}, the quiver $\tilde{Q}_S$ looks locally near each $(i,j) \in \mathbb{Z}^2$ like $Q_S$ does around $\phi(i,j)$.  To get a global picture of the relationship between the quivers, it is necessary to better understand the function $\phi$.

\begin{lem}
Let $S$ be a $Y$-pin and define $\phi: \mathbb{Z}^2 \to V$ as above.  Then $\phi$ is surjective.  Moreover, $\tilde{I} = \phi^{-1}(0,0) \subseteq \mathbb{Z}^2$ is a nonzero principal ideal and $\phi(i,j) = \phi(i',j')$ if and only if $(i',j') - (i,j) \in \tilde{I}$.
\end{lem}

\begin{proof}
Note that $\phi$ is the composition of the linear map
\begin{displaymath}
(i,j) \mapsto (d-b)i + (a-d)j
\end{displaymath}
followed by reduction modulo $I$.  It follows that $\tilde{I}$ is an ideal and that the nonempty inverse images $\phi^{-1}(v)$ for $v \in V$ are its cosets.  An expression $\phi(i,j) = (0,0)$ is equivalent to
\begin{displaymath}
(d-b)i + (a-d)j = k(c+d-a-b)
\end{displaymath}
for some $k \in \mathbb{Z}$ which can be rewritten as
\begin{equation} \label{eqkernel}
(j+k)a + (-i+k)b + (-k)c + (i-j-k)d = 0.
\end{equation}
Note the coefficients sum to 0 so \eqref{eqkernel} is a convex relation on $S$ (see Section \ref{seclattice}).  It is easy to see any such relation in turn comes from a unique $(i,j) \in \phi^{-1}(0,0)$.  In particular, a minimal convex relation comes from a nonzero generator of $\tilde{I}$.

For general $v \in V$, we get by a similar calculation that $\phi(i,j)= v$ if and only if there exists $k \in \mathbb{Z}$ with
\begin{displaymath}
(j+k)a + (-i+k)b + (-k)c + (i-j-k)d = v.
\end{displaymath}
For each $v$ there is a solution in $i,j,k$ since $b-a,c-a,d-a$ generate $\mathbb{Z}^2$.
\end{proof}

\begin{ex}
Let $a=(-1,1), b=(1,1), c=(0,2), d=(0,3)$.  A minimal convex relation is
\begin{displaymath}
a + b - 4c + 2d = 0.
\end{displaymath}
To match \eqref{eqkernel} requires $i=3,j=-3,k=4$.  Therefore $\phi$ is periodic modulo $(i,j) = (3,-3)$ as can be seen in the labeling of Figure \ref{figlift}.  The original quiver $Q_S$ is recovered from $\tilde{Q}_S$ via quotienting by this vector. 
\end{ex}

In general, we can obtain $Q_S$ by taking the quotient of $\tilde{Q}_S$ by $\tilde{I}$.  Since $\tilde{I}$ is generated by a single vector, the result is a drawing of $Q_S$ on the cylinder.  Moreover, this drawing of $Q_S$ is an embedding provided that $\tilde{Q}_S$ is embedded in the plane.  Before proving this last fact, we collect some properties of $\tilde{Q}_S$ that follow directly from properties of $Q_S$.

\begin{prop} \label{propQtilde}
In the quiver $\tilde{Q}_S$:
\begin{enumerate}
\item for each $i,j \in \mathbb{Z}$ there is a single arrow with some orientation connecting $(i,j)$ and $(i+1,j)$,
\item for each $i,j \in \mathbb{Z}$ there is a single arrow with some orientation connecting $(i,j)$ and $(i,j+1)$,
\item there is an arrow from $(i+1,j)$ to $(i,j+1)$ if and only if it breaks the surrounding primitive square into two oriented triangles,
\item there is an arrow from $(i,j)$ to $(i+1,j+1)$ if and only if it breaks the surrounding primitive square into two oriented triangles.
\end{enumerate}
\end{prop}

\begin{prop}
Let $S$ be a $Y$-pin.  Then $\tilde{Q}_S$ is embedded in the plane.  Therefore $Q_S$ can be embedded on a cylinder and $Q_{n,S}$ can be embedded on a torus for all $n$.
\end{prop}
\begin{proof}
The only arrows that could potentially cross would be a Northeast arrow from $(i,j)$ to $(i+1,j+1)$ and a Northwest arrow from $(i+1,j)$ to $(i,j+1)$.  Proposition \ref{propQtilde} implies that these arrows cannot actually both be present.  Indeed, the Northeast arrow can only exist if $(i,j)$, $(i+1,j)$ are joined by a West arrow, whereas the Northwest arrow can only exist if $(i,j)$, $(i+1,j)$ are joined by an East arrow.
\end{proof}

\begin{rmk}
In general, nonequivalent $S$ and $S'$ give rise to different quivers $Q_S$ and $Q_{S'}$, but they may lift to identical $\tilde{Q}_S = \tilde{Q}_{S'}$.  More specifically, the lifted quiver depends only on $p=d_2-b_2$, $q=c_2-b_2$, and $m=c_2+d_2-a_2-b_2$.  There is a simple, direct description of $\tilde{Q}_S$ in terms of these parameters.  Label each $(i,j) \in \mathbb{Z}^2$ by $pi+qj$ reduced modulo $m$ to $\{0,1,\ldots, m-1\}$.  For each $i,j$, draw an arrow connecting $(i,j)$ with $(i+1,j)$ oriented towards the point with the bigger label.  Also for each $i,j$, draw an arrow connecting $(i,j)$ with $(i,j+1)$ oriented towards the point with the smaller label.  Lastly, add in diagonal arrows according to the last two parts of Proposition \ref{propQtilde}.  Up to differing conventions, this matches the construction by Jeong, Musiker, and Zhang \cite{JMZ} of quivers they denote $\tilde{Q}_N^{(r,s)}$.
\end{rmk}

\section{Geometric interpretation of general $y$-variables} \label{secratios}
Let $S$ be a $Y$-pin.  We have seen that the recurrence \eqref{eqmain} is satisfied both by certain $y$-variables under periodic mutations of the quiver $Q_S$ and by certain cross ratios of the points in a $Y$-mesh.  These $y$-variables $y_{i,j} = y_{i,j}^{(m)}$ (where $m=c_2+d_2-a_2-b_2$) are only a subset of the $y_{i,j}^{(k)}$, $1 \leq k \leq m$ that arise under periodic mutation, which in turn represent only a small portion of the full $Y$-pattern.  In this section, we find geometric interpretations for some of the other $y$-variables, including all of the $y_{i,j}^{(k)}$.

In all cases discussed, the $y$-variable ends up being (up to sign) a multi-ratio.  For convenience, define
\begin{displaymath}
[P_1,L_1,\ldots, P_k,L_k] = [P_1, \meet{\join{P_1}{P_2}}{L_1}, P_2, \ldots, P_k, \meet{\join{P_k}{P_1}}{L_k}]
\end{displaymath}
for points $P_1,\ldots, P_k$ and lines $L_1,\ldots, L_k$ with the property that $\join{P_i}{P_{i+1}}$ intersects $L_i$ for all $i$.  As explained in the cross ratio case in Section \ref{seccluster}, such a ratio can be expressed in terms of perpendicular distances as
\begin{displaymath}
[P_1,L_1,\ldots, P_k,L_k] = (-1)^k\frac{d(P_1,L_1)}{d(P_2,L_1)}\frac{d(P_2,L_2)}{d(P_3,L_2)} \cdots \frac{d(P_k,L_k)}{d(P_1,L_k)}
\end{displaymath}

Fix a $Y$-mesh $P$ and let 
\begin{displaymath}
y_r = y_r^{(m)} = -[P_{r+a}, P_{r+c}, P_{r+b}, P_{r+d}]
\end{displaymath}
for all $r \in \mathbb{Z}^2$.

To simplify notation we describe the $y_{c+d}^{(k)}$ for all $k = 1,2,\ldots, m$ as well as some related variables.  General $y_r^{(k)}$ can be founded by shifting subscripts in the formulas that follow.  As explained in Section \ref{secquiver}, all $y_{c+d}^{(k)}$ can be expressed in terms of certain $y_r$.  The two simplest cases are
\begin{align*}
y_{c+d}^{(1)} &= 1/y_{a+b} = -1/[P_{2a+b}, P_{a+b+c}, P_{a+2b}, P_{a+b+d}]\\
&= -[P_{a+b+c}, P_{a+2b}, P_{a+b+d}, P_{2a+b}]
\end{align*}  
and 
\begin{displaymath}
y_{c+d}^{(m)} = -[P_{a+c+d}, P_{2c+d}, P_{b+c+d}, P_{c+2d}].
\end{displaymath}
These two variables are related by \eqref{eqmain} as follows
\begin{displaymath}
y_{c+d}^{(m)} = y_{c+d}^{(1)}\frac{(1+y_{a+c})(1+y_{b+d})}{(1+y_{a+d}^{-1})(1+y_{b+c}^{-1})}
\end{displaymath}
More specifically, starting with $y_{c+d}^{(1)}$ each of the four factors
\begin{displaymath}
1+y_{a+c}, 1+y_{b+d}, \frac{1}{1+y_{a+d}^{-1}}, \frac{1}{1+y_{b+c}^{-1}}
\end{displaymath}
gets multiplied in during one of the next $m-1$ steps ultimately resulting in $y_{c+d}^{(m)}$.  As such, each $y_{c+d}^{(k)}$ equals $y_{c+d}^{(1)}$ multiplied by some subset of these factors.  Which subset corresponds to which $k$ depends on $S$.  Amazingly, whichever of the sixteen subsets is chosen, the result is a multi-ratio of the points of $P$.  

As a first example, 
\begin{align*}
y_{c+d}^{(1)}(1+y_{a+c}) &= -[P_{a+b+c}, P_{a+2b}, P_{a+b+d}, P_{2a+b}](1-[P_{2a+c}, P_{a+2c}, P_{a+b+c}, P_{a+c+d}]) \\
&= -[P_{a+b+c}, P_{a+2b}, P_{a+b+d}, P_{2a+b}][P_{a+c+d}, P_{a+2c}, P_{a+b+c}, P_{2a+c}] \\
&= -[P_{a+b+c}, L_{2b}, P_{a+b+d}, L_{2a}][P_{a+c+d}, L_{2c}, P_{a+b+c}, L_{2a}] \\
&= [P_{a+b+c}, L_{2b}, P_{a+b+d}, L_{2a}, P_{a+c+d}, L_{2c}] \\
&= [P_{a+b+c}, P_{a+2b}, P_{a+b+d}, P_{2a+d}, P_{a+c+d}, P_{a+2c}]
\end{align*}

We lack a simple rule to describe which multi-ratio to take in each case, and it would be a bit cumbersome to list all sixteen formulas.  The following are the only quantities different from $y_{c+d}^{(1)}$ and $y_{c+d}^{(m)}$ that can actually occur among the $y_{c+d}^{(k)}$.

\begin{prop} \label{propothery}
\begin{align*}
&y_{c+d}^{(1)}(1+y_{a+c}) = [P_{a+b+c}, P_{a+2b}, P_{a+b+d}, P_{2a+d}, P_{a+c+d}, P_{a+2c}] \\
&y_{c+d}^{(1)}\frac{1+y_{a+c}}{1+y_{a+d}^{-1}} = [P_{a+b+c}, P_{a+2b}, P_{a+b+d}, P_{a+2d}, P_{a+c+d}, P_{a+2c}] \\
&y_{c+d}^{(1)}\frac{1+y_{a+c}}{1+y_{b+c}^{-1}} = [P_{b+c+d}, P_{2b+d}, P_{a+b+d}, P_{2a+d}, P_{a+c+d}, P_{2c+d}] \\
&y_{c+d}^{(1)}\frac{1+y_{a+c}}{(1+y_{a+d}^{-1})(1+y_{b+c}^{-1})} = [P_{b+c+d}, P_{2b+d}, P_{a+b+d}, P_{a+2d}, P_{a+c+d}, P_{2c+d}]
\end{align*}
\end{prop}

\begin{proof}
The first equation is proven above, and we omit the remaining proofs which are similar.
\end{proof}

For variety, we list the two cases for which the answers are quadruple ratios.

\begin{prop}
\begin{align*}
&y_{c+d}^{(1)}(1+y_{a+c})(1+y_{b+d}) = -[P_{a+b+c}, P_{2b+c}, P_{b+c+d}, P_{b+2d}, P_{a+b+d}, P_{2a+d}, P_{a+c+d}, P_{a+2c}] \\
&y_{c+d}^{(1)}\frac{1}{(1+y_{a+d}^{-1})(1+y_{b+c}^{-1})} = -[P_{a+b+c}, P_{b+2c}, P_{b+c+d}, P_{2b+d}, P_{a+b+d}, P_{a+2d}, P_{a+c+d}, P_{2a+c}] 
\end{align*}
\end{prop}

The cases not appearing in Proposition \ref{propothery} do not correspond to any $y_{c+d}^{(k)}$ but they do have cluster significance.  Specifically, these other quantities appear in certain $Y$-seeds different from (but close in the exchange graph to) the ones reached by periodic mutation.  Understanding more distant $y$-variables is open.

There is a simple graphical representation that captures the relative positions of the points used in each multi-ratio.  Starting from some $[P_{v_1}, \ldots, P_{v_{2k}}]$, compute the differences $v_2-v_1, v_3-v_2, v_4-v_3, \ldots, v_1-v_{2m}$.  It turns out that each is either one of the six displacements from Table \ref{tabcompass}, or equals $b-a$ or $d-c$.  As such, each can be converted into a compass direction via the Table and the new rules that $b-a$ and $d-c$ correspond to Southwest and Southeast respectively.  The multi-ratio is illustrated by a cycle consisting of taking one step in each of these directions successively.  The steps $v_{2i}-v_{2i-1}$ are drawn solid and the $v_{2i+1}-v_{2i}$ dashed to differentiate which distances appear in the numerator versus the denominator.

\begin{ex}
Consider the triple ratio $[P_{a+b+c}, P_{a+2b}, P_{a+b+d}, P_{2a+d}, P_{a+c+d}, P_{a+2c}]$.  The differences in the subscripts are $b-c$, $d-b$, $a-b$, $c-a$, $c-d$, $b-c$.  The corresponding compass directions are South, East, Northeast, West, Northwest, South.  The representations for this triple ratio and several other quantities appear in Figure \ref{figcycles}.
\end{ex}

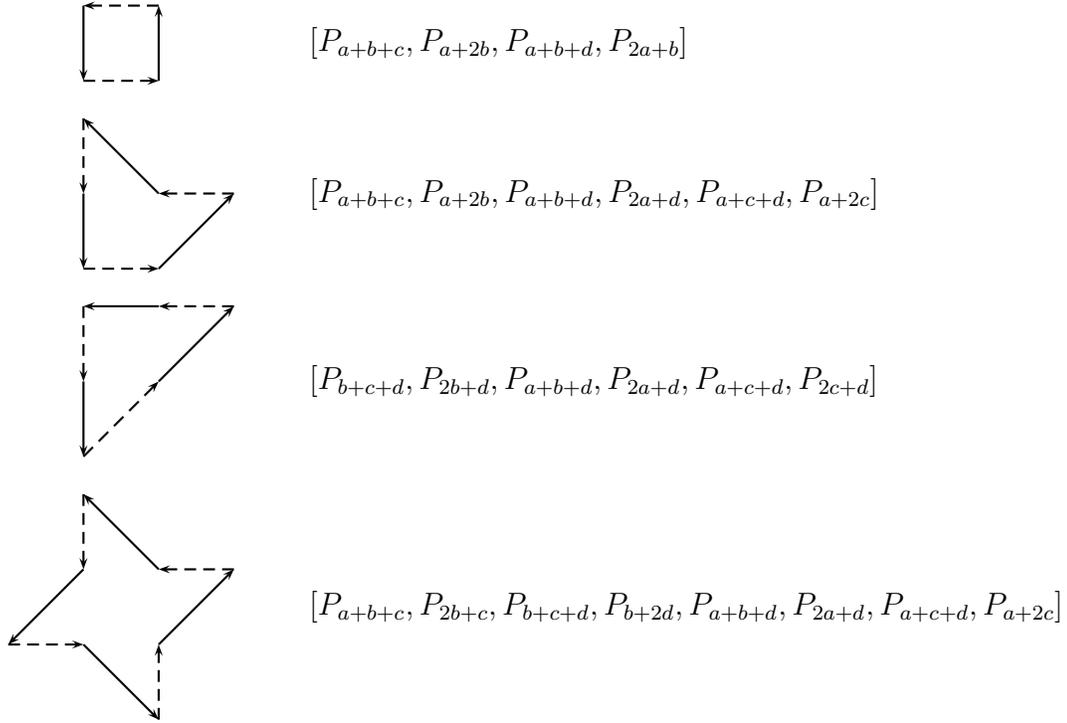
\begin{figure}
\begin{pspicture}(12,11)

\rput(-.5,8.5){
\psline{->}(1,2)(1,1)
\psline[linestyle=dashed]{->}(1,1)(2,1)
\psline{->}(2,1)(2,2)
\psline[linestyle=dashed]{->}(2,2)(1,2)
\rput[l](4,1.5){$[P_{a+b+c}, P_{a+2b}, P_{a+b+d}, P_{2a+b}]$}
}

\rput(-.5,6){
\psline{->}(1,2)(1,1)
\psline[linestyle=dashed]{->}(1,1)(2,1)
\psline{->}(2,1)(3,2)
\psline[linestyle=dashed]{->}(3,2)(2,2)
\psline{->}(2,2)(1,3)
\psline[linestyle=dashed]{->}(1,3)(1,2)
\rput[l](4,2){$[P_{a+b+c}, P_{a+2b}, P_{a+b+d}, P_{2a+d}, P_{a+c+d}, P_{a+2c}]$}
}

\rput(-.5,3.5){
\psline{->}(1,2)(1,1)
\psline[linestyle=dashed]{->}(1,1)(2,2)
\psline{->}(2,2)(3,3)
\psline[linestyle=dashed]{->}(3,3)(2,3)
\psline{->}(2,3)(1,3)
\psline[linestyle=dashed]{->}(1,3)(1,2)
\rput[l](4,2){$[P_{b+c+d}, P_{2b+d}, P_{a+b+d}, P_{2a+d}, P_{a+c+d}, P_{2c+d}]$}
}

\rput(-.5,0){
\psline{->}(1,3)(0,2)
\psline[linestyle=dashed]{->}(0,2)(1,2)
\psline{->}(1,2)(2,1)
\psline[linestyle=dashed]{->}(2,1)(2,2)
\psline{->}(2,2)(3,3)
\psline[linestyle=dashed]{->}(3,3)(2,3)
\psline{->}(2,3)(1,4)
\psline[linestyle=dashed]{->}(1,4)(1,3)
\rput[l](4,2.5){$[P_{a+b+c}, P_{2b+c}, P_{b+c+d}, P_{b+2d}, P_{a+b+d}, P_{2a+d}, P_{a+c+d}, P_{a+2c}]$}
}
\end{pspicture}
\caption{Cycles representing the formulas for some of the $y$-variables.}
\label{figcycles}
\end{figure}

\begin{rmk}
The $y_{i,j}$ with $0 \leq j < m$ can be taken as initial conditions for the recurrence \eqref{eqmain}.  Interpreting these variables as cross ratios, they can be thought of roughly as coordinates on the space of $Y$-meshes of type $S$.  However, from a cluster algebra point of view, it is more natural to consider a collection of variables that comprise a $Y$-seed, for instance
\begin{displaymath}
\{y_{i,j}^{(m-j)}: i,j \in \mathbb{Z}, 0 \leq j < m\}.
\end{displaymath}
Such coordinates possess nice properties, such as being log canonical with respect to a compatible Poisson bracket (see \cite{GSV}).  An upshot of the results of this section is they provide a concrete description of these coordinates.
\end{rmk}

We conclude this section by describing its results in the context of the lifted quiver $\tilde{Q}_S$ defined in Section \ref{seclift}.  Begin from $\tilde{Q}_S$ and perform some mutations, all at vertices of degree 4.  It is not hard to show that the local structure around each vertex $(i,j)$ will satisfy the following properties:
\begin{itemize}
\item $(i,j)$ is connected by an arrow to each of its four neighbors $(i \pm 1, j), (i, j \pm 1)$.
\item $(i,j)$ is not connected by an arrow to any other vertices, except perhaps some of the $(i \pm 1, j \pm 1)$.
\item The arrows incident to $(i,j)$ in cyclic order alternate between incoming and outgoing. 
\end{itemize}

There are sixteen possible local configurations around $(i,j)$ that are consistent with the above rules.  The sixteen types of $y$-variables described in this section are precisely those that arise via mutations in $\tilde{Q}_S$ at degree 4 vertices, and the type of variable at each vertex is determined by the local configuration of the quiver.  

We have observed, but can only prove case by case, that the cycles representing the various multi-ratios are related to these local configurations.  More precisely, given a $y$-variable at some vertex $(i,j) \in \mathbb{Z}^2$, it can be expressed as a multi-ratio which is represented by some cycle.  The solid arrows of the cycle end up matching the outgoing arrows from $(i,j)$ and the dashed arrows match the incoming arrows.  For example, one can obtain a variable $y^{(1)}_r(1+y_{r+a-d})$ by a single mutation in $\tilde{Q}_S$.  The variable will be at a vertex with outgoing arrows pointing South, Northeast, and Northwest, and incoming arrows pointing East, West, and South.  These six arrows can be placed head to tail in an interlacing fashion to obtain the corresponding cycle, the second from the top in Figure \ref{figcycles}. 

\medskip

\textbf{Acknowledgments.} We thank Boris Khesin, Gregg Musiker, and Bernd Sturmfels for comments and suggestions regarding a previous draft of this paper.


\begin{thebibliography}{99}
\bibitem{BS} A. Bobenko and Y. Suris, Discrete Differential Geometry. Integrable Structure, Graduate Studies in Mathematics. American Mathematical Society, Providence, RI (2008).
\bibitem{E} R. Eager, Brane tilings and non-commutative geometry, \textsl{J. High Energy Phys.} (2011), 026, 20pp.
\bibitem{FZ1} S. Fomin and A. Zelevinsky, Cluster algebras I: Foundations, \textsl{J. Amer. Math. Soc.} \textbf{15} (2002), 497--529.
\bibitem{FZ2} S. Fomin and A. Zelevinsky, Cluster algebras IV: Coefficients, \textsl{Compos. Math.} \textbf{143} (2007), 112--164.
\bibitem{FM} A. Fordy and R. Marsh, Cluster mutation-periodic quivers and associated Laurent sequences, \textsl{J. Algebraic Combin.} \textbf{34} (2011), 19--66.
\bibitem{FHMSVW} S. Franco, A. Hanany, D. Martelli, J. Sparks, D. Vegh, and B. Wecht, Gauge theories from toric geometry and brane tilings, \textsl{J. High Energy Phys.} (2006), 128, 44pp.
\bibitem{GSTV} M. Gekhtman, M. Shapiro, S. Tabachnikov, and A. Vainshtein, Higher pentagram maps, weighted directed
 networks, and cluster dynamics, \textsl{Electron. Res. Announc. Math. Sci.} \textbf{19} (2012), 1--17.
\bibitem{GSV} M. Gekhtman, M. Shapiro, and A. Vainshtein, Cluster algebras and Poisson geometry, \textsl{Mosc. Math. J.} \textbf{3} (2003), 899--934.
\bibitem{G} M. Glick, The pentagram map and Y-patterns, \textsl{Adv. Math.} \textbf{227} (2011), 1019--1045.
\bibitem{GK} A. Goncharov and R. Kenyon, Dimers and cluster integrable systems, \textsl{Ann. Sci. \'Ec. Norm. Sup\'er.} \textbf{46} (2013) 747--813.
\bibitem{JMZ} I. Jeong, G. Musiker, and S. Zhang, Gale-Robinson sequences and brane tilings, \textsl{25th Int. Conf. on Formal Power Series and Alg. Combinatorics} (2013), 707--718.
\bibitem{KS1}  B. Khesin and F. Soloviev, Integrability of higher pentagram maps, \textsl{Math. Ann.} \textbf{357} (2013), 1005--10047.
\bibitem{KS2} B. Khesin and F. Soloviev, The geometry of dented pentagram maps, \textsl{J. Eur. Math. Soc.} \textbf{18} (2016), 147--179.
\bibitem{KS3} B. Khesin and F. Soloviev, Non-integrability vs. integrability in pentagram maps, \textsl{J. Geom. Phys.} \textbf{87} (2015), 275--285.
\bibitem{M1} G. Mari Beffa, On generalizations of the pentagram map: discretizations of AGD flows, \textsl{J. Nonlinear Sci.} \textbf{23} (2013), 303--334.
\bibitem{M2} G. Mari Beffa, On integrable generalizations of the pentagram map, \textsl{Int. Math. Res. Not. IMRN} (2015), 3669--3693.
\bibitem{M3} G. Mari Beffa, On the integrability of the shift map on twisted pentagram spirals, \textsl{J. Phys. A} \textbf{48} (2015), 285202, 25 pp.
\bibitem{OST1} V. Ovsienko, R. Schwartz, and S. Tabachnikov, The pentagram map: a discrete integrable system, \textsl{Comm. Math. Phys.} \textbf{299} (2010), 409-446.
\bibitem{OST2} V. Ovsienko, R. Schwartz, and S. Tabachnikov, Liouville-Arnold integrability of the pentagram map on closed polygons, \textsl{Duke Math. J.} \textbf{162} (2013), 2149--2196.
\bibitem{S1} R. Schwartz, The pentagram map, \textsl{Experiment. Math.} \textbf{1} (1992), 71--81.
\bibitem{S2} R. Schwartz, Discrete monodromy, pentagrams, and the method of condensation, \textsl{J. Fixed Point Theory Appl.} \textbf{3} (2008), 379--409.
\bibitem{S3} R. Schwartz, Pentagram spirals, \textsl{Exp. Math.} \textbf{22} (2013), 384--405.
\bibitem{So} F. Soloviev, Integrability of the Pentagram Map, \textsl{Duke Math. J.} \textbf{162} (2013), 2815--2853.
\end{thebibliography}
\end{document}